\documentclass[11pt]{amsart}
\usepackage{amsmath,amssymb,amsthm}
\usepackage{tikz}
\usepackage[letterpaper, margin=1in]{geometry}
\usepackage{todonotes}
\usetikzlibrary{decorations.markings,intersections,calc}
\usetikzlibrary{shapes.misc}
\usepackage{verbatim}
\usepackage{enumitem}
\usepackage{graphicx}
\usepackage{subcaption}

\usepackage[backref]{hyperref}
\hypersetup{
    colorlinks,
    linkcolor={red!60!black},
    citecolor={green!60!black},
    urlcolor={blue!60!black},
    pdftitle={A unified Erdős-Pósa theorem for cycles in graphs labelled by multiple abelian groups},
    pdfauthor={J. Pascal Gollin, Kevin Hendrey, O-joung Kwon,  Sang-il Oum, and Youngho Yoo}}

\usepackage{thmtools, thm-restate}

\newcommand\abs[1]{\ensuremath{\lvert #1\rvert}}

\newtheorem{theorem}{Theorem}[section]
\newtheorem{lemma}[theorem]{Lemma}
\newtheorem{corollary}[theorem]{Corollary}
\newtheorem{proposition}[theorem]{Proposition}
\newtheorem{conjecture}[theorem]{Conjecture}
\newtheorem{observation}[theorem]{Observation}
\newtheorem{claim}{Claim}
\newtheorem*{claim*}{Claim}

\newtheorem{problem}{Problem}

\newenvironment{proofofclaim}[1][Proof.]{%
    \begin{proof}[{#1}]%
        }{%
    \end{proof}}

\usetikzlibrary{decorations.pathmorphing}
\usetikzlibrary{decorations.markings}

\theoremstyle{definition}
\newtheorem{definition}[theorem]{Definition}

\newcommand\cT{\mathcal{T}}
\newcommand{\gen}[1]{\ensuremath{\langle #1\rangle}}
\newcommand{\branch}{\ensuremath{V_{\neq 2}}}
\newcommand{\maxk}{\ensuremath{k^{\star}}}

\def\lowfwd #1#2#3{{\mathop{\kern0pt #1}\limits^{\kern#2pt\raise.#3ex
\vbox to 0pt{\hbox{$\scriptscriptstyle\rightarrow$}\vss}}}}
\def\lowbkwd #1#2#3{{\mathop{\kern0pt #1}\limits^{\kern#2pt\raise.#3ex
\vbox to 0pt{\hbox{$\scriptscriptstyle\leftarrow$}\vss}}}}

\DeclareMathOperator{\arc}{arc}

\usepackage{ifthen}
\usepackage{listofitems}

        \newcommand{\tikzwall}[6]{%
            \begin{scope}
                \pgfmathtruncatemacro{\c}{ #1 * 2 - 1 }
                \pgfmathtruncatemacro{\r}{ #2 - 1}
                \pgfmathtruncatemacro{\basex}{#3}
                \pgfmathtruncatemacro{\basey}{#4}
                \pgfmathtruncatemacro{\bm}{Mod(\basex+\basey,2)}
                \pgfmathtruncatemacro{\bmi}{Mod(\basex+\basey+1,2)}
                
                \ifthenelse{\c > 1 \and \r > 0}{
                    \foreach \x in {0, ..., \c}{
                        \ifthenelse{0 < \x \and \x < \c}{
                            \pgfmathtruncatemacro{\start}{0}
                            \pgfmathtruncatemacro{\t}{\r}
                        }{
                            \ifthenelse{\x = 0}{
                                \pgfmathtruncatemacro{\start}{\bmi}
                                \pgfmathtruncatemacro{\t}{\r - Mod(\r+\bm,2)}
                            }{
                                \pgfmathtruncatemacro{\start}{\bm}
                                \pgfmathtruncatemacro{\t}{\r - Mod(\r+\bmi,2)}
                            }
                        }
                        \foreach \y in {\start,...,\t} {
                            \ifthenelse{\start < \t}{
                                \node at (\basex+\x,\basey+\y) [#5] (v\x-\y) {};
                            }{}
                        }
                    }
                    \foreach \x in {0, ..., \c}{
                        \foreach \y in {0,...,\r} {
                            \pgfmathtruncatemacro{\nextx}{\x+1}
                            \pgfmathtruncatemacro{\nexty}{\y+1}
                            \pgfmathtruncatemacro{\sum}{\x+\y+\bm}
                            \pgfmathtruncatemacro{\p}{Mod(\r+\bmi,2)}
                            \pgfmathtruncatemacro{\q}{Mod(\r+\bm,2)}
                            
                            \ifthenelse{\isodd{\sum} \and \y < \r}{
                                \draw [#6] (v\x-\y) edge (v\x-\nexty);
                            }{}
                            
                            \ifthenelse{\x < \c}{
                                \ifthenelse{0 < \y \and \y < \r}{
                                    \draw [#6] (v\x-\y) edge (v\nextx-\y);
                                }{
                                    \ifthenelse{\y = 0 \and \the\numexpr \bmi - 1 \relax < \x \and \x < \the\numexpr \c - \bm \relax}{
                                        \draw [#6] (v\x-\y) edge (v\nextx-\y);
                                    }{}
                                    \ifthenelse{\y = \r \and \the\numexpr \q - 1 \relax < \x \and \x < \the\numexpr \c - \p \relax}{
                                        \draw [#6] (v\x-\y) edge (v\nextx-\y);
                                    }{}
                                    
                                }
                            }{
                            }
                        }
                    }
                }{}
            \end{scope}
        }
        
        \newcommand{\subdividewall}[3]{%
            \begin{scope}
                \pgfmathtruncatemacro{\c}{ #1 * 2 - 1 }
                \pgfmathtruncatemacro{\r}{ #2 - 1}
                
                \foreach \x in {0, ..., \the\numexpr \c \relax} {
                    \foreach \y in {0, ..., \the\numexpr \r \relax} {
                        \pgfmathtruncatemacro{\z}{\x+\y}
                        \ifthenelse{\z = 0}{}{
                            \ifthenelse{\isodd{\r}}{
                                \ifthenelse{\x = 0 \and \y = \r}{}{
                                    \ifthenelse{\x = \c}{}{
                                        \node [#3] at (\x+0.5,\y) {};
                                    }
                                    \ifthenelse{\isodd{\z}}{
                                        \ifthenelse{\y = \r}{}{
                                            \node [#3] at (\x,\y+0.5) {};
                                        }
                                    }{}
                                }
                            }{
                                \ifthenelse{\x = \the\numexpr \c-1 \relax \and \y = \r}{}{
                                    \ifthenelse{\x = \c}{}{
                                        \node [#3] at (\x+0.5,\y) {};
                                    }
                                    \ifthenelse{\isodd{\z}}{
                                        \ifthenelse{\y = \r}{}{
                                            \node [#3] at (\x,\y+0.5) {};
                                        }
                                    }{}
                                }
                            }
                        }
                    }
                }
            \end{scope}
        }
        
        \newcommand{\eschercycle}[3]{%
            \begin{scope}
                \pgfmathtruncatemacro{\h}{#1}
                \pgfmathtruncatemacro{\x}{#2}
                \renewcommand{\c}[0]{#3}
                \pgfmathtruncatemacro{\y}{\h-\x-1}
                \pgfmathtruncatemacro{\p}{0}
                
                \pgfmathtruncatemacro{\z}{Mod(\h-1,2)}
                \draw [out=20, \c, dashed] (2*\x+1,0) .. controls (3*\h,-\h+1) and (3*\h,2*\h-2) .. (2*\y+\z,\h-1);
                
                \ifthenelse{\x < \the\numexpr \h/2 \relax}{
                    \pgfmathtruncatemacro{\p}{0}
                }{
                    \pgfmathtruncatemacro{\p}{3}
                }
                
                \foreach \i in {0,...,\the\numexpr \x-1 \relax}{
                    \ifthenelse{\x > 0}{
                        \ifthenelse{\isodd{\i}}{
                            \draw [\c, thick, dashed, dash phase = \p pt] (2*\x,\i) edge (2*\x,\i+1) {};
                        }{
                            \draw [\c, thick, dashed, dash phase = \p pt] (2*\x+1,\i) edge (2*\x+1,\i+1) {};
                        }
                        \ifthenelse{\i = \the\numexpr \x - 1 \relax}{
                        }{
                            \draw [\c, thick, dashed, dash phase = \p pt] (2*\x,\i+1) edge (2*\x+1,\i+1) {};
                        }
                    }{}
                }
                \foreach \i in {\x,...,\the\numexpr \h-2 \relax}{
                    \ifthenelse{\x < \the\numexpr \h-1 \relax}{
                        \ifthenelse{\isodd{\i}}{
                            \draw [\c, thick, dashed, dash phase = \p pt] (2*\y,\i) edge (2*\y,\i+1) {};
                        }{
                            \draw [\c, thick, dashed, dash phase = \p pt] (2*\y+1,\i) edge (2*\y+1,\i+1) {};
                        }
                        \ifthenelse{\i = \x}{
                        }{
                            \draw [\c, thick, dashed, dash phase = \p pt] (2*\y,\i) edge (2*\y+1,\i) {};
                        }
                    }{}
                }
                
                \ifthenelse{\isodd{\h}}{
                    \ifthenelse{\x = 0}{
                        \pgfmathtruncatemacro{\a}{1}
                        \pgfmathtruncatemacro{\b}{2*\h-2}
                    }{
                        \ifthenelse{\x = \the\numexpr \h-1 \relax}{
                            \pgfmathtruncatemacro{\a}{0}
                            \pgfmathtruncatemacro{\b}{2*\h-3}
                        }{
                            \ifthenelse{\the\numexpr 2*\x \relax = \the\numexpr \h-1 \relax}{
                                \pgfmathtruncatemacro{\a}{\h-1}
                                \pgfmathtruncatemacro{\b}{\h-1}
                            }{
                                \ifthenelse{\x < \the\numexpr \h/2 \relax}{
                                    \ifthenelse{\isodd{\x}}{
                                        \pgfmathtruncatemacro{\a}{2*\x+1}
                                        \pgfmathtruncatemacro{\b}{2*\y-1}
                                    }{
                                        \pgfmathtruncatemacro{\a}{2*\x}
                                        \pgfmathtruncatemacro{\b}{2*\y}
                                    }
                                }{
                                    \ifthenelse{\isodd{\x}}{
                                        \pgfmathtruncatemacro{\a}{2*\y}
                                        \pgfmathtruncatemacro{\b}{2*\x}
                                    }{
                                        \pgfmathtruncatemacro{\a}{2*\y+1}
                                        \pgfmathtruncatemacro{\b}{2*\x-1}
                                    }
                                }
                            }
                        }
                    }
                }{
                    \ifthenelse{\x = 0 \or \x = \the\numexpr \h-1 \relax}{
                        \pgfmathtruncatemacro{\a}{1}
                        \pgfmathtruncatemacro{\b}{2*\h-2}
                    }{
                        \ifthenelse{\x < \the\numexpr \h/2 \relax}{
                            \ifthenelse{\isodd{\x}}{
                                \pgfmathtruncatemacro{\a}{2*\x+1}
                                \pgfmathtruncatemacro{\b}{2*\y-1}
                            }{
                                \pgfmathtruncatemacro{\a}{2*\x}
                                \pgfmathtruncatemacro{\b}{2*\y}
                            }
                        }{
                            \ifthenelse{\isodd{\x}}{
                                \pgfmathtruncatemacro{\a}{2*\y}
                                \pgfmathtruncatemacro{\b}{2*\x}
                            }{
                                \pgfmathtruncatemacro{\a}{2*\y+1}
                                \pgfmathtruncatemacro{\b}{2*\x-1}
                            }
                        }
                    }
                }
                
                \foreach \i in {\a,...,\b} {
                    \ifthenelse{\x < \the\numexpr \h/2 \relax \and \i < \the\numexpr \h \relax}{
                        \pgfmathtruncatemacro{\q}{3-\p}
                    }{
                        \ifthenelse{\i < \the\numexpr \h-1 \relax}{
                            \pgfmathtruncatemacro{\q}{3-\p}
                        }{
                            \pgfmathtruncatemacro{\q}{\p}
                        }
                    }
                    \draw [\c, thick, dashed, dash phase = \q pt] (\i,\x) edge (\i+1,\x) {};
                }
            \end{scope}
        }

        \newcommand{\serialcycle}[3]{%
            \begin{scope}
                \pgfmathtruncatemacro{\k}{#1}
                \pgfmathtruncatemacro{\h}{2*\k-1}
                \pgfmathtruncatemacro{\x}{#2}
                \renewcommand{\c}[0]{#3}
                \pgfmathtruncatemacro{\p}{0}
                
                \draw [out=20, \c, thick, dashed] (0,2*\x-2) .. controls (-2,2*\x-1.5) .. (0,2*\x-1);
                \draw [out=20, \c, thick, dashed] (2*\h+1,2*\x-2) .. controls (2*\h+3,2*\x-1.5) .. (2*\h+1,2*\x-1);
                \draw [out=20, \c, thick, dashed] (4*\x-3,0) .. controls (4*\x-2,-2) .. (4*\x-1,0);
                
                \ifthenelse{\x = 1}{
                    \draw [\c, thick, dashed] (0,0) -- (1,0) {};
                }{}
                \ifthenelse{\x = \k}{
                    \draw [\c, thick, dashed] (0,2*\k-1) -- (1,2*\k-1) {};
                }{}
                
                \foreach \i in {0, ..., \the\numexpr 2*\h \relax} {
                    \ifthenelse{\i < 3 \and \x = 1}{}{
                        \ifthenelse{\i = \the\numexpr 4*\x-4 \relax}{}{
                            \ifthenelse{\i = \the\numexpr 4*\x-3 \relax}{}{
                                \draw [\c, thick, dashed] (\i,2*\x-2) -- (\i+1,2*\x-2) {};
                            }
                         }
                    }
                    \ifthenelse{\i = 0 \and \x = \k}{}{
                        \draw [\c, thick, dashed] (\i,2*\x-1) -- (\i+1,2*\x-1) {};
                    }
                }
                
                \foreach \i in {0, ..., \the\numexpr 2*\x-3 \relax} {
                    \pgfmathtruncatemacro{\p}{Mod(\i,2)}
                    \ifthenelse{\x = 1}{}{
                        \draw [\c, thick, dashed, dash phase = 3pt] (4*\x-3-\p,\i) -- (4*\x-3-\p,\i+1) {};
                        \draw [\c, thick, dashed, dash phase = 3pt] (4*\x-1-\p,\i) -- (4*\x-1-\p,\i+1) {};
                        \ifthenelse{\i > 0}{
                            \draw [\c, thick, dashed, dash phase = 3pt] (4*\x-4,\i) -- (4*\x-3,\i) {};
                            \draw [\c, thick, dashed, dash phase = 3pt] (4*\x-2,\i) -- (4*\x-1,\i) {};
                        }{}
                    }
                }
                
            \end{scope}
        }

        \newcommand{\mixedcycle}[3]{%
            \begin{scope}
                \pgfmathtruncatemacro{\k}{#1}
                \pgfmathtruncatemacro{\h}{2*\k}
                \pgfmathtruncatemacro{\x}{#2}
                \renewcommand{\c}[0]{#3}
                \pgfmathtruncatemacro{\p}{0}
                
                \draw [out=20, \c, thick, dashed] (0,\x-1) .. controls (-1-\k+\x,\h/2 - 0.5) .. (0,\h-\x);
                \draw [out=20, \c, thick, dashed] (4*\x-3,0) .. controls (4*\x-2,-2) .. (4*\x-1,0);
                
                \ifthenelse{\x = 1}{
                    \draw [\c, thick, dashed] (0,0) -- (1,0) {};
                    \draw [\c, thick, dashed] (0,\h-1) -- (1,\h-1) {};
                }{}
                
                \foreach \i in {0, ..., \the\numexpr \h-\x-1 \relax} {
                    \pgfmathtruncatemacro{\p}{Mod(\i,2)}
                    \draw [\c, thick, dashed, dash phase = 3pt] (4*\x-1-\p,\i) -- (4*\x-1-\p,\i+1) {};
                        \ifthenelse{\i < \the\numexpr \h-\x-1 \relax}{
                            \draw [\c, thick, dashed, dash phase = 3pt] (4*\x-2,\i+1) -- (4*\x-1,\i+1) {};
                        }{}
                }
                
                \ifthenelse{\x > 1}{
                    \foreach \i in {0, ..., \the\numexpr \x-2 \relax} {
                        \pgfmathtruncatemacro{\p}{Mod(\i,2)}
                        \draw [\c, thick, dashed, dash phase = 3pt] (4*\x-3-\p,\i) -- (4*\x-3-\p,\i+1) {};
                        \ifthenelse{\i < \the\numexpr \x-2 \relax}{
                            \draw [\c, thick, dashed, dash phase = 3pt] (4*\x-4,\i+1) -- (4*\x-3,\i+1) {};
                        }{}
                    }
                }{}
                
                \pgfmathtruncatemacro{\q}{Mod(\x+1,2)}
                \foreach \i in {0, ..., \the\numexpr 4*\x \relax} {
                    \ifthenelse{\x = 1 \and \i = 0}{}{
                        \ifthenelse{\i < \the\numexpr 4*\x-4+\q \relax}{
                            \draw [\c, thick, dashed] (\i,\x-1) -- (\i+1,\x-1) {};
                        }{}
                    }
                    \ifthenelse{\x = 1 \and \i = \the\numexpr \h-\x \relax}{}{
                        \ifthenelse{\i < \the\numexpr 4*\x-1-\q \relax}{
                            \draw [\c, thick, dashed] (\i,\h-\x) -- (\i+1,\h-\x) {};
                        }{}
                    }
                }
            \end{scope}
        }

\begin{document}
\title[Erd\H{o}s-P\'{o}sa for cycles in graphs labelled by multiple abelian groups]{A unified Erd\H{o}s-P\'{o}sa theorem for cycles in graphs labelled by multiple abelian groups}
\author{J.~Pascal Gollin} 
\author{Kevin Hendrey}

\address[Gollin]{\small FAMNIT, University of Primorska, Koper, Slovenia.}
\address[Hendrey]{\small School of Mathematics, Monash University, Melbourne, Australia}
\address[Kwon, Oum]{\small Discrete Mathematics Group, Institute~for~Basic~Science~(IBS), Daejeon,~South~Korea.}

\author{O-joung Kwon}

\address[Kwon]{\small Department of Mathematics, Hanyang~University, Seoul,~South~Korea.}

\author{Sang-il~Oum}

\author{Youngho Yoo}
\address[Yoo]{\small Department of Mathematics, Texas A\&M University, College Station, TX, USA}

\email{pascal.gollin@famnit.upr.si}
\email{kevin.hendrey1@monash.edu}
\email{ojoungkwon@hanyang.ac.kr}
\email{sangil@ibs.re.kr}
\email{yyoo.math@gmail.com}
\thanks{All authors except the last are supported by the Institute for Basic Science (IBS-R029-C1). The first author is supported by the Institute for Basic Science (IBS-R029-Y3) and in part by the Slovenian Research and Innovation Agency (research project N1-0370). 
The second author is also supported by the Australian Research Council.
The third author is also supported by the National Research Foundation of Korea (NRF) grant funded by the Ministry of Science and ICT (No. NRF-2021K2A9A2A11101617 and RS-2023-00211670). The last author is partially supported by NSERC PGSD2-532637. }
\date{}

\begin{abstract}
    In 1965, Erd\H{o}s and P\'{o}sa proved that there is an (approximate) duality between the maximum size of a packing of cycles and the minimum size of a vertex set hitting all cycles.
    Such a duality does not hold for odd cycles, and Dejter and Neumann-Lara asked in 1988 to find all pairs~${(\ell, z)}$ of integers where such a duality holds for the family of cycles of length~$\ell$ modulo~$z$. 
    We characterise all such pairs, and we further generalise this characterisation to cycles in graphs labelled with a bounded number of abelian groups, whose values avoid a bounded number of elements of each group. 
    This unifies almost all known types of cycles that admit such a duality, and it also provides new results. 
    Moreover, we characterise the obstructions to such a duality in this setting, and thereby obtain an analogous characterisation for cycles in graphs embeddable on a fixed compact orientable surface. 
\end{abstract}

\keywords{Erd\H{o}s-P\'osa property, group-labelled graph}

\subjclass[2020]{05C38, 05C70, 05C78, 05C25}

\maketitle

\section{Introduction}
\label{sec:intro}

Erd\H{o}s and P\'{o}sa~\cite{ErdosP1965} proved in 1965 that every graph contains either~$k$ vertex-disjoint cycles or a set of~${\mathcal{O}(k\log k)}$ vertices that hits all cycles of the graph. 
This breakthrough result sparked extensive research on finding hitting-packing dualities for various graph classes. 
In particular, cycles with modularity constraints have been considered. 
For example, Thomassen~\cite{Thomassen1988} showed that for every positive integer~$z$, an analogue of the Erd\H{o}s-P\'{o}sa theorem holds for cycles of length~$0$ modulo~$z$, and Thomas and Yoo~\cite{YooR2020} proved that for every integer~${\ell}$ and every odd prime power~${z}$, an analogue of the Erd\H{o}s-P\'{o}sa theorem holds for cycles of length~${\ell}$ modulo~${z}$. 
However, this property does not hold for all pairs of integers~$\ell$ and~$z$. 
Lov\'asz and Schrijver (see~\cite{Thomassen1988}) found a class of graphs, so called \emph{Escher walls} (see Figure~\ref{fig:obstructions-intro}(a)), which demonstrate that such a duality does not exist for odd cycles. 
Reed~\cite{Reed1999} showed that large Escher walls are contained in every graph that contains neither many vertex-disjoint odd cycles nor a small set of vertices hitting all odd cycles, yielding a structural characterisation of graphs failing to satisfy this type of hitting-packing duality for odd cycles. 
Using this structural characterisation, Reed concluded that a \emph{half-integral} version of the Erd\H{o}s-P\'{o}sa theorem holds: every graph contains either a set of~$k$ odd cycles such that each vertex of the graph is contained in at most two of the cycles, or a set of at most~$f(k)$ vertices that hits all odd cycles of the graph. 
In the precursor to this paper, Gollin, Hendrey, Kawarabayashi, Kwon, and Oum~\cite{GollinHKKO2021} generalised this conclusion and established a unified framework for generating half-integral Erd\H{o}s-P\'{o}sa results, as we will discuss. 

Escher walls can be modified to give infinitely many pairs~${(\ell, z)}$ for which an analogue of the Erd\H{o}s-P\'{o}sa theorem does not hold for cycles of length~$\ell$ modulo~$z$.
This was essentially shown by Dejter and Neumann-Lara~\cite{DejterN1988}, who then asked to find all pairs~${(\ell, z)}$ of integers for which an analogue of the Erd\H{o}s-P\'{o}sa theorem does hold for cycles of length~$\ell$ modulo~$z$. 
Note that a half-integral analogue of the Erd\H{o}s-P\'{o}sa theorem holds for all pairs~${(\ell,z)}$ (see~\cite{GollinHKKO2021}). 
In this paper, we answer the question of Dejter and Neumann-Lara completely as a corollary of our main result. 
For an integer~$m$, let~${[m]}$ denote the set of positive integers~${z}$ with~${z \leq m}$. 

\begin{theorem}
    \label{cor:mainmod}
    Let~$\ell$ and~$z$ be integers with~${z \geq 2}$, and let~${p_1^{a_1} \cdots p_n^{a_n}}$ be the prime factorisation of~$z$ with~${p_{i} < p_{i+1}}$ for all~${i \in [n-1]}$.
    The following statements are equivalent. 
    \begin{itemize}
        \item There is a function~${f \colon \mathbb{N} \to \mathbb{N}}$ such that 
            for every positive integer~$k$, every graph contains~$k$ vertex-disjoint cycles of length~$\ell$ modulo~$z$ or a set of at most~${f(k)}$ vertices hitting all such cycles. 
    \item Both of the following conditions are satisfied. 
        \begin{enumerate}
            \item If~${p_1 = 2}$, then~${\ell \equiv 0 \pmod {p_1^{a_1}}}$. 
            \item There do not exist distinct~${i_1, i_2, i_3 \in [n]}$ such that~${\ell \not\equiv 0 \pmod {p_{i_j}^{a_{i_j}}}}$ for each~${j \in [3]}$.
        \end{enumerate}
    \end{itemize}
\end{theorem}

Other types of constraints for cycles have been considered. 
Given a set~$S$, an \emph{$S$-cycle} is a cycle containing a vertex in~$S$. 
Kakimura, Kawarabayashi, and Marx~\cite{KakimuraKM2011} showed that an analogue of the Erd\H{o}s-P\'{o}sa theorem holds for $S$-cycles. 
Bruhn, Joos, and Schaudt~\cite{BruhnJS2017} extended this result to $S$-cycles of length at least some fixed integer~$L$. 
Note that there cannot be an extension of their result to odd $S$-cycles due to Escher walls, but there are other types of obstructions, see Figure~\ref{fig:obstructions-intro}(b). 
Kakimura and Kawarabayashi~\cite{KakimuraK2013} showed a half-integral analogue of the Erd\H{o}s-P\'{o}sa theorem for odd $S$-cycles. 

Given a family~$\mathcal{S}$ of sets, an \emph{$\mathcal{S}$-cycle} is a cycle containing at least one vertex from each member of~$\mathcal{S}$. 
Huynh, Joos, and Wollan~\cite{HuynhJW2017} proved an analogue of the Erd\H{o}s-P\'{o}sa theorem for ${(S_1,S_2)}$-cycles. 
An extension of their result to ${(S_1,S_2,S_3)}$-cycles fails, and 
a third type of obstruction appears in this setting, see Figure~\ref{fig:obstructions-intro}(c). 
A half-integral analogue of the Erd\H{o}s-P\'{o}sa theorem holds for all finite families~$\mathcal{S}$ (see~\cite{GollinHKKO2021}). 

\begin{figure}[htbp]
    \centering

    \begin{subfigure}{0.45\linewidth}
        \centering
        \begin{tikzpicture}
            [scale=0.4]
                \tikzset{vx/.style = {circle, draw, fill=black!0, inner sep=0pt, minimum width=4pt}}
                \tikzset{vxsub/.style = {circle, draw, fill=black!50, inner sep=0pt, minimum width=2pt}}
                
                \pgfmathtruncatemacro{\h}{6}

                \pgfmathtruncatemacro{\height}{\h+1}

                \tikzwall{\height}{\height}{0}{0}{vx}{gray}
                \foreach \x in {0,...,\the\numexpr \h-1 \relax} {
                    \pgfmathtruncatemacro{\y}{\h-\x-1}
                    \pgfmathtruncatemacro{\source}{\x+1}
                    \pgfmathtruncatemacro{\targetx}{2*\height-1}
                    \pgfmathtruncatemacro{\targety}{\y}
                    \draw [red,dashed,rounded corners=5pt] %
                    (v0-\source) 
                    to (-0.5*\x-1,\x+1) 
                    to (-0.5*\x-1,\h+.2*\x+.5)
                    to (2*\h+0.5*\x+2,\h+.2*\x+.5)
                    to (2*\h+0.5*\x+2,\y)
                    to (v\targetx-\targety);

                }
        \end{tikzpicture}
    \caption{An \emph{Escher wall}, the obstruction for odd cycles. We refer to the arrangement of the red dashed paths around the wall as `crossing'}
    \end{subfigure}
    \hfill
    \begin{subfigure}{0.48\linewidth}
        \centering
        \begin{tikzpicture}[scale=0.4]
            \tikzset{vx/.style = {circle, draw, fill=black!0, inner sep=0pt, minimum width=4pt}}
            \tikzset{vxsub/.style = {circle, draw, fill=black!50, inner sep=0pt, minimum width=2pt}}
            \tikzset{vx1/.style = {circle, draw, fill=red, inner sep=0pt, minimum width=4pt}}
            
            \pgfmathtruncatemacro{\kay}{4}
            \pgfmathtruncatemacro{\height}{2 * \kay}
            \tikzwall{\height}{\height}{0}{0}{vx}{gray}
            \pgfmathtruncatemacro{\targetx}{2*\height-1}
            \pgfmathtruncatemacro{\targetxm}{2*\height-2}

            \foreach \i in {1,...,\kay} {
                \pgfmathtruncatemacro{\s}{\i-1}
                \pgfmathtruncatemacro{\t}{\height-\i}
                
                \ifthenelse{\i=1}
                {
                    \draw [blue, thick, dashed,rounded corners=5pt,out=180,in=180] (v1-\s) -- ++(-3-0.2*\kay+0.5*\i,0) -- ++(0,\height-1) -- (v1-\t);
                } 
                {
                    \draw [blue, thick, dashed,rounded corners=5pt,out=180,in=180] (v0-\s) -- ++(-2-0.2*\kay+0.5*\i,0) -- ++(0,\t-\i+1) -- (v0-\t);
                }

                \node at (2*\height,2*\i-1.5) (a\i) [vx1] {};
                \pgfmathtruncatemacro{\tone}{2*\i-2}
                \pgfmathtruncatemacro{\ttwo}{2*\i-1}
                \draw [red,thick] (v\targetx-\tone)  to (a\i) to (v\targetx-\ttwo);

            }

        \end{tikzpicture}
        \caption{An obstruction for odd $S$-cycles, where vertices in $S$ are shown in red. We refer to the arrangement of the blue dashed paths around the wall as `nested', and of the red solid paths as `in series'.}
    \end{subfigure}
    \hfill
        \begin{subfigure}{\linewidth}
            \centering
            \begin{tikzpicture}[scale=0.3]
                \tikzset{vx/.style = {circle, draw, fill=black!0, inner sep=0pt, minimum width=4pt}}
                \tikzset{vxsub/.style = {circle, draw, fill=black!50, inner sep=0pt, minimum width=2pt}}
                \tikzset{vx1/.style = {circle, draw, fill=red!30, inner sep=0pt, minimum width=4pt}}
                \tikzset{vx2/.style = {circle, draw, fill=blue!30, inner sep=0pt, minimum width=4pt}}
                \tikzset{vx3/.style = {circle, fill=green!70!blue, draw,
                inner sep=0pt, minimum width=4pt}}
                
                \pgfmathtruncatemacro{\kay}{6}
                \pgfmathtruncatemacro{\height}{2 * \kay}
                \pgfmathtruncatemacro{\kaym}{\kay-1}
                \tikzwall{\height}{\height}{0}{0}{vx}{gray}
                \pgfmathtruncatemacro{\targetx}{2*\height-1}
                \foreach \i in {2,...,\kaym} {
                    \pgfmathtruncatemacro{\di}{2 * \i}
                    \pgfmathtruncatemacro{\yone}{2*\i-2}
                    \pgfmathtruncatemacro{\ytwo}{2*\i-1}

                    \ifthenelse{\di>\kay}{
                        \node [vx2] at (-1.5,2*\i-1.5) (b\i) {\tiny 2};
                        \draw [blue, thick, dotted] (v0-\yone) [out=180,in=-60] to (b\i) [out=60,in=180] to (v0-\ytwo);
                    }{
                        \node [vx1] at (-1.5,2*\i-1.5) (b\i) {\tiny 1};
                        \draw [red, thick, dotted] (v0-\yone) [out=180,in=-60] to (b\i)[out=60,in=180] to (v0-\ytwo);
                    }
                    \node at (2*\height+0.5,2*\i-1.5) (a\i) [vx3] {\tiny 3};
                    \draw [green!70!blue, thick, dotted] (v\targetx-\yone)  [out=0,in=-120] to (a\i)[out=120,in=0] to (v\targetx-\ytwo);
                }
                \pgfmathtruncatemacro{\yone}{2*\kay-2}
                \pgfmathtruncatemacro{\ytwo}{2*\kay-1}
                \node [vx2] at (-1.5,2*\kay-1.5) (b) {\tiny 2};
                \draw [blue, thick, dotted] (v0-\yone)[out=180,in=-60]  to(b) [out=60,in=180]to (v1-\ytwo);
                \node [vx1] at (-1.5,2*1-1.5) (c) {\tiny 1};

                \draw [red, thick, dotted] (v1-0) to[out=180,in=-60]  (c)[out=60,in=180] to (v0-1);

            \end{tikzpicture}
        \caption{An obstruction for $(S_1,S_2,S_3)$-cycles, where vertices in $S_1$, $S_2$, and $S_3$ are shown in red, blue, and green colours respectively.}
    \end{subfigure}

    \begin{subfigure}{\linewidth}
        \centering
        \begin{tikzpicture}[scale=0.3]
            \tikzset{vx/.style = {circle, draw, fill=black!0, inner sep=0pt, minimum width=4pt}}
            \tikzset{vxsub/.style = {circle, draw, fill=black!50, inner sep=0pt, minimum width=2pt}}
            
            \tikzset{c1/.style={cyan, line width=6pt,opacity=0.8,rounded corners}}
            \tikzset{c2/.style={purple, line width=6pt,opacity=0.8,rounded corners}}
            \tikzset{vx1/.style = {circle, draw, solid, black, fill=red!30, inner sep=0pt, minimum width=4pt}}
            \tikzset{vx2/.style = {circle, draw, solid, black, fill=blue!30, inner sep=0pt, minimum width=4pt}}
            \tikzset{vx3/.style = {circle, black, solid, fill=orange, draw,
            inner sep=0pt, minimum width=4pt}}

            \draw [c1] (0,11) to (25,11);

            \draw [c1] (0,2) to (13,2)
            to ++(0,1) to ++(1,0) to ++(0,1) to ++(1,0) to ++(0,1) to ++(1,0) to ++(0,1) to (25,6);

            \draw [c2] (0,9) to ++(3,0) to ++(0,-1) to ++(1,0) to ++(0,-1) to ++(-1,0) to ++(0,-1) to (0,6);

            \draw [c2] (0,5) to (10,5)
            to ++(0,1) to ++(1,0) to ++(0,1) to ++(1,0) 
            to ++(0,1) to (25,8);
            
            \draw [c2] (25,3) to ++(-3,0) to ++(0,1) to (25,4); 

            \draw [c1] (0,1) to ++(7,0) to ++(0,-1) to (25,0);

            \tikzwall{13}{13}{0}{0}{vx}{gray}
            \foreach \i in {1,2,3,4} {
                \pgfmathtruncatemacro{\yone}{12-\i}
                \pgfmathtruncatemacro{\yonee}{-1+\i}
                \draw [red,thick,dotted,rounded corners=5pt] (v25-\yonee) 
                to ++(7-0.5*\i,0) to ++(0,3.5) node [vx1] {\tiny 3}
                to ++(0,9.5-2*\i)
                to (v25-\yone);
                \pgfmathtruncatemacro{\ytwo}{13-\i}
                \pgfmathtruncatemacro{\iminus}{8-\i}
                \draw [blue,thick,dotted,rounded corners=5pt] (v0-\ytwo)
                to ++(-1-0.5*\i,0) to ++(0,1.5*\i) 
                to ++(10+2*\i,0) node [vx2]{\tiny 2} 
                to ++(17-\i,0) to ++(0,-0.5*\i-5-\i)                
                to (v25-\iminus);
            }   
            \foreach \i in {1,2,3,4}{
                \pgfmathtruncatemacro{\seriesone}{2*\i-1}
                \pgfmathtruncatemacro{\seriestwo}{2*\i}
                \draw [orange,thick,dotted,out=180,in=-180] (v0-\seriesone) 
                to ++(-3,.5) node [vx3]{\tiny 1}
                to (v0-\seriestwo) ;
            }
        \end{tikzpicture}
        \caption{An obstruction for cycles containing at least one vertex of $S_1$, odd number of vertices in $S_2$, and odd number of vertices in $S_3$, where vertices in $S_1$, $S_2$, and $S_3$ are shown in orange, blue, and red colours respectively. This is an example where there are two vertex-disjoint such cycles (marked in the figure) but no three vertex-disjoint such cycles.}
    \end{subfigure}
    \caption{Some obstructions for Erd\H{o}s-P\'{o}sa type results for constrained cycles. Dashed lines represent paths of odd length, solid lines (including lines in the wall) represent paths of even length, and dotted lines represent arbitrary paths.%
    }
    \label{fig:obstructions-intro}
\end{figure}

\subsection*{Our main theorem on group-labelled graphs} We consider a unified approach to discuss a vast number of such constraints in a common setting. 
For an abelian group~${\Gamma}$, a \emph{$\Gamma$-labelling} of a graph~$G$ is a function~${\gamma \colon E(G) \to \Gamma}$.
The \emph{$\gamma$-value} of a subgraph~$H$ of~$G$ is the sum of~${\gamma(e)}$ over all edges~$e$ in~$H$. 

Cycles of length~$\ell$ modulo~$z$ can be naturally encoded in the setting of~${\mathbb{Z}_z}$-labelled graphs, where each edge has value~$1$ and the target cycles have values exactly~$\ell$. 
Given a set~$S$, $S$-cycles can be encoded as non-zero cycles with respect to the $\mathbb{Z}$-labelling which assigns value~$1$ to edges incident with vertices in~$S$ and~$0$ to all other edges. 
Using multiple abelian groups, we may encode cycles satisfying several properties simultaneously.  
A more comprehensive discussion on how to encode different types of constraints can be found in~\cite{GollinHKKO2021}.

Gollin et al.~\cite{GollinHKKO2021} considered graphs labelled by multiple abelian groups and proved that a half-integral analogue of the Erd\H{o}s-P\'{o}sa theorem holds for cycles whose values avoid a fixed finite set for each abelian group.
Henceforth we shall call these cycles (or the cycles in any class for which we are attempting to prove or disprove an Erd\H{o}s-P\'{o}sa property) the \emph{allowable} cycles.
In this paper, we extend their work by proving necessary and sufficient conditions on such sets of values for which the allowable cycles satisfy an analogue of the Erd\H{o}s-P\'osa theorem.

In fact, we prove a characterisation of the structural obstructions to an analogue of the Erd\H{o}s-P\'osa theorem in this setting; this is a far-reaching generalisation of Reed's result~\cite{Reed1999}. 
These obstructions are described in Section~\ref{sec:structuralthm} 
(see Definition~\ref{def:obstructions}).
This allows us to prove, for all positive integers~$k$ and~$t$, that if a graph~$G$ is labelled with multiple abelian groups and~$\mathcal{O}$ is the set of cycles in~$G$ whose values avoid a fixed finite set for each abelian group, then $G$ contains either 
\begin{itemize}
    \item a \emph{packing} in~$\mathcal{O}$ of size~$k$ (that is a subset of~$\mathcal{O}$ of pairwise vertex-disjoint cycles), 
    \item a \emph{hitting set} for~$\mathcal{O}$ of bounded size (that is a set of vertices hitting each cycle in~$\mathcal{O}$), or
    \item a subgraph that is `equivalent' to some obstruction described in Definition~\ref{def:obstructions} 
    for some subset of abelian groups,
    and which contains a \emph{half-integral packing} in~$\mathcal{O}$ of size~$t$ (that is a subset of~$\mathcal{O}$ such that each vertex of~$G$ is contained in at most two of the cycles), but no packing of more than two cycles in~$\mathcal{O}$. 
\end{itemize}
For the precise statement, see Theorem~\ref{thm:mainobstruction}. 

Let us now give a loose description of the obstructions. 
Each obstruction consists of a wall, in which every cycle has value zero in every group, together with a collection of sets of paths arranged around the boundary of the wall, so that each set of paths is  `nested', `crossing', or `in series' (see Figure~\ref{fig:obstructions-intro} for examples). 
Moreover, this collection of sets is minimally sufficient to find allowable cycles, in that every allowable cycle contains a path from each of these sets, and every cycle which contains exactly one path from each set is allowable.
Additionally, every allowable cycle must contain an odd number of paths from each set that is not in series.
Finally, one of the following conditions must be satisfied: 
\begin{itemize}
    \item the number of crossing sets of paths is odd (see for example Figure~\ref{fig:obstructions-intro}(a)), 
    \item at least one but not all sets are arranged in series (see for example Figure~\ref{fig:obstructions-intro}(b)), or
    \item at least three sets of these paths are arranged in series (see for example Figure~\ref{fig:obstructions-intro}(c)). 
\end{itemize}

As we show in Subsection~\ref{subsec:reallyobstructions},  
these obstructions do not contain a packing of more than two allowable cycles. 
On the other hand, the minimum size of a hitting set for the allowable cycles can be made arbitrarily large by taking the wall and the sets of paths to be sufficiently large (see Theorem~\ref{thm:mainobstruction}).

\subsection*{Corollaries to graphs labelled by a single group} As a corollary of this structural result, we find a pair of necessary and sufficient conditions on abelian groups and forbidden values for an analogue of the Erd\H{o}s-P\'{o}sa theorem for allowable cycles as defined above. 
These conditions are motivated by the obstructions as follows. 
Let~$\Gamma$ be an abelian group and let~$A\subseteq \Gamma$ be a set of allowable values.
First, if there is an allowable value~$a\in A$ such that the subgroup $\gen{2a}$ generated by $2a$ does not contain an allowable value, then it is possible to construct an analogue of the Escher wall obstruction in Figure~\ref{fig:obstructions-intro}(a).
Second, if there exist three group elements generating an allowable value such that no two of them generate an allowable value, then it is possible to construct an analogue of the obstruction in Figure~\ref{fig:obstructions-intro}(c).

\begin{restatable}{theorem}{groupnecessary}
    \label{thm:maingroup2}
    Let~$A$ be a subset of an abelian group~$\Gamma$
    such that at least one of the following conditions fails to hold:
   \begin{enumerate}
        [label=(\arabic*)]
        \item \label{item:mainintro1-2} ${\gen{2a} \cap A \neq \emptyset}$ for all~${a \in A}$,
        \item \label{item:mainintro2-2} if~${a,b,c \in \Gamma}$ and ${\gen{a,b,c} \cap A \neq \emptyset}$, then ${(\gen{a,b} \cup \gen{b,c} \cup \gen{a,c}) \cap A \neq \emptyset}$.
    \end{enumerate}
    Then for every positive integer~$t$, there is a graph~$G_{\Gamma,A,t}$ with a $\Gamma$-labelling~$\gamma$ such that for the set~$\mathcal{O}$ of cycles of~$G_{\Gamma,A,t}$ with~$\gamma$-values in~$A$,
    there are no two vertex-disjoint cycles in~$\mathcal{O}$ and there is no hitting set for~$\mathcal{O}$ of size at most~$t$.
\end{restatable}
Note that conditions~\ref{item:mainintro1-2} and~\ref{item:mainintro2-2} are necessary in general.
It turns out that the obstruction in  Figure~\ref{fig:obstructions-intro}(b) does not give rise to any new necessary conditions here; the natural condition would be two elements~$a$ and~$b$ generating an allowable value such that~$2a$ and~$b$ do not generate an allowable value, but in this case it is easy to see that it is also possible to construct an analogue of the Escher wall.
We will see later that Figure~\ref{fig:obstructions-intro}(b) does give a new necessary condition when restricted to graphs embedded on a fixed compact orientable surface because such a surface does not admit an embedding of large Escher walls (see Theorem~\ref{thm:planarmod} and Subsection~\ref{subsec:orientable}).

Let us now discuss the sufficiency of conditions~\ref{item:mainintro1-2} and~\ref{item:mainintro2-2} in Theorem~\ref{thm:maingroup2}.
Labellings of graphs by multiple abelian groups~$\Gamma_1$, $\ldots$, $\Gamma_m$ can be regarded as a single~$\Gamma$-labelling by the product group~$\Gamma = \prod_{j\in[m]}\Gamma_j$.
In this case, for~${g = (g_j \colon j \in [m]) \in \Gamma}$, we write~${\pi_j(g)}$ to denote~${g_j \in \Gamma_j}$.
We prove that conditions~\ref{item:mainintro1-2} and~\ref{item:mainintro2-2} are sufficient under the additional assumption that~$A$ is the set of values avoiding a fixed finite set in each $\Gamma_j$. 

\begin{theorem}
    \label{thm:maingroup}
    For all positive integers~$m$ and~$\omega$, there is a function~${f_{m,\omega} \colon \mathbb{N} \to \mathbb{N}}$ satisfying the following property. 
    Let ${\Gamma = \prod_{j \in [m]} \Gamma_j}$ be a product of~$m$ abelian groups, and for each~${j \in [m]}$, let~$\Omega_j$ be a subset of~$\Gamma_j$ with~${\abs{\Omega_j} \leq \omega}$.
    Let~$A$ be the set of all elements~${g \in \Gamma}$ such that~${\pi_j(g) \in \Gamma_j \setminus \Omega_j}$ for all~${j \in [m]}$. 
    If 
    \begin{enumerate}
        [label=(\arabic*)]
        \item \label{item:mainintro1} ${\gen{2a} \cap A \neq \emptyset}$ for all~${a \in A}$ and 
        \item \label{item:mainintro2} if~${a,b,c \in \Gamma}$ and~${\gen{a,b,c} \cap A \neq \emptyset}$, then~${(\gen{a,b} \cup \gen{b,c} \cup \gen{a,c}) \cap A \neq \emptyset}$, 
    \end{enumerate}
    then for every $\Gamma$-labelled graph~$G$ with a $\Gamma$-labelling~$\gamma$ and its set~${\mathcal{O}}$ of all cycles whose $\gamma$-values are in~$A$
    and for all~${k \in \mathbb{N}}$, there exists 
    a set of~$k$ pairwise vertex-disjoint cycles in~$\mathcal{O}$
    or a hitting set for~$\mathcal{O}$ of size at most~${f_{m,\omega}(k)}$. 
\end{theorem}

Note that for fixed $m$ and $\omega$, Theorem~\ref{thm:maingroup} produces a single function $f_{m,\omega}$ that does not depend on the specific abelian groups considered. 
These theorems completely characterise when such a duality holds in the setting where allowable cycles are those whose values avoid a fixed finite subset of each abelian group. 
In particular, if~$\Gamma$ is finite, then this is a complete characterisation of the sets of allowable values satisfying this duality.

Considering additional restrictions on the structure of the group-labelled graphs,
we strengthen Theorem~\ref{thm:maingroup} by observing that when checking conditions~\ref{item:mainintro1} and~\ref{item:mainintro2}, we may ignore any group~$\Gamma_j$ for which every large subwall of~$G$ contains a cycle whose $\gamma$-value $g$ satisfies~$\pi_j(g)\neq0$ (see Theorem~\ref{thm:maingroup-involved} in Section~\ref{sec:applications}). 
This strengthening allows us to encode a wide variety of properties of cycles.
For example, for fixed integers~$p$,~$\ell$ and given a subgraph~$H$ of tree-width at most~$p$ in a graph~$G$, consider the cycles containing at least~$\ell$ edges not contained in~$H$.
Such cycles can be represented with the $\mathbb{Z}$-labelling which assigns value~$1$ to edges not in~$H$ and~$0$ to all edges in~$H$. 
If~$H$ has no edges, then these are exactly the cycles of length at least~$\ell$.

Theorem~\ref{thm:maingroup} does not hold if the size bound on~$\Omega_j$ is removed. 
To see this, we prove in Subsection~\ref{subsec:finiteA} that if~$\Gamma$ is infinite but the set~$A$ is finite, then a duality such as the one in Theorem~\ref{thm:maingroup} does not hold. 
In fact, no fractional version of the Erd\H{o}s-P\'osa theorem holds in this case.
 
\begin{theorem}
    Let~${A}$ be a finite nonempty subset of an infinite abelian group~$\Gamma$.
    For integers~${s \geq 2}$ and~${t \geq 1}$, 
    there is a graph~$G$ with a $\Gamma$-labelling~$\gamma$ 
    such that 
    \begin{itemize}
        \item for every set of~$s$ cycles of~$G$ whose $\gamma$-values are in~$A$, there is a vertex that belongs to all of the~$s$ cycles and 
        \item there is no hitting set of size at most~$t$ 
        for the set of all cycles of~$G$ whose $\gamma$-values are in~$A$.
    \end{itemize}
\end{theorem}

\noindent 
Note that Theorem~\ref{thm:maingroup} applies to some cases where~$A$ and~${\Gamma \setminus A}$ are both infinite, for example if~${m = 2}$, ${\Gamma = \mathbb{Z} \times \mathbb{Z}}$ and~${\Omega_1 = \Omega_2 = \{0\}}$.
Extending our characterisation to general~$A\subseteq \Gamma$ is left as an open problem (see Subsection~\ref{subsec:openproblems}).

\subsection*{Corollaries to graphs embeddable in a fixed orientable surface}
One nice upside of our structural main theorem is its application to graphs of bounded orientable genus (for example planar graphs). 
Since a fixed compact orientable surface does not admit an embedding of an arbitrarily large Escher wall, condition~\ref{item:mainintro1-2} (which prevented the Escher wall obstruction) is no longer necessary when restricted to graphs embeddable on this surface.
In Subsection~\ref{subsec:orientable}, we give a characterisation analogous to Theorems~\ref{thm:maingroup2} and~\ref{thm:maingroup} for graphs that are embeddable in a fixed compact orientable surface.
Consequently, we obtain the following analogue of Theorem~\ref{cor:mainmod}. 

\begin{restatable}{theorem}{planarmod} \label{thm:planarmod}
    Let~$\ell$ and~$z$ be integers with~${z \geq 2}$, let~${p_1^{a_1} \cdots p_n^{a_n}}$ be the prime factorisation of~$z$ with~${p_{i} < p_{i+1}}$ for all~${i \in [n-1]}$, and let~$\mathbb{S}$ be a compact orientable surface.
    The following statements are equivalent. 
    \begin{itemize}
        \item There is a function~${f \colon \mathbb{N} \to \mathbb{N}}$ such that 
        for every integer~$k$, 
        every graph embeddable in~$\mathbb{S}$ contains~$k$ vertex-disjoint cycles of length~$\ell$ modulo~$z$ or a set of at most~${f(k)}$ vertices hitting all such cycles. 
    \item Both of the following conditions are satisfied. 
        \begin{enumerate}
            [label=(\arabic*)]
            \item \label{enum:planarmod1} 
            If~${p_1 = 2}$, then ${\ell \equiv 0 \pmod {p_1^{a_1}}}$ or~${\ell \equiv 0 \pmod {z/p_1^{a_1}}}$. 
            \item \label{enum:planarmod2} 
            There do not exist distinct~${i_1, i_2, i_3 \in [n]}$ such that~${\ell \not\equiv 0 \pmod {p_{i_j}^{a_{i_j}}}}$ for each~${j \in [3]}$.
        \end{enumerate}
    \end{itemize}
\end{restatable}

\noindent
For graphs embedded in a compact orientable surface, our results allow us to derive an Erd\H{o}s-P\'{o}sa type theorem for the cycles whose $\mathbb{Z}_2$-homology class is in a fixed set of allowable values. 
This result complements an analogous half-integral Erd\H{o}s-P\'{o}sa type theorem for graphs embedded in an arbitrary compact surface 
(see~\cite[Corollary~8.10]{GollinHKKO2021}). 
We discuss this in more detail in Subsection~\ref{subsec:orientable}. 
Previously, 
Kawarabayashi and Nakamoto~\cite{KenN2007} proved a similar result for odd cycles in graphs embedded in a fixed orientable surface, which Conforti, Fiorini, Huynh, Joret, and Weltge~\cite{CFHJW2020} extended to \emph{$2$-sided} odd cycles in graphs embedded in any fixed surface.

\subsection*{Related work}
It is worth taking a moment to highlight the differences between our results and the work of Huynh, Joos, and Wollan~\cite{HuynhJW2017}, who considered group labellings of orientations of edges in a graph, where the two orientations of each edge are assigned labels that are inverse to each other. 
For a graph imbued with two such labellings, they considered cycles with non-zero value in each coordinate and obtained a structural result analogous to our structural main theorem for these cycles.
There is no general translation between the labellings of edges which we use and the labellings of orientations of edges which they considered, but many interesting properties can be encoded in either setting. 
As an example, they apply their result to obtain canonical obstructions to an Erd\H{o}s-P\'{o}sa type result for odd cycles intersecting a prescribed set~$S$, and our structural theorem gives the same result. 
Whereas their result applies to arbitrary groups, dealing with non-abelian groups is more complicated in our setting, and it is unclear 
how to extend our result to non-abelian groups. 
However, modularity constraints with modulus greater than~$2$ cannot be encoded in their setting in general. 
Furthermore, we do not only consider the cycles that are non-zero in each coordinate, and we are able to consider any finite number of group labellings.

The results of this paper unify and generalise many of the previous results in this area, including all of the results we have mentioned so far. 
In particular, Theorem~\ref{thm:Sml} characterises when an Erd\H{o}s-P\'{o}sa type result holds for $\mathcal{S}$-cycles of length~$\ell$ modulo~$z$ and length at least~$L$, which yields Theorem~\ref{cor:mainmod} as a special case (which generalises the aforementioned results of Thomassen~\cite{Thomassen1988} and of Thomas and Yoo~\cite{YooR2020}), 
and also recovers the results for $S$-cycles of Kakimura, Kawarabayashi, and Marx~\cite{KakimuraKM2011}, for $S$-cycles of length at least~$L$ of Bruhn, Joos, and Schaudt~\cite{BruhnJS2017}, and for $(S_1,S_2)$-cycles of Huynh, Joos, and Wollan~\cite{HuynhJW2017}. 
Wollan~\cite{Wollan2011} proved that when an abelian group~$\Gamma$ has no element of order~$2$, an analogue of the Erd\H{o}s-P\'{o}sa theorem holds for non-zero cycles in $\Gamma$-labelled graphs. 
This result is recovered by taking~${\Gamma_1 = \Gamma}$ and~${\Omega_1 = \{0\}}$ in Theorem~\ref{thm:maingroup}.
Huynh, Joos, and Wollan~\cite{HuynhJW2017} used group labellings of orientations of edges to show that non-null-homologous (in the $\mathbb{Z}$-homology group) cycles in graphs embedded in a fixed compact orientable surface satisfy an Erd\H{o}s-P\'{o}sa type theorem. 
In Subsection~\ref{subsec:orientable} we recover this result. 

\subsection*{Vertex-labelled graphs}
Our results can also be applied to the setting where vertices instead of edges are labelled. 
Gollin et al.~discussed in~\cite{GollinHKKO2021} a method of converting between vertex-labelled graphs and edge-labelled graphs.
Unfortunately, the translation they described effects the structure of the abelian group, and is therefore not immediately sufficient for proving a vertex-labelling analogue of Theorem~\ref{thm:maingroup}.
However by carefully adapting the main structural theorem of this paper to the setting of vertex-labellings we can obtain analogues of all of our results which reference edge-labellings, as we demonstrate in Subsection~\ref{subsec:vertexlabellings}.

\subsection*{Overview
}
This paper is organised as follows. 
In Section~\ref{sec:prelim}, we introduce some preliminary concepts and notation. 
In Section~\ref{sec:structuralthm}, we state our main structural theorem and give a high-level overview of its proof.
In Section~\ref{sec:recycle}, we recall useful lemmas from the literature, especially from~\cite{GollinHKKO2021}. 
In Section~\ref{sec:handlebars}, we discuss how to find sets of paths arranged nicely around the boundary of the wall as in Figure~\ref{fig:obstructions-intro}, which we call \emph{handlebars}, and what to do with them once we have found them. 
In Section~\ref{sec:grouplemmas}, we present some useful lemmas about abelian groups. 
We complete the proof of our main structural result in Section~\ref{sec:proof}. 
Finally, in Section~\ref{sec:applications}, we demonstrate how to derive our other results and applications
and we also present some open problems.

\section{Preliminaries}
\label{sec:prelim}

All graphs in this paper are undirected simple graphs that have neither loops nor parallel edges. 
For an integer~$m$, we write~${[m]}$ for the set of positive integers~${z}$ with~${z \leq m}$. 

Let~$G$ be a graph. 
We denote by~${V(G)}$ and~${E(G)}$ the vertex set and the edge set of~$G$, respectively. 
For a vertex set~$A$ of~$G$, we denote by~${G - A}$ the graph obtained from~$G$ by deleting all the vertices in~$A$ and all edges incident with vertices in~$A$, 
and denote by~${G[A]}$ the subgraph of~$G$ induced by~$A$, which is~${G-(V(G)\setminus A)}$.
If~${A = \{v\}}$, then we write~${G - v}$ for~${G - A}$. 
For an edge~$e$ of~$G$, we denote by~${G - e}$ the graph obtained by deleting~$e$. 
For two graphs~$G$ and~$H$, let 
\[
    {G \cup H := (V(G) \cup V(H), E(G) \cup E(H))}
    \ \textnormal{ and } \ 
    {G \cap H := (V(G) \cap V(H), E(G) \cap E(H))}. 
\]
For a set~$\mathcal{G}$ of graphs, we denote by~${\bigcup \mathcal{G}}$ the union of the graphs in~$\mathcal{G}$. 
By slight abuse of notation, we say two sets~$\mathcal{G}_1$ and~$\mathcal{G}_2$ of graphs are \emph{vertex-disjoint} if the graphs~${\bigcup \mathcal{G}_1}$ and~${\bigcup \mathcal{G}_2}$ are vertex-disjoint. 

Let~$A$ and~$B$ be vertex sets of~$G$. 
An \emph{${(A, B)}$-path} is a path from a vertex in~$A$ to a vertex in~$B$ such that all internal vertices are not contained in~${A \cup B}$. 
An \emph{$A$-path} is a nontrivial~${(A,A)}$-path. 
For a subgraph~$H$ of~$G$, we refer to a ${V(H)}$-path as an \emph{${H}$-path} for brevity.

For a graph~$G$, let~$\branch(G)$ denote the set of vertices of~$G$ whose degree is not equal to~$2$. 

\emph{Subdividing} an edge~$uv$ in a graph~$G$ is an operation that yields a graph obtained by removing the edge~$uv$ and adding new edges~$uw$ and $vw$ for some vertex~$w$ not in~$G$.
A graph~$H$ is a \emph{subdivision} of a graph~$G$ if~$H$ can be obtained from~$G$ by subdividing edges repeatedly.

\subsection{Walls}

Let~${c,r \geq 3}$ be integers. 
The \emph{elementary $(c,r)$-wall $W_{c,r}$} is the graph obtained from the graph on the vertex set~${[2c] \times [r]}$ whose edge set is 
\begin{align*}
    \left\{ (i,j) (i+1,j) \, \colon \, i \in [2c-1],\, j \in [r] \right\} 
    \cup 
    \left\{ (i,j) (i,j+1) \, \colon \, i \in [2c],\, j \in [r-1],\, i+j \textnormal{ is odd} \right\}
\end{align*}
by deleting both degree-$1$ vertices.

\begin{figure}%

    \begin{subfigure}{0.45\linewidth}
    \centering
        \begin{tikzpicture}
            [scale=0.5]
            \tikzset{vx/.style = {circle, draw, fill=black!0, inner sep=0pt, minimum width=4pt}}
            \tikzset{vx2/.style = {circle, draw, fill=black!50, inner sep=0pt, minimum width=4pt}}
            
            \tikzwall{6}{7}{0}{0}{vx}{gray}
            \tikzwall{3}{7}{4}{0}{vx}{red, very thick, dashed}
            \tikzwall{6}{3}{0}{1}{vx}{blue, very thick, dashed, dash phase = 3pt}
            
            \node[vx2] at (5,0) {};
            \node[vx2] at (8,6) {};
            \foreach \i in {1,...,6} {
                \node[vx2] at (4,\i) {};
                \node[vx2] at (9,6-\i) {};
            }        
        \end{tikzpicture}
        \caption{A $(6,7)$-wall $W$ with a $3$-column-slice $W'$ highlighted in red and a $3$-row-slice highlighted in blue. 
        The column-boundary of the red $3$-column-slice is indicated by the solid vertices. }
    \end{subfigure}
    \hfill 
    \begin{subfigure}{0.45\linewidth}
        \centering
            \begin{tikzpicture}
                [scale=0.5]
                \tikzset{vx/.style = {circle, draw, fill=black!0, inner sep=0pt, minimum width=4pt}}
                \tikzset{vx2/.style = {circle, draw, fill=black!50, inner sep=0pt, minimum width=4pt}}
               \tikzset{c1/.style={brown, line width=4pt,opacity=0.8,rounded corners}}
                
                \tikzwall{3}{7}{4}{0}{vx}{red, very thick, dashed}
                \tikzwall{6}{7}{0}{0}{vx}{gray}
                
                \node[vx2] at (5,0) {};
                \node[vx2] at (8,6) {};
                \foreach \i in {1,...,6} {
                    \node[vx2] at (4,\i) {};
                    \node[vx2] at (9,6-\i) {};
                }
                \draw [blue,dotted,thick](v0-3) ..controls (-2,3) and (-2, 4) .. 
                (v0-5) ;    
                \draw [c1] (v0-3) -- (v1-3)--(v2-3)--(v3-3)--(v4-3);
                \draw [c1] (v0-5)--(v1-5)--(v2-5)--(v3-5) -- (v4-5);

                \draw [blue,dotted,thick,rounded corners=5pt ](v0-1) to ++(-1,0) to ++(0,-2) to ++(14,0) to ++(0,3) to (v11-2);
                \draw [c1] (v0-1) -- (v1-1)--(v2-1)--(v3-1)--(v4-1);
                \draw [c1] (v11-2)--(v10-2)--(v9-2);

            \end{tikzpicture}
            \caption{The blue dotted paths represent two $W$-handles 
            and the row-extension of each blue path to $W'$ is obtained by extending it by two brown paths.}
        \end{subfigure}
    \caption{Illustrations of $c$-column-slices, $r$-row-slices, column-boundaries, $W$-handles, and the row-extension of a $W$-handle to $W'$.}
    \label{fig:wall}
\end{figure}

For ${j \in [r]}$, the \emph{$j$-th row}~$R_j$ of~$W_{c,r}$ is the path~${W_{c,r} \big[ \big\{ (i, j)\in V(W_{c,r}) \, \colon \, i \in [2c] \big\} \big]}$. 
For ${i \in [c]}$, the \emph{$i$-th column}~$C_i$ of~$W_{c,r}$ is the path~${W_{c,r} \big[ \big\{ (i',j)\in V(W_{c,r}) \, \colon \, i' \in \{ 2i - 1, 2i \}, \, j \in [r] \big\} \big]}$. 

A \emph{$(c,r)$-wall} is a subdivision~$W$ of the elementary $(c,r)$-wall. 
If~$W$ is a~$(c,r)$-wall for some suitable integers~$c$ and~$r$, then we say~$W$ is a \emph{wall of order~$\min\{c,r\}$}. 
We call a vertex corresponding to the vertex~${(i,j)}$ of the elementary wall a \emph{nail} of~$W$, and denote by~$N^W$ the set of nails of~$W$. 

For a $(c,r)$-wall~$W$ and a subgraph~$H$ of the elementary $(c,r)$-wall~$W_{c,r}$, we denote by~$H^W$ the subgraph of~$W$ corresponding to a subdivision of~$H$.  
We call~$R_j^W$ or $C_i^W$ the \emph{$j$-th row} or \emph{$i$-th column} of~$W$, respectively. 
A subgraph~$W'$ of a wall~$W$ that is itself a wall is called a \emph{subwall of~$W$}. 
For a set~$S$ of vertices, we say a wall~$W$ is \emph{$S$-anchored} if~${N^W \subseteq S}$. 

For an integer~${c \geq 3}$, we call a subwall~$W'$ of a wall~$W$ a \emph{$c$-column-slice of~$W$} if 
\begin{itemize}
    \item the set of nails of~$W'$ is exactly~${N^W \cap V(W')}$,
    \item there is a column of~$W'$ that is a column of~$W$, 
and
    \item $W'$ has exactly~$c$ columns. 
\end{itemize}
See Figure~\ref{fig:wall} for an example. 
Similarly, for an integer~${r \geq 3}$, we call a subwall~$W'$ of a wall~$W$ an \emph{$r$-row-slice of~$W$} if 
\begin{itemize}
    \item the set of nails of~$W'$ is exactly~${N^W \cap V(W')}$,
    \item there is a row of~$W'$ that is a row of~$W$, 
    and
    \item $W'$ has exactly~$r$ rows. 
\end{itemize}
Note that in an $r$-row-slice~$W'$ of~$W$, depending on the location, the first column of~$W'$ may be in the last column of~$W$ by the definition of a wall.

Let~$W$ be a wall in a graph $G$. 
The \emph{column-boundary of~$W$} is the set of all endvertices of rows of~$W$. 
A \emph{$W$-handle} is a $W$-path in $G$ whose endvertices are in the column-boundary of~$W$. 

Let~$W$ be a ${(c,r)}$-wall and let~$W'$ be a $c'$-column-slice of~$W$ for some~${3 \leq c' \leq c}$. 
For a path~$P$ whose endvertices are nails of~$W$,
the \emph{row-extension of~$P$ to~$W'$ in~$W$}
is a $W'$-handle containing~$P$ that is contained in the union of~$P$ and the rows of~$W$. 
We can easily observe that if such a $W'$\nobreakdash-handle exists, then it is unique.
Note that the row-extension of a $W$-handle to~$W'$ always exists. 
For a set~$\mathcal{P}$ of pairwise vertex-disjoint $W$-handles, we define the \emph{row-extension of~$\mathcal{P}$ to~$W'$ in~$W$} to be the set of row-extensions of the paths in~$\mathcal{P}$ to~$W'$ in~$W$. 
Note that these $W'$-handles are also pairwise vertex-disjoint. 
See Figure~\ref{fig:wall} for an illustration.

\subsection{Groups}
For a non-empty set~$S=\{a_i:i\in[t]\}$ of elements in a group~$\Gamma$, we write $\langle S\rangle$ or $\langle a_i: i\in [t]\rangle$ for the subgroup of~$\Gamma$ generated by~$S$, which is the intersection of all subgroups of~$\Gamma$ containing~$S$.

The \emph{direct product} of groups $\Gamma_1$, $\Gamma_2$, $\ldots$, $\Gamma_m$ is denoted by $\prod_{i=1}^m \Gamma_i$. We write $\pi_j$ for the projection map from $\prod_{i=1}^m \Gamma_i$ to $\Gamma_j$ for each~${j \in [m]}$.  For an element $g\in \prod_{i=1}^m \Gamma_i$, we call $\pi_j(g)$ the \emph{$j$-th coordinate of~$g$}.

For a subgroup $\Lambda$ of an abelian group~$\Gamma$, we denote by $\Gamma/\Lambda:=\{a+\Lambda: a\in \Gamma\}$ the \emph{quotient group} of~$\Gamma$ by $\Lambda$, which is the set of cosets $a+\Lambda:=\{a+b:b\in \Lambda\}$ of $\Lambda$ in~$\Gamma$, where $(a+\Lambda)+(b+\Lambda)=(a+b)+\Lambda$ and $-(a+\Lambda)=-a+\Lambda$ for all $a,b\in \Gamma$.

\subsection{Group-labelled graphs}

Let~$\Gamma$ be an abelian group. 
A \emph{$\Gamma$-labelled graph} is a pair of a graph~$G$ and a function~${\gamma \colon E(G) \to \Gamma}$.
We say that~$\gamma$ is a \emph{$\Gamma$-labelling} of~$G$. 
A \emph{subgraph} of a $\Gamma$-labelled graph~${(G,\gamma)}$ is a $\Gamma$-labelled graph~$(H,\gamma')$ such that~$H$ is a subgraph of~$G$ and~$\gamma'$ is the restriction of~$\gamma$ to~$E(H)$. 
By a slight abuse of notation, we may refer to this $\Gamma$-labelled graph by~$(H,\gamma)$. 

For a $\Gamma$-labelled graph~${(G,\gamma)}$ and a subgraph~${H \subseteq G}$, we define~$\gamma(H)$ as~${\sum_{e \in E(H)} \gamma(e)}$, 
which we call the \emph{$\gamma$-value of~$H$}. 
Note that this definition implies that the $\gamma$-value of the empty subgraph is~$0$. 
We say that a subgraph~$H$ is \emph{$\gamma$-non-zero} if~${\gamma(H) \neq 0}$, and otherwise, we call it \emph{$\gamma$-zero}. 
A $\Gamma$-labelled graph~${(G,\gamma)}$ is \emph{$\gamma$-bipartite} if every cycle of~$G$ is $\gamma$-zero. 

We will often consider the special case where~$\Gamma$ is the product~${\prod_{j \in [m]} \Gamma_j}$ of~$m$ abelian groups for a positive integer~$m$. 
In this case, we denote by~$\gamma_j$ the composition of~$\gamma$ with the projection to~$\Gamma_j$. 
For a subset~${J \subseteq [m]}$ we denote by~$\Gamma_J$ be the subgroup of~$\Gamma$ of all~${g \in \Gamma}$ with~${\pi_j(g) = 0}$ for all~${j \in [m] \setminus J}$. 

We frequently take a subgroup~$\Lambda$ of~$\Gamma$ and consider a new labelling using the quotient group~${\Gamma/\Lambda}$. 
For a $\Gamma$-labelled graph~${(G, \gamma)}$ and a subgroup~$\Lambda$ of~$\Gamma$, 
the \emph{induced $(\Gamma/\Lambda)$-labelling} of~$(G,\gamma)$ is 
the~${\Gamma/\Lambda}$\nobreakdash-labelling~$\lambda$ defined by~${\lambda(e) := \gamma(e) + \Lambda}$ for all edges~${e \in E(G)}$.

Let~$x$ be a vertex of~$G$ and let~${\delta \in \Gamma}$ be an element of order~$2$. 
For each edge~$e$ of~$G$, let 
\[
    \gamma'(e)=
    \begin{cases}
        \gamma(e) + \delta & \text{if $e$ is incident with $x$,}\\
        \gamma(e) &\text{otherwise.}
    \end{cases}
\]
We say that~$\gamma'$ is obtained from~$\gamma$ by \emph{shifting by~$\delta$ at~$x$}. 
Observe that this shift does not change the value of a cycle because~${\delta + \delta = 0}$.
We say two $\Gamma$-labellings~$\gamma_1$ and~$\gamma_2$ of~$G$ are \emph{shifting-equivalent} 
if~$\gamma_1$ can be obtained from~$\gamma_2$ by a sequence of shifting operations.

\section{The structural main theorem}
\label{sec:structuralthm}

\subsection{Handlebars}
\label{subsec:handlebars}

Let~${(X,\prec)}$ be a linearly ordered set. 
We say two disjoint subsets~${\{x_1,x_2\}}$ and~${\{y_1,y_2\}}$ of~$X$ of size~$2$ with~${x_1 \prec x_2}$ and~${y_1 \prec y_2}$ are
\begin{itemize}
    \item \emph{in series} if either~${x_2 \prec y_1}$ or~${y_2 \prec x_1}$;
    \item \emph{nested} if either~${x_1 \prec y_1 \prec y_2 \prec x_2}$ or~${y_1 \prec x_1 \prec x_2 \prec y_2}$; and
    \item \emph{crossing} otherwise.
\end{itemize}
A set~${S \subseteq \binom{X}{2}}$ of pairwise disjoint sets is \emph{in series}, \emph{nested}, or \emph{crossing}, respectively, if its elements are pairwise in series, nested, or crossing, respectively, and~$S$ is called \emph{pure} if it is in series, nested, or crossing. 

A straightforward argument shows the following lemma (see also~\cite[Lemma 25]{HuynhJW2017}). 
\begin{lemma}
    \label{lem:pureset}
    Let~$t$ be a positive integer, let~${(X,\prec)}$ be a linearly ordered set, and let~${S \subseteq \binom{X}{2}}$ be a set of pairwise disjoint sets. 
    If~${\abs{S}>t^3}$, then~$S$ contains a pure subset of size greater than~$t$. 
\end{lemma}

\begin{proof}
    First consider the partial order~$\prec_1$ on~$S$ 
    such that for~${\{a,b\}, \{c,d\} \in S}$ with~${a \prec b}$ and~${c \prec d}$ we have~${\{a,b\} \prec_1 \{c,d\}}$ if~${b \prec c}$. 
    By Dilworth's Theorem~\cite{Dilworth1950}, $S$ contains a chain of size greater than~$t$ with respect to~$\prec_1$ or an antichain of size greater than~$t^2$ with respect to~$\prec_1$. 
    In the first case we have a subset that is in series, so suppose instead that there is some~${S' \subseteq S}$ of size greater than~$t^2$ that is an antichain with respect to~$\prec_1$. 
    Let~$\prec_2$ be the partial order on~$S'$ 
    such that for~${\{a,b\}, \{c,d\} \in S}$ with~${a \prec b}$ and~${c \prec d}$ we have~${\{a,b\} \prec_2 \{c,d\}}$ 
    if~${a \prec c}$ and~${b \prec d}$. 
    Again by Dilworth's Theorem, there is some~${S'' \subseteq S'}$ of size greater than~$t$ such that~$S''$ is a chain or an antichain with respect to~$\prec_2$, 
    and hence crossing or nested, respectively. 
\end{proof}

Let~$W$ be a ${(c,r)}$-wall. 
Let~$\prec_W$ be the linear order on the column-boundary of~$W$ such that~${v \prec_W w}$ if at least one of the following conditions holds. 
\begin{itemize}
    \item $v$ is in the first column and~$w$ is in the last column. 
    \item Both~$v$ and~$w$ are in the first column and the index of the row containing~$v$ is lower than the index of the row containing~$w$. 
    \item Both~$v$ and~$w$ are in the last column and the index of the row containing~$v$ is higher than the index of the row containing~$w$. 
\end{itemize}

A set~$\mathcal{P}$ of $W$-handles is \emph{pure}, \emph{nested}, \emph{in series}, or \emph{crossing}, respectively, if the set of sets of endvertices of all the paths in~$\mathcal{P}$ is pure, nested, in series, or crossing, respectively, with respect to~$\prec_W$. 
We call a set~$\mathcal{P}$ of pairwise vertex-disjoint $W$-handles a \emph{$W$-handlebar} if~$\mathcal{P}$ is pure and there are two paths~$A$ and~$B$ in~${C_1^W \cup C_c^W}$ such that each $W$-handle in~$\mathcal{P}$ is a ${(V(A),V(B))}$-path. 
Observe that if $\mathcal P$ is a $W$-handlebar in series having at least two $W$-handles, then 
all the endvertices of $W$-handles in $\mathcal P$ are in $C_i^W$ for some $i\in\{1,c\}$.

Two $W$-handlebars $\mathcal{P}_1$ and~$\mathcal{P}_2$ are \emph{non-mixing} if for each~${i \in [2]}$ there are (not necessarily disjoint) paths~$A_i$ and~$B_i$ in~${C_1^W \cup C_c^W}$ such that~$\mathcal{P}_i$ is a set of~${(V(A_i),V(B_i))}$-paths and~${A_1 \cup B_1}$ and~${A_2 \cup B_2}$ are vertex-disjoint.

\subsection{The main theorem}
\label{subsec:mainthm-statement}

We now define the possible obstructions for the Erd\H{o}s-P\'{o}sa property in graphs labelled with multiple abelian groups 
for cycles whose values avoid a bounded number of elements in each group. 

\begin{definition}
    \label{def:obstructions}
    For positive integers~$\kappa$ and~$\theta$, an abelian group~$\Gamma$, and~${A \subseteq \Gamma}$, 
    let~${\mathcal{C}(\kappa,\theta,\Gamma,A)}$ be the class of all $\Gamma$-labelled graphs~${(G,\gamma)}$ 
    having 
    a wall~$W$ of order at least~$\theta$ and 
    a nonempty family~${( \mathcal{P}_i \colon i \in [t] )}$ of pairwise vertex-disjoint non-mixing $W$-handlebars each of size at least~$\kappa$
    such that 
    \begin{enumerate}
        [label=(O\arabic*)]
        \item \label{item:obstructions-union} $G$ is the union of $W$ and  $\bigcup \{ \bigcup \mathcal{P}_i \colon i \in [t] \}$, 
        \item \label{item:obstructions-zerowall}
            every $N^W$-path in $W$ is $\gamma$-zero, 
        \item \label{item:obstructions-allowabletransversals}
            ${\sum_{i \in [t]} \gamma(P_i) \in A}$ for any family ${(P_i \colon i\in [t])}$ such that~${P_i \in \mathcal{P}_i}$ for all~${i \in [t]}$, 
        \item \label{item:obstructions-minimality}
            for each~${i \in [t]}$, we have~${\gen{\gamma(P) \colon P \in \bigcup_{j \in [t] \setminus \{i\}} \mathcal{P}_j} \cap A = \emptyset}$, 
        \item \label{item:obstructions-even}
            if ${\sum_{j \in [t]} \sum_{P \in \mathcal{P}_j} f(P) \gamma(P) \in A}$ for a function~${f \colon \bigcup_{j \in [t]} \mathcal{P}_j \to \mathbb{Z}}$, then 
            for each ${i \in [t]}$, 
            $\mathcal{P}_i$ is in series
            or~${\sum_{P \in \mathcal{P}_i} f(P)}$ is odd, and 
        \item \label{item:obstructions-handlebars}
            at least one of the following properties holds. 
            \begin{enumerate}
                [label=(O6\alph*)]
                \item \label{subitem:obstructions-oddcrossing} 
                The number of crossing $W$-handlebars in~${( \mathcal{P}_i \colon i \in [t] )}$ is odd. 
                \item \label{subitem:obstructions-seriesnonseries} At least one but not all $W$-handlebars in~${( \mathcal{P}_i \colon i \in [t] )}$ are in series. 
                \item \label{subitem:obstructions-3series} At least three $W$-handlebars in~${( \mathcal{P}_i \colon i \in [t] )}$ are in series. 
            \end{enumerate}
    \end{enumerate}
\end{definition}

Now we can state our main theorem. 
Recall that for a product~${\Gamma = \prod_{j \in [m]} \Gamma_j}$ of~$m$ abelian groups and for a subset~${J \subseteq [m]}$, we denote by~$\Gamma_J$ the subgroup consisting of all~${g \in \Gamma}$ with~${\pi_j(g) = 0}$ for all~${j \in [m] \setminus J}$. 

\begin{restatable}{theorem}{mainobstruction}
    \label{thm:mainobstruction}
    For all positive integers~$m$ and~$\omega$, there is a function~${\widehat f_{m,\omega} \colon \mathbb{N}^3 \to \mathbb{Z}}$ satisfying the following property. 
    Let~${\Gamma = \prod_{j \in [m]} \Gamma_j}$ be a product of~$m$ abelian groups, and for every~${j \in [m]}$, let~$\Omega_j$ be a subset of~$\Gamma_j$ with~${\abs{\Omega_j} \leq \omega}$. 
    For each~${j \in [m]}$, let~${A_j := \pi_j^{-1}(\Gamma_j\setminus \Omega_j)\subseteq \Gamma}$ 
    and ${A := \bigcap_{j \in [m]} A_j}$. 
    Let~$G$ be a graph with a $\Gamma$-labelling~$\gamma$ and let~$\mathcal{O}$ be the set of all cycles of~$G$ whose $\gamma$-values are in~$A$. 
    Then for every three positive integers~$k$, $\kappa$, and~$\theta$, there exists a $\Gamma$\nobreakdash-labelling~$\gamma'$ of~$G$ that is shifting equivalent to~$\gamma$ such that at least one of the following statements is true. 
    \begin{enumerate}
        [label=(\roman*)]
        \item \label{item:main-packing} There are~$k$ vertex-disjoint cycles in~$\mathcal{O}$. 
        \item \label{item:main-hittingset} There is a hitting set for~$\mathcal{O}$ of size at most~${\widehat f_{m,\omega}(k, \kappa, \theta)}$. 
        \item \label{item:main-obstruction} There is a subgraph~$H$ of~$G$ such that for some~${J \subseteq [m]}$ and for the $\left(\Gamma / \Gamma_J \right)$-labelling~$\gamma''$ induced by the restriction of~$\gamma'$ to~$H$, we have $(H,\gamma'') \in \mathcal{C}(\kappa, \theta, \Gamma / \Gamma_J, A + \Gamma_J )$ and~$H$ contains a half-integral packing of~$\kappa$ cycles in~$\mathcal{O}$.
    \end{enumerate}
\end{restatable}

Note that if~$G$ is a graph for which statement~\ref{item:main-obstruction} of Theorem~\ref{thm:mainobstruction} holds with~${H = G}$, then the half-integral packing of size~$\kappa$ is a witness that there is no hitting set for the cycles in~$\mathcal{O}$ of size less than~${\kappa/2}$. 
In Subsection~\ref{subsec:reallyobstructions}, we will establish 
that $H$ has no three vertex-disjoint cycles in~$\mathcal{O}$. 
In this sense, these graphs form obstructions for an Erd\H{o}s-P\'{o}sa type result.

\subsection{Proof Sketch}
\label{subsec:proofsketch}

The proof proceeds by induction on~$k$ and will mostly follow a well-established proof structure in this area, using new ideas developed here as well as in~\cite{GollinHKKO2021}. 
We assume that all three statements of Theorem~\ref{thm:mainobstruction} fail, and find a tangle that is `oriented towards' a minimum hitting set, see Lemma~\ref{lem:welllinked}. 
This tangle allows us to find a large wall in~$G$ (see Subsection~\ref{subsec:findingthewall}) which in particular has the property that it cannot be separated from cycles in~$\mathcal{O}$ by a small set of vertices. 
We then (see Subsection~\ref{subsec:cleaning}) take a sufficiently large subwall~$W$ 
which after possibly shifting the labelling satisfies for some set~${Z \subseteq [m]}$ of coordinates
\begin{itemize}
    \item every $N^W$-path is $\gamma_j$-zero for all~${j \in Z}$,
    \item every large subwall contains a $\gamma_j$-non-zero cycle for all~${j \in [m] \setminus Z}$. 
\end{itemize}
As in~\cite{GollinHKKO2021}, we adapt a theorem of Wollan (see Subsection~\ref{subsec:collectinghandles}) to obtain a collection~$\mathcal{P}$ of $W$-handles that is sufficient to generate a value that is `allowable' with respect to the coordinates in~$Z$. 
These handles will be partitioned into sets such that for each coordinate in~$Z$, all handles within a set either have the same values or distinct values. 
We restrict the parts of this partition to pairwise non-mixing handlebars and throw away any handlebar that is unnecessary to generate a value that is `allowable' with respect to the coordinates in~$Z$. 
Using the outer columns of the wall, we combine handles within each handlebar to form a set of handlebars for a subwall of~$W$ which now satisfies property~\ref{item:obstructions-allowabletransversals}. 
Since we already threw away all unnecessary handlebars, property~\ref{item:obstructions-minimality} is also satisfied. 
Each new handlebar whose handles contain an even number of handles in~$\mathcal{P}$ will now be in series, which allows us to do this in such a way that we additionally satisfy property~\ref{item:obstructions-even}.

Following the approach from~\cite{GollinHKKO2021}, we find a half-integral packing of cycles in~$\mathcal{O}$ of size~$\kappa$, and since we assumed that statement~\ref{item:main-obstruction} fails, we can conclude that property~\ref{item:obstructions-handlebars} fails. 
This means either that each handlebar is in series and there are at most two of them, or that no handlebar is in series and the number of crossing handlebars is even.
In the first case, it is not hard to find a packing of~$k$ cycles whose values are allowable with respect to the coordinates in~$Z$, and techniques from~\cite{GollinHKKO2021} enable us to deal with the coordinates in~${[m] \setminus Z}$ easily and obtain a packing of $k$~cycles in~$\mathcal{O}$ (see Subsection~\ref{subsec:cleaning}). 
In the second case, we iteratively combine pairs of `adjacent' handlebars to obtain one handlebar for a subwall of~$W$, where each new handle contains exactly one handle of each constituent handlebar. 
This will form a nested handlebar, enabling us once again to find a packing of $k$~cycles in~$\mathcal{O}$.
Thus, we have the desired contradiction in each case.

\section{Recycled tools}
\label{sec:recycle}

\subsection{Finding the wall}
\label{subsec:findingthewall}

A \emph{separation} of a graph~$G$ is a pair~${(A, B)}$ of subsets of~$V(G)$ such that~${G[A] \cup G[B] = G}$. 
Its \emph{order} is defined to be~$\abs{A \cap B}$. 
For a positive integer~$t$, 
a set~$\cT$ of separations of order less than~$t$ is a \emph{tangle of order~$t$} in~$G$ if it satisfies the following. 
\begin{enumerate}
    [label=(\arabic*)]
    \item If ${(A, B)}$ is a separation of~$G$ of order less than~$t$, then~$\cT$ contains exactly one of~${(A, B)}$ and~${(B, A)}$. 
    \item If~${(A_1, B_1), (A_2, B_2), (A_3, B_3) \in \cT}$, then~${G[A_1] \cup G[A_2] \cup G[A_3] \neq G}$. 
\end{enumerate}

For a~${(g,g)}$-wall~$W$, let~$\mathcal{T}_W$ be the set of all separations~${(A,B)}$ of~$G$ of order less than~$g$ such that~${G[B]}$ contains a row of~$W$. 
Kleitman and Saks (see~\cite[(7.3)]{RS1991}) showed that~$\cT_W$ is a tangle of order~$g$. 
A tangle~$\cT$ in~$G$ \emph{dominates} the wall~$W$ if~${\cT_W \subseteq \cT}$.
The following theorem of Robertson, Seymour, and Thomas~\cite{RST1994} shows that every tangle in a graph of sufficiently large order dominates a wall. A better bound can be obtained by combining the results of Chuzhoy and Tan~\cite{ChuzhoyT2019} and Kawarabayashi, Thomas, and Wollan~\cite{KTW2020}.
\begin{theorem}[Robertson, Seymour, and Thomas~\cite{RST1994}]
    \label{thm:wall}
    There exists a function~${f_{\ref{thm:wall}} \colon \mathbb{N} \to \mathbb{N}}$ 
    such that 
    if~${g \geq 3}$ is an integer and~$\cT$ is a tangle in a graph~$G$ of order at least~${f_{\ref{thm:wall}}(g)}$,
    then~$\cT$ dominates a ${(g, g)}$-wall~${W}$ in~${G}$. 
\end{theorem}

We will also need the following lemma, stating that 
if a tangle dominates a wall~$W$, 
then it also dominates every $N^W$-anchored subwall of~$W$.

\begin{lemma}[Gollin et al.~{\cite[Lemma 2.8]{GollinHKKO2021}}]
    \label{lem:dominatedsubwall}
    Let~${w \geq t \geq 3}$ be integers, let~$W$ be a wall of order~$w$, and let~$\mathcal{T}$ be a tangle dominating~$W$. 
    If~$W'$ is a subwall of~$W$ of order~$t$ and 
    \[
        {\abs{N^{W'} \cap N^W} > (2t-1)(t-1)},
    \] 
    then~$\mathcal{T}$ dominates~$W'$. 
    
    In particular, if $W'$ is $N^W$-anchored, then~$\mathcal{T}$ dominates~$W'$.
\end{lemma}

We now review packing functions, introduced by Gollin et al.~\cite{GollinHKKO2021}.
Let~$G$ be a graph and 
let~$\nu$ be a function from the set of subgraphs of~$G$ to the set of non-negative integers. 
For subgraphs~${H, H' \subseteq G}$, we say 
\begin{itemize}
    \item $\nu$ is \emph{monotone} if~$\nu(H) \leq \nu(H')$ whenever~$H$ is a subgraph of~$H'$, 
    \item $\nu$ is \emph{additive} if~$\nu(H \cup H') = \nu(H) + \nu(H')$ whenever~$H$ and~$H'$ are vertex-disjoint, and 
    \item $\nu$ is a \emph{packing function for~$G$} if it is monotone and additive.
\end{itemize}
Now let~${\nu}$ be a packing function for a graph~$G$.  
For a subgraph~${H \subseteq G}$, we say a set~${T \subseteq V(H)}$ is a \emph{$\nu$-hitting set for~$H$} if~${\nu(H - T) = 0}$. 
We define~$\tau_\nu(H)$ as the size of a smallest $\nu$-hitting set of~$H$. 
Note that in the traditional sense of the word, a $\nu$-hitting set of~$G$ is a hitting set for the minimal subgraphs~${H \subseteq G}$ for which~${\nu(H) \geq 1}$. 

The following lemma, which will be useful for the inductive step of our proof, says that if a graph~$G$ itself does not have a small $\nu$-hitting set but all subgraphs of $G$ with a smaller $\nu$-value do, then for any given minimum $\nu$-hitting set $T$, we can find a tangle of large order that is `oriented towards'~$T$. Similar arguments have been used in many Erd\H{o}s-P\'osa type results,  see~\cite{HuynhJW2017} and~\cite{Wollan2011} for instance.
\begin{lemma}[Gollin et al.~{\cite[Lemma 4.1]{GollinHKKO2021}}]
    \label{lem:welllinked}
    Let~$\nu$ be a packing function for a graph~$G$ 
    and let~${T \subseteq V(G)}$ be a minimum $\nu$-hitting set for~$G$ of size~${t}$. 
    Let~$\mathcal{T}_T$ be the set of all separations~${(A,B)}$ of~$G$
    of order less than~${t/6}$ such that~${\abs{B \cap T} > 5t/6}$. 
    If~${\tau_\nu(H) \leq t/12}$ whenever~$H$ is a subgraph of~$G$ with~${\nu(H) < \nu(G)}$, 
    then~$\mathcal{T}_T$ is a tangle of order~${\lceil t/6 \rceil}$. 
\end{lemma}

\subsection{Cleaning the wall and finding allowable cycles}
\label{subsec:cleaning}

Let ${\Gamma = \prod_{j \in [m]} \Gamma_j}$ be a product of~$m$ abelian groups 
and let~$(G,\gamma)$ be a $\Gamma$-labelled graph. 
Let~$W$ be a wall in~$G$ and let $j\in[m]$. If $W$ is sufficiently large in terms of $\ell$ and no $(\ell,\ell)$-subwall of~$W$ is $\gamma_j$-bipartite, then there are many $\gamma_j$-non-zero cycles that we can use to help build our allowable cycles. Otherwise, there is an $(\ell,\ell)$-subwall of~$W$ that is $\gamma_j$-bipartite, and we can restrict to this subwall to utilize the $\gamma_j$-bipartiteness.
This leads to the following definition.
Given a subset~${Z \subseteq [m]}$ and an integer~$\ell$, 
we say that a wall~$W$ in~$G$ is \emph{$(\gamma,Z,\ell)$-clean} if 

\begin{enumerate}
    [label=(\arabic*)]
    \item\label{item:clean1} every $N^W$-path in~$W$ is $\gamma_j$-zero for all~${j \in Z}$ and
    \item\label{item:clean2} $W$ has no ${(\ell,\ell)}$-subwall that is $\gamma_j$-bipartite for all~${j \in [m] \setminus Z}$. 
\end{enumerate}
We write $\mathbb{N}_{\geq 3}$ to denote the set of integers greater than or equal to $3$.

The following lemma shows how to obtain a clean subwall from a wall.
\begin{lemma}[Gollin et al.~{\cite[Lemma 5.1]{GollinHKKO2021}}]
    \label{lem:cleansubwall}
    Let~${\Gamma = \prod_{j \in [m]} \Gamma_j}$ be a product of~$m$ abelian groups, 
    let~${(G,\gamma)}$ be a $\Gamma$-labelled graph, 
    let~${\psi \colon \{0\} \cup [m+1] \to \mathbb{N}_{\geq 3}}$ be a function, 
    and let~$W$ be a wall of order $\psi(0)+2$ in~$G$. 
    Then there exist a $\Gamma$-labelling~$\gamma'$ of~$G$ shifting-equivalent to~$\gamma$, 
    a subset~$Z$ of~${[m]}$, 
    and a ${(\gamma',Z,\psi(\abs{Z}+1)+2)}$-clean $\branch(W)$-anchored ${(\psi(\abs{Z}),\psi(\abs{Z}))}$-subwall of~$W$.
\end{lemma}

A clean wall can help us to build allowable cycles, as the following lemma demonstrates. 
Roughly speaking, for a $(\gamma, Z, \ell)$-clean wall~$W$,
if $j\notin Z$, then 
we will use a cycle whose $\gamma_j$-value is non-zero,  
which can be found in every $(\ell,\ell)$-subwall of~$W$.
If $j\in Z$, then we will use handlebars to adjust the $\gamma_j$-value
of a cycle.
This process allows us to find a cycle whose $\gamma_j$-value is not in $\Omega_j$ for any~${j \in [m]}$.
\begin{lemma}[Gollin et al.~{\cite[Lemma 8.1]{GollinHKKO2021}}]
    \label{lem:omega-avoiding-cycle}
    There exist functions ${c_{\ref{lem:omega-avoiding-cycle}}, r_{\ref{lem:omega-avoiding-cycle}} \colon \mathbb{N}^4 \to \mathbb{N}}$ satisfying the following. 
    Let~$t$,~$\ell$,~$m$, and~$\omega$ be positive integers with~${\ell \geq 3}$, 
    let~${\Gamma = \prod_{j \in [m]} \Gamma_j}$ be a product of~$m$ abelian groups, 
    and for each~${j \in [m]}$, let~$\Omega_j$ be a subset of~$\Gamma_j$ of size at most~$\omega$. 
    Let~${(G,\gamma)}$ be a $\Gamma$\nobreakdash-labelled graph,
    let~$Z$ be a subset of~${[m]}$, and 
    for integers ${c \geq c_{\ref{lem:omega-avoiding-cycle}}(t, \ell, m, \omega)}$ 
    and~${r \geq r_{\ref{lem:omega-avoiding-cycle}}(t, \ell, m, \omega)}$,
    let~$W$ be a ${(\gamma,Z,\ell)}$-clean $(c,r)$-wall in~$G$. 
    Then for every set~$\mathcal{P}$ of at most~$t$ pairwise vertex-disjoint $W$\nobreakdash-handles such that ${\gamma_j\left( \bigcup \mathcal{P} \right) \notin \Omega_j}$ for all~${j \in Z}$,
    there is a cycle~$O$ in~${W \cup \bigcup \mathcal{P}}$ with~${\gamma_j(O) \notin \Omega_j}$ for all~${j \in [m]}$. 
\end{lemma}

\subsection{Collecting handles}
\label{subsec:collectinghandles}

Wollan~\cite{Wollan2010} proved the following Erd\H{o}s-P\'osa  type result for $\gamma$-non-zero $A$-paths in groups-labelled graphs $(G, \gamma)$. 

\begin{theorem}[Wollan~\cite{Wollan2010}]
    \label{thm:tpath}
    Let~$k$ be a positive integer, let~$\Gamma$ be an abelian group, let~$(G,\gamma)$ be a $\Gamma$-labelled graph, and let~${A \subseteq V(G)}$. 
    Then~$G$ contains $k$ vertex-disjoint $\gamma$-non-zero $A$-paths or a vertex set of size at most~${f_{\ref{thm:tpath}}(k) := 50k^4}$ hitting all $\gamma$-non-zero $A$-paths. 
\end{theorem}
In this paper we do not use this theorem directly, but instead apply the following technical corollary of it proved in the previous paper of Gollin et al.~{\cite{GollinHKKO2021}}.

\begin{lemma}[Gollin et al.~{\cite[Lemma 4.3]{GollinHKKO2021}}]
    \label{lem:cover}
    Let~$u$, $k$ be positive integers such that~${f_{\ref{thm:tpath}}(k) < u - 2}$.
    Let~$\Lambda$ be an abelian group, 
    let~${(G, \lambda)}$ be a $\Lambda$-labelled graph, and let~$\nu$ be a packing function for~$G$ such that
    \begin{itemize}
        \item every minimal subgraph~${H}$ of~$G$ with~${\nu(H) \geq 1}$ is a $\lambda$-non-zero cycle,
        \item ${\tau_\nu(H) \leq 3 u}$ for every subgraph~${H}$ of $G$ with~${\nu(H) < \nu(G)}$, and 
        \item ${\tau_\nu(G) \geq u}$.
    \end{itemize} 
    Let~${T \subseteq V(G)}$ be a minimum $\nu$-hitting set for~$G$ and let~${N \subseteq V(G)}$ such that for every~${S \subseteq V(G)}$ of size less than~${u}$, there is a component of~${G-S}$ containing a vertex of~$N$ and at least~$4u$ vertices of~$T$. 
    Then~$G$ contains~$k$ vertex-disjoint $\lambda$-non-zero $N$-paths.
\end{lemma}

We will apply this lemma to a $\Gamma$-labelled graph $(G,\gamma)$, with $\nu(H)$ equal to the number of vertex-disjoint cycles in $(H,\gamma)$ with allowable $\gamma$-values and $N$ being the set of degree $3$ vertices of a carefully chosen wall in $G$.
However instead of taking $\lambda$ to be the original $\Gamma$-labelling, we instead consider some quotient group $\Lambda$ (specifically, $\Gamma/\Lambda$ in the proof) and the corresponding labelling $\lambda$, with the property that every cycle with an allowable value is $\lambda$-non-zero. 
By iterating this, we can avoid the scenario where the $\gamma$-values of all of the paths we obtain are in some subgroup which does not include any allowable value.

The other key ingredient for this process is the following lemma which enables us to use the output of Lemma~\ref{lem:cover} to extend a set of ($W$-)handles, at the cost of shrinking the wall which the handles attach to.

\begin{lemma}[Gollin et al.~{\cite[Lemma 6.1]{GollinHKKO2021}}]
    \label{lem:addlinkage}
    There exist functions~${w_{\ref{lem:addlinkage}}\colon \mathbb{N}^2 \to \mathbb{N}}$ and~$f_{\ref{lem:addlinkage}} \colon \mathbb{N} \to \mathbb{N}$ satisfying the following. 
    Let~$k$, $t$, and~$c$ be positive integers with~${c \geq 3}$, let~${\Gamma}$ be an abelian group, and let~${(G, \gamma)}$ be a $\Gamma$-labelled graph.
    Let~$W$ be a wall in~$G$ of order at least~${w_{\ref{lem:addlinkage}}(k,c)}$ such that all $\branch(W)$-paths of~$W$ are $\gamma$-zero. 
    For each~${i \in [t-1]}$, let~$\mathcal{P}_i$ be a set of~$4k$ $W$-handles in~$G$ such that the paths in~${\bigcup_{i \in [t-1]} \mathcal{P}_i}$
    are pairwise vertex-disjoint. 
    If~$G$ contains at least~${f_{\ref{lem:addlinkage}}(k)}$ vertex-disjoint $\gamma$-non-zero ${\branch(W)}$-paths, 
    then there exist a $c$-column-slice~$W'$ of~$W$ and a set~$\mathcal{Q}_i$ of~$k$ pairwise vertex-disjoint $W'$-handles for each~${i \in [t]}$ such that 
    \begin{enumerate}
        [label=(\roman*)]
        \item for each~${i \in [t-1]}$, the set~$\mathcal{Q}_i$ is a subset of the row-extension of~$\mathcal{P}_i$ to~$W'$ in~$W$, 
        \item the paths in~${\bigcup_{i \in [t]} \mathcal{Q}_i}$ are pairwise vertex-disjoint, and
        \item the paths in~$\mathcal{Q}_t$ are $\gamma$-non-zero. 
    \end{enumerate}
\end{lemma}

\section{Handling handlebars}\label{sec:handlebars}

We remind the readers that two sets~$\mathcal{G}_1$ and~$\mathcal{G}_2$ of graphs are said to be vertex-disjoint if~${\bigcup \mathcal{G}_1}$ and~${\bigcup \mathcal{G}_2}$ are vertex-disjoint. 
First, we show that given a family of pairwise vertex-disjoint sets of $W$-handles, we can throw away some $W$-handles from each set to obtain a family of pairwise vertex-disjoint non-mixing $W$-handlebars. 

\begin{lemma}
    \label{lem:handlebars}
    There is a function ${f_{\ref{lem:handlebars}} \colon \mathbb{N}^2 \to \mathbb{N}}$ satisfying the following property. 
    Let $t$, $\theta$, $c$, and~$r$ be positive integers with~${r \geq 3}$ and~${c \geq 3}$, let~$W$ be a ${(c,r)}$-wall, and 
    let~${(\mathcal{P}_i \colon i \in [t])}$ be a family of pairwise vertex-disjoint sets of~${f_{\ref{lem:handlebars}}(t,\theta)}$ $W$-handles. 
    If the $W$-handles in $\bigcup_{i=1}^t \mathcal{P}_i$ are pairwise vertex-disjoint, 
    then there exists a family~${( \mathcal{P}^\ast_i \colon i \in [t] )}$ of pairwise non-mixing $W$-handlebars such that~${\mathcal{P}^\ast_i \subseteq \mathcal{P}_i}$ and~${\abs{\mathcal{P}^\ast_i} \geq \theta}$ for all~${i \in [t]}$. 
\end{lemma}

\begin{proof}
    Let 
    \[ 
        f_{\ref{lem:handlebars}}(t,\theta) :=
            \begin{cases}
                 \max\{3((2t-1)\theta-1)^3+1, 30f_{\ref{lem:handlebars}}(t-1,\theta)\}
                 &\text{if $t>1$,}\\
                 3(\theta-1)^3+1 & \text{if $t=1$.}
            \end{cases}
    \] 
    We proceed by induction on~$t$. 
    If~${t = 1}$, then there is a subset ${\mathcal{P}' \subseteq \mathcal{P}_1}$ of size~$(\theta-1)^3+1$ whose paths all have the same number of endvertices in~$C^W_1$. 
    The result then follows from Lemma~\ref{lem:pureset}. 
    
    Suppose~${t \geq 2}$. 
    By the above argument, there is a $W$-handlebar~${\mathcal{P}' = \{P'_j \colon j \in [(2t-1)\theta] \} \subseteq \mathcal{P}_t}$ of size~${(2t-1)\theta}$. 
    For each~${j \in [(2t-1)\theta]}$, let~$v_j$ and~$w_j$ be the endvertices of~$P'_j$ with~${v_j \prec_W w_j}$. 
    Without loss of generality, we may assume that for all~${j,j' \in [(2t-1)\theta]}$ with~${j < j'}$, we have that~${v_j \prec_W v_{j'}}$. 
    For each~${x \in [2t-1]}$, let~$A_x$ be the subpath of~${C^W_1 \cup C^W_c}$ from~$v_{1+(x-1)\theta}$ to~$v_{x\theta}$, and let~$B_x$ be the subpath of~${C^W_1 \cup C^W_c}$ from~$w_{1+(x-1)\theta}$ to~$w_{x\theta}$. 
    Note that for distinct~$x$ and~$y$ in~${[2t-1]}$, we have that~${A_x \cup B_x}$ and~${A_y \cup B_y}$ are vertex-disjoint. 
    Hence, for each~${i \in [t-1]}$, there are at most two integers~${x \in [2t-1]}$ such that~${A_x \cup B_x}$ contains more than a third of the endvertices of paths in~$\mathcal{P}_i$. 
    Hence, there exists~${x \in [2t-1]}$ such that~${A_x \cup B_x}$ contains at most a third of the endvertices of paths in~$\mathcal{P}_i$ for all~${i \in [t-1]}$.
    For every~${i \in [t-1]}$, let~${\mathcal{P}'_i \subseteq \mathcal{P}_i}$ of size at least~${10f_{\ref{lem:handlebars}}(t-1,\theta)}$ such that~${A_x \cup B_x}$ and~${\bigcup \mathcal{P}'_i}$ are vertex-disjoint. 
    
    Since each of $A_x$ and $B_x$ is a subpath of $C^W_1\cup C^W_c$, the graph ${H := C^W_1 \cup C^W_c - V(A_x \cup B_x)}$ has at most four components.
    Let $\{H_j:j\in [q]\}$ be the set of components of $H$, and for every path $P$ in~${\bigcup \mathcal{P}'_i}$, let $I(P)$ be the set of integers $i\in [q]$ where $H_i$ contains an endvertex of $P$. Then $I(P)$ is a set in $\{\{i\}:i\in [q] \}\cup \{\{i,j\}:i,j\in [q], i\neq j\}$, which consists of at most~${10 = \binom{4}{1}+\binom{4}{2}}$ sets.

    As~${\mathcal{P}'_i}$ has size at least~${10f_{\ref{lem:handlebars}}(t-1,\theta)}$, there is a subset~${\mathcal{P}''_i \subseteq \mathcal{P}'_i}$ of size  $f_{\ref{lem:handlebars}}(t-1,\theta)$ such that for every pair of $W$\nobreakdash-handles~${P,P' \in \mathcal{P}''_i}$, each component of~$H$ contains the same number of endvertices of~$P$ and~$P'$. 
    By the induction hypothesis, there is a 
    family~${(\mathcal{P}^\ast_i \colon i \in [t-1])}$
    of pairwise non-mixing $W$-handlebars such that ${\mathcal{P}^\ast_i \subseteq \mathcal{P}''_i}$ and ${\abs{\mathcal{P}^\ast_i} \geq \theta}$ for each $i\in [t-1]$. 
    Together with~${\mathcal{P}^\ast_t := \{ P'_j \colon j \in [x\theta] \setminus [(x-1)\theta] \}}$, these $W$-handlebars satisfy the lemma. 
\end{proof}

The paths of a $W$-handlebar~$\mathcal{P}$ can be pieced together through the outer columns of~$W$ to form a $W^\ast$-handlebar $\mathcal{P}^\ast$ for some column-slice~$W^\ast$ of~$W$ such that each path in~$\mathcal{P}^\ast$ contains exactly~$d$ paths of~$\mathcal{P}$ for any desired~$d$, provided that~$\mathcal{P}$ and~$W$ are large enough. 
The following lemma shows that this can be done simultaneously to a family of pairwise vertex-disjoint non-mixing $W$-handlebars so that the resulting family of $W^\ast$-handlebars is also pairwise vertex-disjoint and non-mixing. See Figure~\ref{fig:lemma-combining-handles} for an illustration.

\begin{figure}[tbp]
  \centering
        \begin{tikzpicture}[scale=0.3]
            \tikzset{vx/.style = {circle, draw, fill=black!0, inner sep=0pt, minimum width=4pt}}
            \tikzset{vxsub/.style = {circle, draw, fill=black!50, inner sep=0pt, minimum width=2pt}}
            
            \tikzset{c1/.style={green, line width=6pt,opacity=0.8,rounded corners}}
            \tikzset{c2/.style={brown, line width=6pt,opacity=0.8,rounded corners}}
             \tikzset{c3/.style={cyan, line width=6pt,opacity=0.8,rounded corners}}
            \tikzset{c4/.style={purple, line width=6pt,opacity=0.8,rounded corners}}
            \tikzset{vx1/.style = {circle, draw, solid, black, fill=red!30, inner sep=0pt, minimum width=4pt}}
            \tikzset{vx2/.style = {circle, draw, solid, black, fill=blue!30, inner sep=0pt, minimum width=4pt}}
            \tikzset{vx3/.style = {circle, black, solid, fill=orange, draw,
            inner sep=0pt, minimum width=4pt}}

           \draw [c1] (0,24) to (2,24);
           \draw [c1] (29,14) to (28,14) to (28,13) to (29,13);
           \draw [c1] (0,23) to (1,23) to (1,22) to (0,22);
           \draw [c1] (29,12) to (27,12);

           \draw [c2] (0,21) to (2,21);
           \draw [c2] (29,11) to (28,11) to (28,10) to (29,10);
           \draw [c2] (0,20) to (1,20) to (1,19) to (0,19);
           \draw [c2] (29,9) to (27,9);

           \draw [c3] (0,16) to (2,16);
           \draw [c3] (0,15) to (1,15) to (1,14) to (0,14);
           \draw [c3] (0,13) to (1,13) to (1,12) to (0,12);
           \draw [c3] (0,11) to (1,11) to (1,10) to (0,10);
           \draw [c3] (0,9) to (2,9);

           \draw [c4] (0,8) to (2,8);
           \draw [c4] (0,7) to (1,7) to (1,6) to (0,6);
           \draw [c4] (0,5) to (1,5) to (1,4) to (0,4);
           \draw [c4] (0,3) to (1,3) to (1,2) to (0,2);
           \draw [c4] (0,1) to (2,1);
            
             \tikzwall{13}{25}{2}{0}{vx}{red, very thick}
            \tikzwall{15}{25}{0}{0}{vx}{gray}
            \foreach \i in {1,2,3,4,5,6} {
                \pgfmathtruncatemacro{\yone}{9+\i}
                \pgfmathtruncatemacro{\ytwo}{25-\i}
                \pgfmathtruncatemacro{\iminus}{15-\i}
                \draw [blue, thick,dotted,rounded corners=5pt] (v0-\ytwo)
                to ++(-1-0.5*\i,0) to ++(0,1.5*\i) 
                to ++(8+2*\i,0) 
                to ++(23-\i,0) to ++(0,-0.5*\i-10-\i)                
                to (v29-\iminus);
            }   
            \foreach \i in {1,2,3,4,5,6,7,8}{
                \pgfmathtruncatemacro{\seriesone}{2*\i}
                \pgfmathtruncatemacro{\seriestwo}{2*\i-1}
                \draw [red, thick,dotted, out=180,in=-180] (v0-\seriesone) 
                to ++(-3,-.5) 
                to (v0-\seriestwo) ;
            }
        \end{tikzpicture}
    \caption{Combining paths in Lemma~\ref{lem:combining-handles} where $t=2$, $\theta=2$, $d_1=3$, and $d_2=4$. The red subwall $W^*$ is a $(c-2)$-column slice of a given $(c,r)$-wall $W$, and the paths in $W$-handlebars $\mathcal{P}_1$ and $\mathcal{P}_2$ are shown in blue and red, respectively.  Following the marked paths in $W$, we construct $W^*$-handlebars $\mathcal{P}_1^*$ and $\mathcal{P}_2^*$ of size $\theta$. As $d_1$ is odd, $\mathcal{P}_1^*$ is of the same type as $\mathcal{P}_1$.  
    }
    \label{fig:lemma-combining-handles}
\end{figure}

\begin{lemma}
    \label{lem:combining-handles} 
    Let~$t$, $c$, $r$, and $\theta$ be positive integers with ${c \geq 5}$ and ${r \geq 3}$, and let $d_i$ be a positive integer for each~${i \in [t]}$. 
    Let~$W$ be a ${(c,r)}$-wall in a graph~$G$ and let $W^\ast$ be a $(c-2)$-column-slice of~$W$ containing~$C^W_2$ and~$C^W_{c-1}$. 
    Let~${(\mathcal{P}_i \colon i \in [t])}$ be a family of pairwise vertex-disjoint non-mixing $W$-handlebars with~${\abs{\mathcal{P}_i} \geq \theta d_i}$ 
    for all~${i \in [t]}$. 
    Then
    there is a family~${(\mathcal{P}^\ast_i \colon i \in [t])}$ of pairwise vertex-disjoint non-mixing $W^\ast$-handlebars each of size~$\theta$ such that
    for each~${i \in [t]}$ and~${Q \in \mathcal{P}^\ast_i}$, there is a set~${\{ P_{j,Q} \in \mathcal{P}_i \colon j \in [d_i] \}}$ of size~$d_i$ such that 
    \[
        \bigcup_{j=1}^{d_i} P_{j,Q} \subseteq Q \subseteq W \cup \bigcup_{j=1}^{d_i} P_{j,Q} .
    \]
    Moreover, for each~${i \in [t]}$, if~$d_i$ is even, then~$\mathcal{P}^\ast_i$ is in series and if~$d_i$ is odd, 
    then~$\mathcal{P}^\ast_i$ is of the same type as~$\mathcal{P}_i$. 
\end{lemma}

\begin{proof}
    For each~${i \in [t]}$, let~${\mathcal{P}_i =: \{ P_{i,x} \colon x \in [\theta d_i] \}}$ such that if~${x, y \in [\theta d_i]}$ with~${x < y}$, then some endvertex of~$P_{i,x}$ is $\prec_W$-smaller than both endvertices of~$P_{i,y}$. 
    For each~${i \in [t]}$ and~${y \in [\theta]}$, it is easy to verify that there is a unique path in~${C_1^W \cup C_c^W \cup \bigcup_{x=(y-1)d_i+1}^{yd_i} P_{i,x}}$ that contains~${\bigcup_{x=(y-1)d_i+1}^{yd_i} P_{i,x}}$ whose set of endvertices contains the $\prec_W$-smallest endvertex of $P_{i,(y-1)d_i+1}$ and some endvertex of~$P_{i,y d_i}$. 
    Let~$Q_{i,y}$ denote the row-extension of this path to~$W^\ast$. 
    Now with~$\mathcal{P}_i^\ast := \{ Q_{i, y} \colon y \in [\theta] \}$, we easily observe that~${(\mathcal{P}^\ast_i \colon i \in [t])}$ is as desired. 
\end{proof}

Next we show that if~${(\mathcal{P}_i \colon i \in [t])}$ is a family of pairwise vertex-disjoint non-mixing $W$-handlebars 
none of which is in series, then we can construct a $W'$-handlebar for some subwall~$W'$ of~$W$ such that each $W'$-handle contains exactly one path from each~$\mathcal{P}_i$. 

\begin{figure}[tbp]
  \centering
        \begin{tikzpicture}[scale=0.3]
            \tikzset{vx/.style = {circle, draw, fill=black!0, inner sep=0pt, minimum width=4pt}}
            \tikzset{vxsub/.style = {circle, draw, fill=black!50, inner sep=0pt, minimum width=2pt}}
            
            \tikzset{c1/.style={purple, line width=6pt,opacity=0.8,rounded corners}}
            \tikzset{c2/.style={brown, line width=6pt,opacity=0.8,rounded corners}}
             \tikzset{c3/.style={cyan, line width=6pt,opacity=0.8,rounded corners}}
            \tikzset{vx1/.style = {circle, draw, solid, black, fill=red!30, inner sep=0pt, minimum width=4pt}}
            \tikzset{vx2/.style = {circle, draw, solid, black, fill=blue!30, inner sep=0pt, minimum width=4pt}}
            \tikzset{vx3/.style = {circle, black, solid, fill=orange, draw,
            inner sep=0pt, minimum width=4pt}}

           \draw [c1] (0,16) to (0,15) to (1,15) to (1,14) to (0,14) to (0,13);
           \draw [c2] (0,17) to (3,17) to (3,16) to (2,16) to (2,15) to (3,15) to (3,14) to (2,14) to (2,13) to (3,13) to (3,12) to (0,12);
            \draw [c3] (0,18) to (4,18) to (4,17) to (5,17) to (5,16) to (4,16) to (4,15) to (5,15) to (5,14) to (4,14) to (4,13) to (5,13) to (5,12) to (4,12) to (4,11) to (0,11);

            \draw [c1] (29,4) to (29,5) to (28,5) to (28,6) to (29,6);
            \draw [c2] (29,3) to (26,3) to (26,4) to (27,4) to (27,5) to (26,5) to (26,6) to (27,6) to (27,7) to (29,7);
            \draw [c3] (29,2) to (25,2) to (25,3) to (24,3) to (24,4) to (25,4) to (25,5) to (24,5) to (24,6) to (25,6) to (25, 7) to (24,7) to (24,8) to (29,8);

            \draw[c3] (29,16) to (23,16);
            \draw[c2] (29,15) to (23,15);
            \draw[c1] (29,14) to (23,14);

            \draw[c3] (0,3) to (6,3);
            \draw[c2] (0,4) to (6,4);
            \draw[c1] (0,5) to (6,5);

           \tikzwall{9}{19}{6}{0}{vx}{red, very thick}
            \tikzwall{15}{19}{0}{0}{vx}{gray}
            \foreach \i in {1,2,3} {
                \pgfmathtruncatemacro{\yone}{9+\i}
                \pgfmathtruncatemacro{\ytwo}{19-\i}
                \pgfmathtruncatemacro{\iminus}{9-\i}
                \draw [blue, thick,dotted,rounded corners=5pt] (v0-\ytwo)
                to ++(-1-0.5*\i,0) to ++(0,1.5*\i) 
                to ++(8+2*\i,0) 
                to ++(23-\i,0) to ++(0,-0.5*\i-10-\i)                
                to (v29-\iminus);
            }   
            \foreach \i in {1,2,3}{
                \pgfmathtruncatemacro{\seriesone}{\i+2}
                \pgfmathtruncatemacro{\seriestwo}{10+\i}
                \draw [red, thick,dotted,rounded corners=5pt] (v0-\seriesone) 
                to ++(-4.5+0.5*\i,0) to  ++(0,8) 
                to (v0-\seriestwo) ;
            }
             \foreach \i in {1,2,3}{
                \pgfmathtruncatemacro{\seriesone}{\i+1}
                \pgfmathtruncatemacro{\seriestwo}{17-\i}
                \draw [purple, thick,dotted,rounded corners=5pt] (v29-\seriesone) 
                to ++(5-0.5*\i,0) to  ++(0,16-2*\i) 
                to (v29-\seriestwo) ;
            }
        \end{tikzpicture}
    \caption{Combining handlebars in Lemma~\ref{lem:combining-handlebars} where $k=t=3$. In this example, we can recursively find a path in the first column or the last column connecting two handlebars. As the number of crossing $W$-handlebars is odd, the resulting $W'$-handlebar is crossing.
    }
    \label{fig:lemma-combining-handlebars-1}
\end{figure}

\begin{lemma}
    \label{lem:combining-handlebars}
    Let~$t$, $k$, $c$, and~$r$ be positive integers 
    with~${k \geq 2}$ and~${c, r \geq 3}$. 
    Let~$W$ be a ${(c',r')}$-wall in a graph~$G$
    with ${c' \geq c_{\ref{lem:combining-handlebars}}(t,k,c) := c + kt}$ and~${r' \geq r_{\ref{lem:combining-handlebars}}(k, r) := r + k }$. 
    For each $i\in[t]$, let $\mathcal{P}_i$ be a $W$-handlebar of size~$k$ in~$G$, not in series, such that $\mathcal{P}_i$ and  $\mathcal{P}_j$ are vertex-disjoint and non-mixing for all $j\in [t]\setminus\{i\}$.
    Then there exist an $N^W$-anchored subwall~$W'$ of~$W$ having at least~$c$ columns and at least~$r$ rows
    and a $W'$-handlebar~$\mathcal{Q}$ in~$G$ of size~$k$ such that for each~${Q \in \mathcal{Q}}$, 
    there is a set~${\{ P_{i,Q} \in \mathcal{P}_i \colon i \in [t] \}}$ 
    such that
    \[
        {\bigcup_{i=1}^t P_{i,Q}  \subseteq Q \subseteq W \cup \bigcup_
        {i=1}^t P_{i,Q} }.
    \]
    Moreover, $\mathcal{Q}$ is crossing if and only if the number of crossing $W$-handlebars in~${(\mathcal{P}_i \colon i \in [t])}$ is odd. 
\end{lemma}

\begin{proof}
    We proceed by induction on~$t$. 
    This lemma is trivial if~${t = 1}$ and therefore we may assume that~${t > 1}$.
    First, suppose that for some distinct~${j', j'' \in [t]}$,
    there is a path~$Q$ in~${C_1^W \cup C_{c'}^W}$ that contains exactly one endvertex of each path in~${\mathcal{P}_{j'} \cup \mathcal{P}_{j''}}$ and no endvertex of any path in ${\bigcup \{ \mathcal{P}_x \colon x \in [t] \setminus \{j',j''\} \}}$.  We illustrate this case in Figure~\ref{fig:lemma-combining-handlebars-1}.
    Without loss of generality, we may assume that~${j' = t-1}$, that~${j'' = t}$, and that~${Q \subseteq C^W_1}$. 
    Let~${( a_j \colon j \in [2k] )}$ be a strictly increasing sequence of integers in~${[r']}$ such that~${R_{a_j}^W \cap Q}$ contains an endvertex of a path in~${\mathcal{P}_{t-1} \cup \mathcal{P}_t}$ for all~${j \in [2k]}$. 
    For~${j \in [k]}$, let~$Q_j$ be a subpath of~$C^W_{k+1-j}$ from a vertex in~$R^W_{a_j}$ to a vertex in~$R^W_{a_{2k+1-j}}$.
    Then it is easy to observe that for each~${j \in [k]}$,
    there is a unique path in
    \[
        {\bigcup \mathcal{P}_{t-1}  \cup R_{a_j}^W 
        \cup Q_j
        \cup R_{a_{2k+1-j}}^W 
        \cup \bigcup\mathcal{P}_t}
    \] 
    that contains exactly one path in~$\mathcal{P}_{t-1}$ and exactly one path in~$\mathcal{P}_t$. 
    Let~$W^\ast$ be a $(c'-k)$-column-slice of~$W$ containing $C^W_{k+1}$ and~$C^W_{c'}$.
    Then the row-extensions of all of these paths to~$W^\ast$ yield 
    a $W^\ast$-handlebar~$\mathcal{P}'_{t-1}$ that is vertex-disjoint and non-mixing with the row-extension of~$\mathcal{P}_i$ to~$W^\ast$ for each~${i \in [t-2]}$. 
    Note that~$\mathcal{P}'_{t-1}$ is crossing if and only if exactly one of~$\mathcal{P}_{t-1}$ and~$\mathcal{P}_{t}$ is crossing. 
    By applying the induction hypothesis to~$\mathcal{P}'_{t-1}$ and row-extensions of~$\mathcal{P}_i$ to~$W^\ast$ for all~${i \in [t-2]}$, we deduce the lemma in this case. 

\begin{figure}[tbp]
  \centering
        \begin{tikzpicture}[scale=0.3]
            \tikzset{vx/.style = {circle, draw, fill=black!0, inner sep=0pt, minimum width=4pt}}
            \tikzset{vxsub/.style = {circle, draw, fill=black!50, inner sep=0pt, minimum width=2pt}}
            
            \tikzset{c1/.style={purple, line width=6pt,opacity=0.8,rounded corners}}
            \tikzset{c2/.style={brown, line width=6pt,opacity=0.8,rounded corners}}
             \tikzset{c3/.style={cyan, line width=6pt,opacity=0.8,rounded corners}}
            \tikzset{vx1/.style = {circle, draw, solid, black, fill=red!30, inner sep=0pt, minimum width=4pt}}
            \tikzset{vx2/.style = {circle, draw, solid, black, fill=blue!30, inner sep=0pt, minimum width=4pt}}
            \tikzset{vx3/.style = {circle, black, solid, fill=orange, draw,
            inner sep=0pt, minimum width=4pt}}

           \draw [c1] (29,16) to (29,17) to (28,17) to (28,18) to (0,18) to (0,17) to (1,17) to (1,16) to (0,16) to (0,15) to (1,15) to (1,14) to (0,14) to (0,13);
           \draw [c2] (29,15) to (26,15) to (26,16) to (27, 16) to (27,17) to (3,17) to (3,16) to (2,16) to (2,15) to (3,15) to (3,14) to (2,14) to (2,13) to (3,13) to (3,12) to (0,12);
            \draw [c3] (29,14) to (25,14) to (25,15) to (24,15) to (24,16) to (4,16) to (4,15) to (5,15) to (5,14) to (4,14) to (4,13) to (5,13) to (5,12) to (4,12) to (4,11) to (0,11);

            \draw [c3] (29,4) to (23,4);
            \draw [c2] (29,3) to (23,3);
            \draw [c1] (29,2) to (23,2);

            \draw[c3] (0,3) to (6,3);
            \draw[c2] (0,4) to (6,4);
            \draw[c1] (0,5) to (6,5);

           \tikzwall{9}{16}{6}{0}{vx}{red, very thick}
            \tikzwall{15}{19}{0}{0}{vx}{gray}
            \foreach \i in {1,2,3}{
                \pgfmathtruncatemacro{\seriesone}{\i+2}
                \pgfmathtruncatemacro{\seriestwo}{10+\i}
                \draw [red, thick,dotted,rounded corners=5pt] (v0-\seriesone) 
                to ++(-4.5+0.5*\i,0) to  ++(0,8) 
                to (v0-\seriestwo) ;
            }
             \foreach \i in {1,2,3}{
                \pgfmathtruncatemacro{\seriesone}{\i+1}
                \pgfmathtruncatemacro{\seriestwo}{17-\i}
                \draw [purple, thick,dotted,rounded corners=5pt] (v29-\seriesone) 
                to ++(5-0.5*\i,0) to  ++(0,16-2*\i) 
                to (v29-\seriestwo) ;
            }
        \end{tikzpicture}
    \caption{Combining handlebars in Lemma~\ref{lem:combining-handlebars} where $k=3$, $t=2$, and there is no path in the first column or the last column connecting two handlebars. In this case, we use $k$ rows to combine two handlebars.
    }
    \label{fig:lemma-combining-handlebars-2}
\end{figure}
    
    Now suppose that there is no path~$Q$ as defined above for any pair of $W$-handlebars in~${(\mathcal{P}_i \colon i \in [t])}$.
    We illustrate this case in Figure~\ref{fig:lemma-combining-handlebars-2}.
    Since no $W$-handlebar in~$(\mathcal{P}_i \colon i \in [t])$ is in series, it follows that 
    each of~$C^W_1$ and~$C^W_{c'}$ meets at most one $W$-handlebar in~${(\mathcal{P}_i \colon i \in [t])}$
    and so~${t = 2}$. 
    Let~$W''$ be a ${(c'-2k)}$-column-slice of~$W$ containing~$C^W_{k+1}$ and~$C^W_{c'-k}$ and let~$W'$ be a ${(r'-k)}$-row-slice of~$W''$ containing~$R_{k+1}^{W''}$ and~$R_{r'}^{W''}$.     
    Without loss of generality, we may assume that the endvertices of~$\mathcal{P}_1$ are contained in~$C^W_1$ and the endvertices of~$\mathcal{P}_2$ are contained in~$C^W_{c'}$. 
    Let~${( a_j \colon j \in [k] )}$ be a strictly increasing sequence of integers in~${[r']}$ such that $R_{a_j}^W$ contains the endvertex of a path in~$\mathcal{P}_{1}$ that is $\prec_W$-smaller than its other endvertex for all~${j \in [k]}$, 
    and let~${( b_j \colon j \in [k] )}$ be a strictly increasing sequence of integers in~${[r']}$ such that~$R_{b_j}^W$ contains the endvertex of a path in~$\mathcal{P}_{2}$ that is~$\prec_W$-larger than its other endvertex for all~${j \in [k]}$. 
    Let~$W^0$ be the $k$-column-slice of~$W$ containing~$C^W_1$
    and let~$W^1$ be the $k$-column-slice of~$W$ containing~$C^W_{c'}$. 
    For~${j \in [k]}$, let~$P_j$ be a subpath of~$C^W_{j}$ from a vertex in~$R^W_{a_j}$ to a vertex in~$R^W_{j}$ 
    and let~$P_j'$ be a subpath of~$C^W_{c'+1-j}$ from a vertex in~$R^W_{j}$ to a vertex in~$R^W_{b_j}$. 
    Again, it is easy to observe that for each~${j \in [k]}$,
    there is a unique path in
    \[
        { \bigcup \mathcal{P}_{1} \cup R_{a_j}^{W^0} \cup P_j \cup R^W_{j}\cup P_j' \cup R_{b_j}^{W^1} \cup \bigcup\mathcal{P}_{2} }
    \] 
    that contains exactly one path in~$\mathcal{P}_{1}$ 
    and exactly one path in~$\mathcal{P}_{2}$. 
    Now the row-extensions of all of these paths to~$W'$ yield a $W'$-handlebar~$\mathcal{Q}$ as desired. 
    As before, note that~$\mathcal{Q}$ is crossing if and only if exactly one of~$\mathcal{P}_{1}$ and~$\mathcal{P}_{2}$ is crossing. 
    This completes the proof. 
\end{proof}

The final lemma of this section shows that if~${(\mathcal{P}_i \colon i \in [q])}$ is a family of pairwise vertex-disjoint non-mixing $W$-handlebars that does not satisfy any of the three properties of Definition~\ref{def:obstructions}\ref{item:obstructions-handlebars}, then 
we can find $k$ vertex-disjoint small subwalls, each equipped with $q$ vertex-disjoint handles such that for each $i\in[q]$, the $i$-th handle is a $W$-handle in~${\mathcal{P}_i}$ extended with paths in~$W$.
Moreover, the handles on each subwall are vertex-disjoint from every other subwall and their handles.
Later, we will find an allowable cycle in each subwall with its handles 
by applying Lemma~\ref{lem:omega-avoiding-cycle}, 
and therefore we will obtain $k$ pairwise vertex-disjoint allowable cycles. We illustrate the cases in Lemma~\ref{lem:almost-finding-cycles} in Figure~\ref{fig:lemma-almost-finding-cycles}.

\begin{figure}[t]
    \begin{subfigure}{0.45\linewidth}
    \centering
        \begin{tikzpicture}[scale=0.3]
            \tikzset{vx/.style = {circle, draw, fill=black!0, inner sep=0pt, minimum width=4pt}}
            \tikzset{vxsub/.style = {circle, draw, fill=black!50, inner sep=0pt, minimum width=2pt}}
            
            \tikzset{c1/.style={black, line width=1.5pt,opacity=0.8,rounded corners}}
            \tikzset{vx1/.style = {circle, draw, solid, black, fill=red!30, inner sep=0pt, minimum width=4pt}}
            \tikzset{vx2/.style = {circle, draw, solid, black, fill=blue!30, inner sep=0pt, minimum width=4pt}}
            \tikzset{vx3/.style = {circle, black, solid, fill=orange, draw,
            inner sep=0pt, minimum width=4pt}}

            \foreach \i in {1,3,5}{
                \pgfmathtruncatemacro{\seriesone}{\i}
                \pgfmathtruncatemacro{\seriestwo}{19-\i}
                \draw [red, thick,dotted,rounded corners=5pt] (v0-\seriesone) 
                to ++(-4.5+0.5*\i,0) to  ++(0,19-2*\i)  
                to (v0-\seriestwo) ;

            }
            \draw [thick, rounded corners=5pt] (0,1) -- (0,20) -- (18,20)--(18,0)--(0,0)--(0,1);

            \draw [red, thick, rounded corners=5pt] (8,1) -- (8,4.7) -- (14,4.7)--(14,0.3)--(14,0.3)--(8,0.3)--(8,1);
            \draw [red, thick, rounded corners=5pt] (8,6) -- (8,9.7) -- (14,9.7)--(14,5.3)--(14,5.3)--(8,5.3)--(8,6);
            \draw [red, thick, rounded corners=5pt] (8,11) -- (8,14.7) -- (14,14.7)--(14,10.3)--(14,10.3)--(8,10.3)--(8,11);

        \node[] at (11,2.5) {$W_1$};
        \node[] at (11,7.5) {$W_2$};
        \node[] at (11,12.5) {$W_3$};

         \draw [c1] (0,1) to (7.5,1) to (7.5,4) to (8,4);
         \draw [c1] (0,3) to (6.5,3) to (6.5,9) to (8,9);
         \draw [c1] (0,5) to (5.5,5) to (5.5,14) to (8,14);

         \draw [c1] (0,14) to (3.5,14) to (3.5, 16.2) to (15,16.2) to (15,14) to (14,14);
         \draw [c1] (0,16) to (2.5,16) to (2.5, 17.6) to (16,17.6) to (16,9) to (14,9);
         \draw [c1] (0,18) to (1.5,18) to (1.5, 19) to (17,19) to (17,4) to (14,4);

    \end{tikzpicture}
    \caption{The case when $q=1$ and $\mathcal{P}_1$ is nested.}
    \end{subfigure}
    \hfill    
    \begin{subfigure}{0.45\linewidth}
    \centering
           \begin{tikzpicture}[scale=0.3]
            \tikzset{vx/.style = {circle, draw, fill=black!0, inner sep=0pt, minimum width=4pt}}
            \tikzset{vxsub/.style = {circle, draw, fill=black!50, inner sep=0pt, minimum width=2pt}}
            
            \tikzset{c1/.style={black, line width=1.5pt,opacity=0.8,rounded corners}}
            \tikzset{vx1/.style = {circle, draw, solid, black, fill=red!30, inner sep=0pt, minimum width=4pt}}
            \tikzset{vx2/.style = {circle, draw, solid, black, fill=blue!30, inner sep=0pt, minimum width=4pt}}
            \tikzset{vx3/.style = {circle, black, solid, fill=orange, draw,
            inner sep=0pt, minimum width=4pt}}

            \foreach \i in {1,2,3}{
                \pgfmathtruncatemacro{\seriesone}{2*\i}
                \pgfmathtruncatemacro{\seriestwo}{2*\i-1}

    \pgfmathtruncatemacro{\seriesthree}{2*\i+10}
                \pgfmathtruncatemacro{\seriesfour}{2*\i+9}

             \draw [red, thick,dotted, out=180,in=-180] (v0-\seriesone) 
                to ++(-3,-.5) 
                to (v0-\seriestwo) ;

            \draw [red, thick,dotted, out=180,in=-180] (v0-\seriesthree) 
                to ++(-3,-.5) 
                to (v0-\seriesfour) ;
            }
            \draw [thick, rounded corners=5pt] (0,1) -- (0,20) -- (18,20)--(18,0)--(0,0)--(0,1);

            \draw [red, thick, rounded corners=5pt] (8,1) -- (8,4.7) -- (14,4.7)--(14,0.3)--(14,0.3)--(8,0.3)--(8,1);
            \draw [red, thick, rounded corners=5pt] (8,6) -- (8,9.7) -- (14,9.7)--(14,5.3)--(14,5.3)--(8,5.3)--(8,6);
            \draw [red, thick, rounded corners=5pt] (8,11) -- (8,14.7) -- (14,14.7)--(14,10.3)--(14,10.3)--(8,10.3)--(8,11);

        \node[] at (11,2.5) {$W_1$};
        \node[] at (11,7.5) {$W_2$};
        \node[] at (11,12.5) {$W_3$};

        \draw[dashed] (1, 9)--(-3,9);
         \draw [c1] (0,1) to (7.5,1) to (7.5,2.5) to (8,2.5);
         \draw [c1] (0,2) to (7,2) to (7,4) to (8,4);
         \draw [c1] (0,3) to (6.5,3) to (6.5,7.5) to (8,7.5);
         \draw [c1] (0,4) to (6,4) to (6,9) to (8,9);
         \draw [c1] (0,5) to (5.5,5) to (5.5,12.5) to (8,12.5);
        \draw [c1] (0,6) to (5,6) to (5,14) to (8,14);

         \draw [c1] (0,11) to (4.5,11) to (4.5, 15.5) to (15,15.5) to (15,14) to (14,14);
         \draw [c1] (0,12) to (4,12) to (4, 16.2) to (15.5,16.2) to (15.5,12.5) to (14,12.5);
        \draw [c1] (0,13) to (3.5,13) to (3.5, 16.9) to (16,16.9) to (16,9) to (14,9);
        \draw [c1] (0,14) to (3,14) to (3, 17.6) to (16.5,17.6) to (16.5,7.5) to (14,7.5);
        \draw [c1] (0,15) to (2.5,15) to (2.5, 18.3) to (17,18.3) to (17,4) to (14,4);
        \draw [c1] (0,16) to (2,16) to (2, 19) to (17.5,19) to (17.5,2.5) to (14,2.5);

    \end{tikzpicture}
    \caption{The case when $q=2$ and $\mathcal{P}_1$, $\mathcal{P}_2$ are in series.}
    \end{subfigure}
    \caption{Finding disjoint subwalls with handles in Lemma~\ref{lem:almost-finding-cycles}.
    }    
    \label{fig:lemma-almost-finding-cycles}
\end{figure}

\begin{lemma}
    \label{lem:almost-finding-cycles}
    Let~$k$,~$c$, and~$r$ be positive integers with~${k \geq 2}$ and~${c, r \geq 3}$.
    Let~${q \in \{0,1,2\}}$.
    Let~$W$ be a ${(c',r')}$-wall in a graph~$G$
    with ${c' \geq c_{\ref{lem:almost-finding-cycles}}(k,c) := c+6k}$ and ${r' \geq r_{\ref{lem:almost-finding-cycles}}(k,r) := k(r+2)}$. 
    Let~${(\mathcal{P}_i \colon i \in [q])}$ be a family of pairwise vertex-disjoint non-mixing $W$-handlebars in~$G$, each of size~$k$, such that one of the following conditions holds.
    \begin{enumerate}
        [label=(\arabic*)]
        \item ${q = 0}$.
        \item ${q = 1}$ and $\mathcal{P}_1$ is either nested or in series.
        \item ${q = 2}$ and~$\mathcal{P}_1$ and~$\mathcal{P}_2$ are both in series. 
    \end{enumerate}
    Then for each~${x \in [k]}$, there exist an $N^W$-anchored $(c,r)$-subwall~$W_x$, a set $\mathcal{H}_{x} = \{H_{x,i} \colon i \in [q]\}$ of~$q$ pairwise vertex-disjoint $W_x$-handles, 
    and a set ${\{ P_{x,i} \in \mathcal{P}_i \colon i \in [q] \}}$
    such that
    \begin{enumerate}
        [label=(\roman*)]
        \item for distinct~${x,x' \in [k]}$ the graphs ${W_x \cup \bigcup \mathcal{H}_{x}}$ and ${W_{x'} \cup \bigcup \mathcal{H}_{x'}}$ are vertex-disjoint and 
        \item ${P_{x,i} \subseteq H_{x,i} \subseteq W \cup P_{x,i}}$ for each~${x \in [k]}$ and each~${i \in [q]}$.
    \end{enumerate}
\end{lemma}

\begin{proof}
    Without loss of generality, we may assume that if~${q > 0}$, then the paths in~$\mathcal{P}_1$ have at least one endvertex in~$C_{1}^W$ 
    and if~${q = 2}$, then each endvertex of each path in~$\mathcal{P}_1$ is $\prec_W$-smaller than each endvertex of each path in~$\mathcal{P}_2$ (since $\mathcal{P}_1$ and $\mathcal{P}_2$ are both in series).  
    If~${q \geq 1}$, then let~${\{ P_{1,x} \colon x \in [k]\}}$ be an enumeration of~$\mathcal{P}_1$ such that for all~${x \in [k-1]}$, the $\prec_W$-smallest endvertex of~$P_{1,x}$ is $\prec_W$\nobreakdash-smaller than both endvertices of~$P_{1,x+1}$, 
    and additionally if~${q = 2}$, then let~${\{ P_{2,x} \colon x \in [k]\}}$ be an enumeration of~$\mathcal{P}_2$ such that for all~${x \in [k-1]}$, the $\prec_W$-smallest endvertex of~$P_{2,x}$ is $\prec_W$-larger than both endvertices of~$P_{2,x+1}$. 
    
    Let~$W^{0}$ be a $4k$-column-slice of~$W$ containing~$C^W_1$ and 
    let~$W^{1}$ be a ${(c'-c-4k)}$-column-slice of~$W$ containing~$C^W_{c'}$. 
    Let~$W^\ast$ be a~$c$-column-slice of~$W$ vertex-disjoint from~${W^{0} \cup W^{1}}$. 
    Let~${\{ W_x \colon x \in [k] \}}$ be a set of $k$ pairwise vertex-disjoint $N^W$-anchored $(c,r)$-subwalls of~$W^\ast$ such that~$W_x$ intersects both~$R^W_{(x-1)r+1}$ and~$R^W_{xr}$ for each~${x \in [k]}$. 
    For each~${x \in [2k]}$ and~${z \in \{0,1\}}$,
    \begin{itemize}
        \item let~$v^z_x$ be the unique nail in the column-boundary of~$W^{z}$ that is contained in both~${R^W_{\lceil xr / 2 \rceil}}$ and~${C^W_{z(c+1) + 4k}}$ and 
        \item let~$w^z_x$ be the unique nail in the column-boundary of~$W^{z}$ that is contained in both~${R^W_{( r' + 1 - x )}}$ and~${C^W_{z(c+1) + 4k}}$. 
    \end{itemize}
    
    Note that for each~${x \in [k]}$, the nails~$v_{2x - 1}^0$,  $v_{2x}^0$, $v_{2x-1}^1$, and~$v_{2x}^1$ are each contained in a row of~$W$ that intersects~$W_x$. 
    For each~${x \in [2k]}$, let~$T_x$ be the unique path in~${R^W_{r'+1-x} \cup C^W_{c'+1-x} \cup R^W_{\lceil xr / 2 \rceil}}$ from~$w^0_x$ to~$v^1_x$. 
    Note that~${\mathcal{T} := \{ T_x \colon x \in [2k] \}}$ is a set of~$2k$ pairwise vertex-disjoint paths that are internally disjoint from~${W^0 \cup \bigcup \{ W_j \colon j \in [k] \}}$. 
    
    If ${q = 0}$, then~$W_x$ with~${\mathcal{H}_x = \emptyset}$ for each~${x \in [k]}$ satisfies the condition and therefore we may assume~${q > 0}$.

    Suppose that ${q = 1}$ and $\mathcal{P}_1$ is in series. 
    As~${k \geq 2}$, each path in~$\mathcal{P}_1$ has both of its endvertices in~$W^{0}$. 
    Since~$W^{0}$ has at least~$2k$ columns, 
    there is a set~$\mathcal{Q}$ of~$2k$ pairwise vertex-disjoint paths from the endvertices of the paths in~${\mathcal{P}_1}$ to the set~${\{ v_x^0 \colon x \in [2k]\}}$ in~$W^{0}$. 
    By the planarity of~$W$, we conclude that for~${x \in [k]}$, the endvertices of~$P_{1,x}$ are linked by two paths~$Q_{x}^{\ast}$ and~$Q_{x}^{\ast\ast}$ in~$\mathcal{Q}$ to~${\{ v^0_{2x-1}, v^0_{2x}\}}$. 
    Moreover, for each~${x \in [k]}$, the path~${Q_{x}^{\ast} \cup Q_{x}^{\ast\ast} \cup P_{1,x}}$ can be easily extended to a $W_x$-handle~$H_{x,1}$ such that all desired properties are satisfied. 
    
    Now suppose that~${q = 1}$ and~$\mathcal{P}_1$ is nested. 
    If each path in~$\mathcal{P}_1$ has one endvertex in~$W^{0}$ and one endvertex in~$W^{1}$, then there is a set~$\mathcal{Q}$ of~$2k$ pairwise vertex-disjoint paths containing for each~${z \in \{0,1\}}$ 
    a subset of~$k$ paths from the endvertices in~$W^z$ of the paths in~$\mathcal{P}_1$ to~${\{ v^z_{2x} \colon x \in [k] \}}$ in~$W^z$. 
    If each path in~$\mathcal{P}_1$ has both of its endvertices in~$W^{0}$, then there are~$2k$ vertex-disjoint paths from the endvertices of the paths in~$\mathcal{P}_1$ to~${\{ v^0_{2x} \colon x \in [k] \} \cup \{ w^0_{2x} \colon x \in [k] \}}$ in~$W^{0}$, which together with the paths in~${\{ T_{2x} \colon x \in [k] \}}$ yield a set~$\mathcal{Q}$ of~$2k$ pairwise vertex-disjoint paths from the endvertices of~$\mathcal{P}_1$ to~${\{ v^z_{2x} \colon x \in [k], z \in \{0,1\} \}}$.  
    Hence, in both of these cases, the set~$\mathcal{Q}$ avoids~${\bigcup \{ W_j \colon j \in [k]\}}$. 
    By the planarity of~$W$, we conclude that for~${x \in [k]}$, the endvertices of the path~${P_{1,x} \in \mathcal{P}_1}$ are linked by two paths~$Q_{x}^{\ast}$ and~$Q_{x}^{\ast\ast}$ in~$\mathcal{Q}$ to~${\{ v^{0}_{2x} , v^{1}_{2x} \}}$. 
    As before, for each~${x \in [k]}$, the path~${Q_{x}^{\ast} \cup Q_{x}^{\ast\ast} \cup P_{1,x}}$ can be easily extended to a $W_x$-handle~$H_{x,1}$ such that all desired properties are satisfied. 
    
    Therefore, we may assume that ${q = 2}$. 
    Recall that the paths in~$\mathcal{P}_1$ have both of their endvertices in~$W^{0}$. 
    If each path in~$\mathcal{P}_2$ has both of its endvertices in~$W^{1}$, then 
    since each of~$W^0$ and~$W^1$ has at least~$2k$ columns, 
    there exist a set~$\mathcal{Q}_1$ of~$2k$ pairwise vertex-disjoint paths from the endvertices of the paths in~$\mathcal{P}_1$ to~${\{ v^0_{x} \colon x \in [2k] \}}$ in~$W^0$ and a set~$\mathcal{Q}_2$ of~$2k$ pairwise vertex-disjoint paths from the endvertices of the paths in~$\mathcal{P}_2$ to~${\{ v^1_{x} \colon x \in [2k] \}}$ in~$W^1$. 
    If each path in~$\mathcal{P}_2$ has both of its endvertices in~${W^0}$ as well, then 
    since~$W^0$ has~$4k$ columns, 
    there are~$4k$ vertex-disjoint paths from the set of endvertices of~${\mathcal{P}_1 \cup \mathcal{P}_2}$ to the set~${\{ v_{x}^0 \colon x \in [2k] \} \cup \{ w_{x}^0 \colon x \in [2k] \}}$ in~$W^0$. 
    In this case, let~$\mathcal{Q}_1$ be the subset of these paths with endvertices in~$\{ v_{x}^0 \colon x \in [2k] \}$ and let~$\mathcal{Q}_2$ be the concatenation of the subset of these paths with endvertices in~$\{ w_{x}^0 \colon x \in [2k] \}$ together with the paths in~$\{ T_x \colon x \in [2k] \}$. 
    Hence, in both of these cases, by the planarity of~$W$, for each~${i \in [2]}$, the endvertices of~$P_{i,x}$ are linked by two paths~$Q_{i,x}^{\ast}$ and~$Q_{i,x}^{\ast\ast}$ in~$\mathcal{Q}_i$ to ${\{v_{2x-1}^{i-1}, v_{2x}^{i-1}\}}$ and these paths avoid~${\bigcup \{ W_j \colon j \in [k]\}}$. 
    As before, for each~${i \in [2]}$ and~${x \in [k]}$, the path~${Q_{i,x}^{\ast} \cup Q_{i,x}^{\ast\ast} \cup P_{i,x}}$ can be easily extended to a $W_x$-handle~$H_{x,i}$ such that all desired properties are satisfied. 
\end{proof}

\section{Lemmas for products of abelian groups}
\label{sec:grouplemmas}

In this section, we present some additional lemmas from~\cite{GollinHKKO2021} and prove useful extensions on finding allowable values. 
The first lemma says that if a set of elements of~$\Gamma$ generates an allowable value, then it does so using each element a bounded number of times. 

\begin{lemma}[Gollin et al.~{\cite[Corollary 7.2]{GollinHKKO2021}}]
    \label{lem:omega-avoiding}
    Let~$m$, $t$, and~$\omega$ be positive integers, let~${\Gamma = \prod_{j \in [m]} \Gamma_j}$ be a product of~$m$ abelian groups and for all~${j \in [m]}$, let~$\Omega_j$ be a subset of~$\Gamma_j$ of size at most~$\omega$. 
    For all~${i \in [t]}$ and~${j \in [m]}$, let~${g_{i,j}}$ be an element of~$\Gamma_j$.
    If there are integers~${c_1, \dots, c_t}$ 
    such that~${\sum_{i=1}^t c_i g_{i,j} \notin \Omega_j}$ 
    for all~${j \in [m]}$, 
    then there are integers~${d_1, \dots, d_t}$ with~${d_i \in [2^{m\omega}]}$ for each~${i \in [t]}$ 
    such that~${\sum_{i=1}^t d_i g_{i,j} \notin \Omega_j}$ for all~${j \in [m]}$. 
\end{lemma}

The next lemma allows us to find large sets of elements of~$\Gamma$ such that for each~${j \in [m]}$, their $\gamma_j$-values are either all equal or all distinct. 

\begin{lemma}[Gollin et al.~{\cite[Lemma 7.6]{GollinHKKO2021}}]
    \label{lem:ramsey}
    There exists a function~$f_{\ref{lem:ramsey}} \colon \mathbb{N}^2 \to \mathbb{N}$ satisfying the following. 
    Let~$m$,~$t$, and~$N$ be positive integers 
    with~${N \geq f_{\ref{lem:ramsey}}(t,m)}$ 
    and let~${\Gamma = \prod_{j \in [m]} \Gamma_j}$ be a product of~$m$ abelian groups. 
    Then for every sequence~${(g_i \colon i \in [N])}$ over~$\Gamma$, 
    there exists a subset~$I$ of~$[N]$ with~${\abs{I} = t}$ such that for each~${j \in [m]}$, either
    \begin{itemize}
        \item ${\pi_j(g_{i}) = \pi_j(g_{i'})}$ for all ${i,i' \in I}$ or 
        \item ${\pi_j(g_{i}) \neq \pi_j(g_{i'})}$ for all distinct~${i,i' \in I}$.
    \end{itemize}
    Furthermore, if~$Z$ is a subset of~${[m]}$ such that for all distinct~$i$ and~$i'$ in~${[N]}$ there exists ${j \in Z}$ such that~${\pi_x(g_i) \neq \pi_x(g_{i'})}$, then the second condition holds for some~${j \in Z}$. 
\end{lemma}

For the coordinates~$j$ for which the $\gamma_j$-values are all distinct, we have the following extension of Lemma~\ref{lem:ramsey}. 

\begin{lemma}
    \label{lem:sidon}
    Let $t$, $m$, and $n$ be positive integers, let ${\Gamma = \prod_{j\in[m]} \Gamma_j}$ be a product of $m$ abelian groups, and let~${(g_i\colon i \in [n]})$ be a family of elements of~$\Gamma$ such that~$\pi_j(g_i) \neq \pi_j(g_{i'})$ for all~${j \in [m]}$ and distinct~$i$ and~$i'$ in~${[n]}$. 
    If ${n \geq f_{\ref{lem:sidon}}(t,m) := m3^{t-1}+t}$, then 
    there is a subset~${I \subseteq [n]}$ of size~$t$ such that 
    \[ \pi_j\big(\sum_{i\in S}g_i\big) \neq \pi_j\big(\sum_{i\in T}g_i\big) \] 
    for every~${j \in [m]}$ and any pair of distinct subsets~$S$ and~$T$ of~$I$.
\end{lemma}

\begin{proof}
    Let~$I$ be a maximal subset of~${[n]}$ such that~${\pi_j(\sum_{i \in S} g_i) \neq \pi_j(\sum_{i \in T} g_i)}$
    for every~${j \in [m]}$ and every pair of distinct subsets~$S$ and~$T$ of~$I$.
    Suppose that~${\abs{I} < t}$. 
    By the maximality of~$I$, 
    for each~${a \in [n] \setminus I}$, there are disjoint subsets~$S'$ and~$T'$ of~$I$ such that~${\pi_j(g_a) = \pi_j(\sum_{i \in S'} g_i - \sum_{i \in T'} g_i)}$ for some~${j \in [m]}$. 
    Note that there are $3^{\abs{I}}$ ways to choose the disjoint subsets~$S'$ and~$T'$ of~$I$. 
    Since~${\pi_j(g_a) \neq \pi_j(g_{a'})}$ for every~${j \in [m]}$ and every pair of distinct elements~$a$ and~$a'$ in~${[n] \setminus I}$, we have that~${n-(t-1) \leq n - \abs{I} \leq m3^{\abs{I}}} \leq m3^{t-1}$, contradicting the assumption on~$n$. 
\end{proof}

We will apply Lemma~\ref{lem:ramsey} multiple times to obtain a family~${(S_i \colon i \in [t])}$ of large subsets of~$\Gamma$ each satisfying the conclusion of Lemma~\ref{lem:ramsey}. 
The following lemma says that there is a choice of an element from each~$S_i$ so that the sum of the chosen elements is allowable in each coordinate~$j$ for which at least one~$S_i$ has all distinct $\gamma_j$-values. 

\begin{lemma}[Gollin et al.~{\cite[Lemma 7.4]{GollinHKKO2021}}]
    \label{lem:vectorsum}
    Let~$m$,~$t$, and~$\omega$ be positive integers, let~${\Gamma = \prod_{j \in [m]} \Gamma_j}$ be a product of~$m$ abelian groups, and 
    for all~${j \in [m]}$, let~$\Omega_j$ be a subset of~$\Gamma_j$ of size at most~$\omega$.
    Let~${(S_i \colon i \in [t])}$ be a family of subsets of~$\Gamma$ 
    such that for each~${j \in [m]}$, there exists~${i \in [t]}$ such that~${\pi_j(g) \neq \pi_j(g')}$ 
    for all distinct~${g, g'}$ in~$S_i$.
    If~${\abs{S_i} > m\omega }$ for all~${i \in [t]}$,
    then for every~${h \in \Gamma}$, there is a sequence~${(g_i \colon i \in [t])}$ of elements of~$\Gamma$ such that 
    \begin{enumerate}
        [label=(\roman*)]
        \item ${g_i \in S_i}$ for each ${i \in [t]}$ and 
        \item ${\pi_j\bigl(h + \sum_{i \in [t]} g_i \bigr) \notin \Omega_j}$ for all~${j \in [m]}$.
    \end{enumerate}
\end{lemma}

The final lemma is an extension of Lemma~\ref{lem:vectorsum} that given a family~${(S_i \colon i \in [t])}$ of large subsets of~$\Gamma$ satisfying the conclusion of Lemma~\ref{lem:ramsey}, there are large subsets~$S_i'$ of~$S_i$ so that for \emph{every} choice of an element from each~$S_i'$, the sum of the chosen elements is allowable in each coordinate~$j$ for which at least one~$S_i$ has all distinct $\gamma_j$-values.

\begin{lemma}
    \label{lem:alltransversals}
    Let~$m$, $\omega$, $\kappa$, $t$, and~$s$ be positive integers with ${s \geq f_{\ref{lem:alltransversals}}(m,\omega,\kappa,t) := \kappa + m\omega\kappa^{t-1}}$, 
    let~${\Gamma = \prod_{j \in [m]} \Gamma_j}$ be a product of~$m$ abelian groups, 
    and for each~${j \in [m]}$, let~$\Omega_j$ be a subset of~$\Gamma_j$ of size at most~$\omega$. 
    Let~${(g_{i,x} \colon i \in [t],\, x \in[s])}$ be a family of elements of~$\Gamma$ 
    such that for each~${j \in [m]}$, we have
    \begin{enumerate}
        [label=\textnormal{(\alph*)}, series=alltransversals]
        \item \label{item:alltransversals1} ${\abs{\{ \pi_j(g_{i,x}) \colon x \in [s] \}} \in \{1,s\}}$ for each~${i \in [t]}$ and
        \item \label{item:alltransversals2} 
        ${\pi_j\bigl (\sum_{i \in [t]} g_{i,1}\bigr) \notin \Omega_j}$. 
    \end{enumerate}
    Then there are subsets~${I_i \subseteq [s]}$ for~${i \in [t]}$, each of size at least~$\kappa$, such that 
    \[
        {\pi_j\Bigl(\sum_{i \in [t]} g_{i,a_i}\Bigr) \notin \Omega_j}
    \] 
    for every~${j \in [m]}$ and every~${(a_i \in I_i \colon i \in [t])}$. 
\end{lemma}

\begin{proof}
    Let~${(I_i \subseteq [s] \colon i \in [t])}$ be a family satisfying 
    \begin{enumerate}
        [label=(\arabic*)]
        \item\label{item:alltransversalsproof-1} ${1 \in I_i}$ for all~${i \in [t]}$,
        \item\label{item:alltransversalsproof-2} ${\abs{I_i} \leq \kappa}$ for all~${i \in [t]}$,
        \item\label{item:alltransversalsproof-3} for every~${j \in [m]}$ and every~${(a_i\in I_i \colon i \in [t]) }$, we have~${\pi_j(\sum_{i \in [t]} g_{i,a_i}) \notin \Omega_j}$, and
        \item\label{item:alltransversalsproof-4} subject to the previous conditions, ${\sum_{i \in [t]} \abs{I_i}}$ is maximised. 
    \end{enumerate}
    By~\ref{item:alltransversals2}, such a family~${(I_i \colon i \in [t])}$ exists. 

    Suppose for contradiction that~${\abs{I_x} < \kappa}$ for some~${x \in [t]}$. 
    Without loss of generality we assume that~${x = t}$. 
    By properties~\ref{item:alltransversalsproof-3} and~\ref{item:alltransversalsproof-4}, for each~${y \in [s] \setminus I_t}$ there exist~${j \in [m]}$ and~${(a_i\in I_{i} \colon i \in [t-1])}$ such that~${\pi_j(g_{t,y} + \sum_{i \in[t-1]} g_{i,a_i}) \in \Omega_j}$. 
    Since $s\geq \kappa + m\omega \kappa^{t-1}$, we have 
    \[
        \frac{\abs{[s] \setminus I_t}}{m\prod_{i \in [t-1]}\abs{I_i}} \geq \frac{\abs{[s] \setminus I_t}}{m\kappa^{t-1}}>\omega \geq \max_{j\in [m]}\abs{\Omega_j}, 
    \]
    so by the pigeonhole principle, there exist an integer~${j \in [m]}$, a family~${(a_i \in I_i \colon i \in [t-1])}$, and distinct integers~${y,y' \in [s] \setminus I_t}$ such that~${\pi_j(g_{t,y} + \sum_{i \in [t-1]} g_{i,a_i}) = \pi_j(g_{t,y'} + \sum_{i \in [t-1]} g_{i,a_i}) \in \Omega_j}$. 
    This implies that~${\pi_j(g_{t,y}) = \pi_j(g_{t,y'})}$ and 
    by~\ref{item:alltransversals1}, we have ${\abs{\{ \pi_j(g_{t,x}) \colon x \in [s] \}} = 1}$. 
    Thus~${\pi_j(g_{t,y}) = \pi_j(g_{t,1})}$ and so~${\pi_j(g_{t,1} + \sum_{i \in[t-1]} g_{i,a_i}) \in \Omega_j}$, contradicting properties~\ref{item:alltransversalsproof-1} and~\ref{item:alltransversalsproof-3}.
\end{proof}

\section{Proof of the main theorem}
\label{sec:proof}

In this section, we prove the main theorem, which we will restate for the convenience of the reader. 

\mainobstruction*

\begin{proof}
    For fixed positive integers~$m$, $\omega$, $\kappa$, and~$\theta$, we will define~${\widehat f_{m,\omega}(k,\kappa,\theta)}$ by recursion on~$k$. 
    First, we set~${\widehat f_{m,\omega}(1,\kappa,\theta) := 0}$. 
    Assume that~${k > 1}$ and ${\widehat f_{m,\omega}(k-1,\kappa,\theta)}$ is already defined. 
    We define~${\maxk := \max \{ k, \kappa \}}$. 
    
    For integers~$p$ and~$z_0$ with~${p > 0}$ and ${0 \leq z_0 \leq m}$, let~${\alpha(p,z_0)}$ and~${\rho(z_0)}$ be recursively defined as follows. 
    For every positive integer~$p$, we define
    \begin{align*}
        \rho(0) &:= m, \\
        \alpha(p,0) &:= f_{\ref{lem:handlebars}}\Bigl(\rho(0),f_{\ref{lem:sidon}}\bigl(2^{m\omega + 1} f_{\ref{lem:alltransversals}}(m,\omega,\maxk,\rho(0)),m\bigr)\Bigr), 
    \end{align*}
    and for~${z_0 > 0}$, 
    we recursively define
    \begin{align*}
        \rho(z_0) &:= m + f_{\ref{lem:ramsey}}(\alpha(1,z_0-1),m), \\
        \alpha(p,z_0) &:= 
        \begin{cases}
            \alpha(1, z_0-1) & \textnormal{ if } p \geq \rho(z_0),\\
            \max\Bigl\{ 4f_{\ref{lem:ramsey}}(\alpha(p+1,z_0),m),\, f_{\ref{lem:handlebars}}\Bigl(p,f_{\ref{lem:sidon}}\bigl(2^{m\omega + 1} f_{\ref{lem:alltransversals}}(m,\omega,\maxk,p),m\bigr)\Bigr) \Bigr\} & \textnormal{ otherwise. } 
        \end{cases}
    \end{align*}
    Let~${\widehat{p} := \rho(m)}$. 
    Note that~${\alpha(x, z_0) \geq \alpha(\rho(z_0),z_0) = \alpha(1,z_0-1) \geq \alpha(x, z_0-1)}$ for~${x > 0}$ and~${z_0 > 0}$. 
    Thus, $\alpha$ is increasing in the second argument. 
    We may also assume that~$f_{\ref{lem:ramsey}}$ is increasing in its first argument. 
    These two properties imply that~${\rho(z_0) \leq \widehat{p}}$ for all~${z_0 \leq m}$. 
    Let 
    \[
        u := \max\{ \lceil  \widehat f_{m,\omega}(k-1,\kappa,\theta)/3\rceil, 
        f_{\ref{thm:tpath}} ( f_{\ref{lem:addlinkage}} ( f_{\ref{lem:ramsey}}(\alpha(1,m),m) ) ) + 3 \}.
    \]
    
    We recursively define~${\beta(p,z_0,z)}$ for integers~$p$, $z_0$, and~$z$ with~${0 \leq z_0 \leq z \leq m}$ and~${0 \leq p \leq \widehat{p}}$, 
    as well as~${\psi(z)}$ for an integer~$z$ with~${0 \leq z \leq m+1}$ and~${c_x(z)}$, ${r_x(z)}$ for~${x \in \{0,1,2\}}$ and a non-negative integer~${z \leq m}$ as follows. 
    We define 
    \begin{align*}
        \psi(m+1) &:= 3, 
        \intertext{and for~${z \leq m}$ we define} 
        c_0(z) &:= c_{\ref{lem:omega-avoiding-cycle}}(2,\psi(z+1)+2,m,\omega), \\ 
        r_0(z) &:= r_{\ref{lem:omega-avoiding-cycle}}(2,\psi(z+1)+2,m,\omega), \\
        c_1(z) &:= c_{\ref{lem:almost-finding-cycles}}(k,c_0(z)), \\
        r_1(z) &:= r_{\ref{lem:almost-finding-cycles}}(k,r_0(z)), \\
        c_2(z) &:= \max\big\{ \theta,
            c_{\ref{lem:combining-handlebars}}(\widehat{p},k,c_1(z)),\, 
            \maxk \cdot c_{\ref{lem:omega-avoiding-cycle}}( \widehat{p},\psi(z+1)+2,m,\omega)
            \big\}, \\ 
        r_2(z) &:= \max\big\{ \theta,
            r_{\ref{lem:combining-handlebars}}(k,r_1(z)),\,
            r_{\ref{lem:omega-avoiding-cycle}}( \widehat{p},\psi(z+1)+2,m,\omega)
            \big\}, \\
        \beta(p,z_0,z) &:=
        \begin{cases}
            \max \big\{ 
                u,\, 
                c_2(z)+2 \big\}
                & \text{if } z_0 = 0,\\
            \beta(1,z_0-1,z) 
                & \text{if } z_0 > 0 \text{ and } p = \widehat{p},\\
            w_{\ref{lem:addlinkage}}(f_{\ref{lem:ramsey}}(\alpha(p+1,z_0),m),\, \beta(p+1,z_0,z)) 
                & \text{if } z_0 > 0 \text{ and } p < \widehat{p}; 
        \end{cases}\\
        \psi(z) &:= \max\big\{ 
            \psi(z+1),\, 
            \beta(0,z,z),\, 
            r_2(z) 
            \big\}.
    \end{align*}
    Observe that~${\beta(p,z_0,z) \geq u}$. 
    Lastly, we define 
    \[
        \widehat f_{m,\omega}(k,\kappa,\theta) := \max\big\{
        6  f_{\ref{thm:wall}}(\psi(0)+2),\,  
        6  u, \,
        12 \widehat f_{m,\omega}(k-1, \kappa, \theta) 
        \big\}.
    \]
    
    \medskip

    We proceed by induction on~$k$. 
    The case~${k = 1}$ is clear. 
    Suppose that~${k > 1}$. 
    For every subgraph~$H$ of~$G$, let~${\nu(H)}$ denote the maximum number of vertex-disjoint cycles~$O$ in~$H$ with~${\gamma(O) \in A}$. 
    Observe that~$\nu$ is a packing function for~$G$. 
    
    Suppose for contradiction that~${\nu(G) < k}$, $\tau_\nu(G)>\widehat f_{m,\omega}(k,\kappa,\theta)$, and there is no $\Gamma$-labelling~$\gamma'$ of~$G$ that is shifting equivalent to~$\gamma$ such that 
    the statement~\ref{item:main-obstruction} holds. 
    Let~$T$ be a minimum $\nu$\nobreakdash-hitting set of size~${t := \tau_\nu(G)}$. 
    By assumption,~${t > \widehat f_{m,\omega}(k,\kappa,\theta) > \widehat f_{m,\omega}(k-1,\kappa,\theta)}$. 
    By the induction hypothesis, $G$ contains ${k-1}$ vertex-disjoint cycles in~$\mathcal{O}$ and therefore~${\nu(G) = k-1}$. 
    For each subgraph~$H$ of~$G$, if~${\nu(H) < \nu(G)}$, then 
    by the induction hypothesis, 
    \[
        \tau_\nu(H) \leq \widehat f_{m,\omega}(k-1,\kappa,\theta) \leq \widehat f_{m,\omega}(k,\kappa,\theta)/12 < t/12.
    \] 
    Let~$\mathcal{T}_T$ be the set of all separations~${(A,B)}$ of~$G$ of order less than~${t/6}$ with~${\abs{B \cap T} > 5t/6}$.
    By Lemma~\ref{lem:welllinked}, $\mathcal{T}_T$ is a tangle of order~${\lceil t/6 \rceil > f_{\ref{thm:wall}}(\psi(0)+2)}$. 
    By Theorem~\ref{thm:wall}, $G$ has a wall of order~${\psi(0)+2}$ dominated by~$\mathcal{T}_T$. 
    By Lemmas~\ref{lem:cleansubwall} and~\ref{lem:dominatedsubwall},
    this wall has a ${(\psi(\abs{Z}),\psi(\abs{Z}))}$-subwall~$W$ that is ${(\gamma',Z,\psi(\abs{Z}+1)+2)}$-clean
    for some subset~${Z \subseteq [m]}$ and some $\Gamma$-labelling~$\gamma'$ of~$G$ shifting-equivalent to~$\gamma$ 
    and dominated by~$\mathcal{T}_T$. 
    Since~${\gamma(O) = \gamma'(O)}$ for every cycle~$O$ in~$G$, we may assume without loss of generality that~${\gamma = \gamma'}$.

    The following claim simplifies the situation by choosing a good column-slice~$W'$ of~$W$ and a good set of $W'$-handlebars so that 
    the $\gamma_j$-values of $W'$-handles in each $W'$-handlebar are all distinct or all same for each $j\in Z$ 
    and further satisfies some properties so that we can later use them to find cycles with allowable $\gamma$-values.
    
    \setcounter{claim}{0}
    \begin{claim}
        \label{clm:manyhandles}
        There exist an integer~${c \geq \beta(1,0,\abs{Z})}$, 
        a set~${I \subseteq [\widehat{p}]}$, 
        a $c$-column-slice~$W'$ of~$W$, 
        a family~${( \mathcal{P}_i \colon i \in I)}$ of pairwise vertex-disjoint non-mixing 
        $W'$-handlebars, 
        a family~${( Z_i \colon i \in I )}$ of subsets of~$Z$, 
        and a family~${( g_i \colon i \in I )}$ of elements of~$\Gamma$
        such that 
        \begin{enumerate}
            [label=\textnormal{(\alph*)}, series=manyhandles]
            \item \label{item:manyhandles-a} 
                ${\abs{\mathcal{P}_i} \geq 2^{m\omega + 1} f_{\ref{lem:alltransversals}}(m,\omega,\maxk,\abs{I})}$ for each~${i \in I}$,  
            \item \label{item:manyhandles-b} 
                ${\abs{\pi_j(\gamma(\mathcal{P}_i))} = \abs{\mathcal{P}_i}}$ for all~${i \in I}$ and ${j \in Z_i}$, 
            \item \label{item:manyhandles-c} 
                ${\pi_j(\gamma(\mathcal{P}_i)) = \{ \pi_j(g_i) \}}$ for all~${i \in I}$ and ${j \in Z \setminus Z_i}$, 
            \item \label{item:manyhandles-d} 
                there is some~${g \in \gen{g_i \colon i \in I}}$ such that ${\pi_j(g) \notin \Omega_j}$ for all~${j \in Z \setminus \bigcup_{i \in I} Z_i}$, 
            \item \label{item:manyhandles-e} 
                for every~${i \in I}$ and every~${g \in \gen{g_{i'} \colon i' \in I \setminus\{i\}}}$, there is some~${j \in Z \setminus \bigcup_{i' \in I \setminus \{i\}} Z_y}$ such that~${\pi_j(g) \in \Omega_j}$, and 
            \item \label{item:manyhandles-f} 
                for each~${i \in I}$ and ${j \in Z_i}$, and every pair of distinct subsets~$\mathcal{S}$ and~$\mathcal{T}$ of~$\mathcal{P}_i$, we have 
                \[
                    \pi_j\Bigl(\sum_{P\in \mathcal{S}}\gamma(P)\Bigr) \neq 
                    \pi_j\Bigl(\sum_{P\in \mathcal{T}}\gamma(P)\Bigr).
                \]
        \end{enumerate} 
    \end{claim}
    
    \begin{proofofclaim}
        For non-negative integers~$c$,~$q$, and~$p$ with~${q \leq p}$, 
        we say that a triple ${(W',\mathfrak{P}, \mathcal{Z})}$ consisting of 
        a wall~$W'$, 
        a family ${\mathfrak{P} := (\mathcal{P}_i \colon i \in [p])}$ of pairwise vertex-disjoint sets of $W'$-handles, 
        and a family ${\mathcal{Z} := ( Z_i \colon i \in \{0\} \cup [p])}$ of subsets of~$Z$ 
        is a \emph{${(c,q,p)}$-McGuffin} 
        if~$W'$ is a $c$-column-slice of~$W$
        such that 
        \begin{enumerate}
            [label=(\arabic*)]
            \item \label{item:mcguffin-1}
               ${\abs{\mathcal{P}_i} \geq \alpha(p,\abs{Z_0})}$ for all~${i \in [p]}$, 
            \item \label{item:mcguffin-2}
                ${\abs{\pi_j(\gamma(\mathcal{P}_i))} = \abs{\mathcal{P}_i}}$  for all~${i \in [p]}$ and~${j \in Z_i}$, 
            \item \label{item:mcguffin-3}
                ${\abs{\pi_j(\gamma(\mathcal{P}_i))} = 1}$  for all~${i \in [p]}$ and~${j \in Z \setminus Z_i}$, 
            \item \label{item:mcguffin-4} 
                ${Z_0 = Z \setminus \bigcup_{i \in [q]} Z_{i}}$,
            \item \label{item:mcguffin-5} 
                ${Z_i \setminus \bigcup_{i' \in[i-1]} Z_{i'} \neq \emptyset}$ for all~${i \in [q]}$, and 
            \item \label{item:mcguffin-6} 
                for all distinct~${i, i' \in [p] \setminus [q]}$, there is ${j \in Z_{0}}$ such that~${\pi_j(\gamma(\mathcal{P}_{i})) \cap \pi_j(\gamma(\mathcal{P}_{i'})) = \emptyset}$. 
        \end{enumerate}
        Note that~${(W,\emptyset,(Z))}$ is a ${(\psi(\abs{Z}),0,0)}$-McGuffin
        and by the definition, $\psi(\abs{Z})\geq \beta(0,\abs{Z},\abs{Z})$.
        Furthermore, if ${(W',\mathfrak{P},\mathcal{Z})}$ is a ${(c,q,p)}$-McGuffin, then~${q \leq \abs{Z}}$ by~\ref{item:mcguffin-5}
        and~${\abs{Z_0} \leq m}$, which implies that~${\rho(\abs{Z_0}) \leq \widehat{p}}$. 
        Let~${(q,p)}$ be a lexicographically maximal pair of non-negative integers 
        with ${q \leq p \leq \widehat{p}}$ 
        for which there is a ${(c,q,p)}$-McGuffin~${(W',\mathfrak{P},\mathcal{Z})}$ for some~${c \geq \beta(p,\abs{Z_0},\abs{Z})}$. 
        
        First, we claim that~${p < \rho(\abs{Z_0})}$.
        Suppose that~${p \geq \rho(\abs{Z_0})}$. 
        Then 
        \[ 
            {p - q \geq \rho(\abs{Z_0}) - m \geq  f_{\ref{lem:ramsey}}(\alpha(q+1,\abs{Z_0}-1),m)}
        \] 
        since~${q \leq m}$ by~\ref{item:mcguffin-5} and~$\alpha$ is decreasing in its first argument. 
        Let~$\mathcal{P}''$ be a set of~${p-q}$ pairwise vertex-disjoint $W'$-handles containing exactly one element of~$\mathcal{P}_i$ for each~${i \in [p] \setminus [q]}$. 
        For~${i \in [q]}$, let~${\mathcal{P}'_i := \mathcal{P}_i}$ and~${Z'_i := Z_i}$. 
        Note that~${\abs{\mathcal{P}_i} \geq \alpha(p,\abs{Z_0}) \geq 4 f_{\ref{lem:ramsey}}(\alpha(p+1,\abs{Z_0}),m)
        \geq 4 \alpha(p+1,\abs{Z_0})}$ for each~${i \in [p]}$.
        Thus, by Lemma~\ref{lem:ramsey}, there is a subset~$\mathcal{P}'_{q+1}$ of~$\mathcal{P}''$ 
        with~${\abs{\mathcal{P}'_{q+1}} = \alpha(q+1,\abs{Z_0}-1)}$
        such that for each~${j \in [m]}$, either 
        \begin{itemize}
            \item ${\pi_j(\gamma(P)) = \pi_j(\gamma(Q))}$ for all~${P,Q \in \mathcal{P}'_{q+1}}$ or
            \item ${\pi_j(\gamma(P)) \neq \pi_j(\gamma(Q))}$ for all distinct~${P,Q \in \mathcal{P}'_{q+1}}$,
        \end{itemize}
        and the second condition holds for some~${j \in Z_0}$ 
        since by~\ref{item:mcguffin-6}, 
        for all distinct paths~$P$ and~$Q$ in~$\mathcal{P}''$, there exists ${j \in Z_0}$ such that~${\pi_j(\gamma(P)) \neq \pi_j(\gamma(Q))}$. 
        Let 
        \[
            Z'_{q+1} := 
            \{ j \in Z_0 \colon \pi_j(\gamma(P)) \neq \pi_j(\gamma(Q)) \text{ for all distinct } P, Q \in \mathcal{P}'_{q+1} \}
            \text{ and } 
            Z'_0 := Z_0 \setminus Z'_{q+1}. 
        \]
        Let~${\mathfrak{P}' := (\mathcal{P}_i' \colon i \in [q+1])}$ and~${\mathcal{Z}' := (Z_i' \colon i \in \{0\} \cup [q+1])}$. 
        Then~${(W',\mathfrak{P}', \mathcal{Z}')}$ is a ${(c,q+1,q+1)}$-McGuffin, since~\ref{item:mcguffin-1} follows from the fact that~${\abs{\mathcal{P}'_{q+1}} \geq \alpha(q+1,\abs{Z_0}-1) \geq \alpha(q+1,\abs{Z_0'})}$ and the remaining conditions are easy to check. 
        This contradicts the maximality of~${(q,p)}$ since~${q+1 \leq p \leq \widehat{p}}$. 
        Therefore,~${p < \rho(\abs{Z_0}) \leq \widehat{p}}$.
        
        Now let us show that~${(W',\mathfrak{P},\mathcal{Z})}$ satisfies the following statement:
        \begin{enumerate}
            [label=($\ast$)]
            \item \label{item:mcguffin-ast}
                There is some~${g \in \gen{\bigcup_{i \in [p]} \gamma(\mathcal{P}_{i})}}$ such that~${\pi_j(g) \notin \Omega_j}$ for all~${j \in Z_0}$. 
        \end{enumerate}
        Suppose to the contrary that such~$g$ does not exist.
        Then~$Z_0$ is nonempty. 
        Let~$\Lambda$ be the subgroup of~$\Gamma$ consisting of all~${g \in \Gamma}$ 
        for which there is ${g' \in \gen{\bigcup_{i \in [p]} \gamma(\mathcal{P}_i)}}$ such that ${\pi_j(g) = \pi_j(g')}$ for all~${j \in Z_{0}}$.
        Let~$\lambda$ be the induced ${\Gamma/\Lambda}$-labelling of~$G$. 
        Note that by the negation of~\ref{item:mcguffin-ast}, 
        neither ${\gen{\bigcup_{i \in [p]} \gamma(\mathcal{P}_i)}}$ nor~$\Lambda$ contains an element~$g$ such that~${\pi_j(g) \notin \Omega_j}$ for all~${j \in Z_0}$. 
        Therefore, 
        \begin{enumerate}
            [label=($\dagger$)]
            \item \label{item:mcguffin-dagger} 
                every cycle~$O$ of~$G$ for which~${\pi_j(\gamma(O)) \notin \Omega_j}$ for all~${j \in [m]}$ is $\lambda$-non-zero. 
        \end{enumerate}
        
        Note that~$W'$ is a subwall of~$W$ of order~${c \geq u}$. 
        For any~${S \subseteq V(G)}$ of size at most~${u-1}$, there is a component~$X$ of~${G-S}$ containing a row of~$W'$, 
        which contains a vertex in~${\branch(W')}$ because~${u \geq 3}$. 
        By Lemma~\ref{lem:dominatedsubwall},~$\mathcal{T}_T$ dominates~$W'$, so the separation~${(V(G) \setminus V(X), S \cup V(X))}$ is in~$\mathcal{T}_T$ and hence~$X$ contains a vertex of~$\branch(W')$ and at least 
        \[ 
            5t/6 - (u-1) > 5\widehat{f}_{m,\omega}(k,\kappa,\theta)/6-(u-1) > 4u 
        \] 
        vertices of~$T$. 
        By~\ref{item:mcguffin-dagger}, every minimal subgraph~$H$ with~${\nu(H) \geq 1}$ is a $\lambda$-non-zero cycle. 
        Moreover, if~$H$ is a subgraph of~$G$ with~${\nu(H) < \nu(G) = k - 1}$, then 
        by the induction hypothesis,
        \[ 
            \tau_\nu(H) \leq \widehat{f}_{m,\omega}(k-1,\kappa,\theta) \leq 3u. 
        \] 
        Hence, by Lemma~\ref{lem:cover},~$G$ has~${f_{\ref{lem:addlinkage}} ( f_{\ref{lem:ramsey}}(\alpha(1,m),m) )}$ vertex-disjoint $\lambda$-non-zero~$\branch(W')$-paths. 
        Note that we may assume that the function~$w_{\ref{lem:addlinkage}}$ is increasing in both of its arguments. 
        As~${\abs{Z_0} > 0}$ and~${p < \widehat{p}}$, we have
        \begin{align*}
            c \geq \beta(p, \abs{Z_0}, \abs{Z}) 
            &\geq w_{\ref{lem:addlinkage}}(
                f_{\ref{lem:ramsey}}(\alpha(p+1,\abs{Z_0}),m),\beta(p+1,\abs{Z_0},\abs{Z})).
        \end{align*}
        Recall that~${\abs{\mathcal{P}_i}\geq \alpha(p,\abs{Z_0})
        \geq 4  f_{\ref{lem:ramsey}}(\alpha(p+1,\abs{Z_0}),m)}$ for each~${i \in [p]}$. 
        Thus, by Lemma~\ref{lem:addlinkage} applied to~$W'$, there exists a $c'$-column-slice~$W''$ of~$W'$ 
        for some 
        \[
            c' \geq \beta(p+1,\abs{Z_0},\abs{Z}) \geq \beta(q+1,\abs{Z_0}-1,\abs{Z})
        \] 
        and for each~${i \in [p+1]}$ there exists a set~$\mathcal{P}'_i$ of~${f_{\ref{lem:ramsey}}(\alpha(p+1,\abs{Z_0}),m)}$ pairwise vertex-disjoint $W''$\nobreakdash-handles such that 
        \begin{itemize}
            \item for each~${i \in [p]}$, the set~$\mathcal{P}'_i$ is a subset of the row-extension of~$\mathcal{P}_i$ to~$W''$ in~$W'$, 
            \item the paths in~${\bigcup_{i \in [p+1]} \mathcal{P}'_i}$ are pairwise vertex-disjoint, and
            \item the paths in~$\mathcal{P}'_{p+1}$ are $\lambda$-non-zero. 
        \end{itemize}
        Note that since~$W$ is ${(\gamma',Z,\psi(\abs{Z}+1)+2)}$-clean, 
        every $N^W$-path in~$W$ is $(\pi_j\circ\gamma)$-zero for all~${j \in Z}$
        and therefore if~$P'$ is the row-extension of a $W'$-handle~$P$ to~$W''$ in~$W'$, 
        then~${\pi_j(\gamma(P')) = \pi_j(\gamma(P))}$ for all~${j \in Z}$.

        Since~${\abs{Z} \leq m}$, 
        by Lemma~\ref{lem:ramsey}, there exist a subset~${\mathcal{R}}$ of~${\mathcal{P}'_{p+1}}$ and a subset~$Z'$ of~$Z$ such that
        \begin{itemize}
            \item ${\abs{\pi_j(\gamma(\mathcal{R}))} = \abs{\mathcal{R}}}$ for all~${j \in Z'}$,
            \item ${\abs{\pi_j(\gamma(\mathcal{R}))} = 1}$ for all~${j \in Z \setminus Z'}$, and
            \item ${\abs{\mathcal{R}} = \alpha(p+1,\abs{Z_0}) \geq \alpha(q+1,\abs{Z_0}-1)}$. 
        \end{itemize}
        Let~${p'' := p+1}$ and~${q'' := q}$ if~${Z' \cap Z_0}$ is empty and let~${p'' := q+1}$ and~${q'' := q+1}$ if~${Z' \cap Z_0}$ is nonempty, and 
        for~${i \in \{0\} \cup [p'']}$, let 
        \[
            Z''_i := 
            \begin{cases}
                Z_0 \setminus Z' & \text{if } i = 0,\\
                Z_i & \text{if } i \in [p''-1],\\
                Z' & \text{if } i = p''.\\
            \end{cases}
        \]
        For~${i \in [p''-1]}$, let~${\mathcal{P}_i'' := \mathcal{P}_i'}$ 
        and let~${\mathcal{P}''_{p''} := \mathcal{R}}$. 

        We now show that~${\big( W'', (\mathcal{P}''_i \colon i \in [p'']), (Z''_i \colon i \in \{0\} \cup [p'']) \big)}$ is a ${(c',q'',p'')}$-McGuffin; 
        if true, then since~${p'' \leq \widehat{p}}$, it 
        contradicts the maximality of~${(q,p)}$. 

        To observe property~\ref{item:mcguffin-1}, 
        note that~${\alpha(p,\abs{Z_0}) \geq \alpha(p+1,\abs{Z_0})}$, and
        if~${Z' \cap Z_0}$ is nonempty, then ${\alpha(p+1,\abs{Z_0}) \geq \alpha(q+1,\abs{Z_0\setminus Z'})}$. 
        If~$P'$ is the row-extension of a $W'$-handle~$P$ to~$W''$ in~$W'$, 
        then~${\pi_j(\gamma(P')) = \pi_j(\gamma(P))}$ for all~${j \in Z}$, 
        implying 
        properties~\ref{item:mcguffin-2} and~\ref{item:mcguffin-3} 
        for~${i < p''}$. 
        By the definition of~$Z'$, properties~\ref{item:mcguffin-2} and~\ref{item:mcguffin-3} hold for~${i = p''}$. 
        Property~\ref{item:mcguffin-4} holds trivially.
        If~${Z' \cap Z_0 = \emptyset}$, then~${Z_i'' = Z_i}$ for all~${i \in \{0\} \cup [q]}$, and therefore property~\ref{item:mcguffin-5} holds.
        When~${Z' \cap Z_0 \neq \emptyset}$, property~\ref{item:mcguffin-5} holds
        because~${Z_0 \cap \bigcup_{i \in [q''-1]} Z_i = \emptyset}$ by property~\ref{item:mcguffin-4} and every element of~${Z' \cap Z_0}$ is not contained in~${\bigcup_{i \in [q''-1]} Z_i}$. 
        It remains to check property~\ref{item:mcguffin-6} when~${Z' \cap Z_0}$ is empty, ${q < i \leq p}$, and~${i' = p'' = p + 1}$. 
        This is implied by the property that 
        the paths in~$\mathcal{P}'_{p+1}$ are $\lambda$-non-zero.
        We conclude that~${(W',\mathfrak{P},\mathcal{Z})}$ satisfies~\ref{item:mcguffin-ast}.
        
        If~${p = 0}$, then 
        by property~\ref{item:mcguffin-ast}, property~\ref{item:manyhandles-d} holds with~$\mathcal{Z}$ and~${I := \emptyset}$, 
        and properties~\ref{item:manyhandles-a}, \ref{item:manyhandles-b}, \ref{item:manyhandles-c}, \ref{item:manyhandles-e}, and~\ref{item:manyhandles-f} hold vacuously. 

        Therefore, we may assume that~${0 < p < \rho(\abs{Z_0})}$. 
        Let~${I' := [p]}$. 
        Since 
        \[
            \abs{\mathcal{P}_i}\geq \alpha(p,\abs{Z_0})\geq f_{\ref{lem:handlebars}}\Bigl(p,f_{\ref{lem:sidon}}\bigl(2^{m\omega + 1} f_{\ref{lem:alltransversals}}(m,\omega,\maxk,p),m\bigr)\Bigr)
        \] 
        for each~${i \in [p]}$, 
        by Lemma~\ref{lem:handlebars} and property~\ref{item:mcguffin-1}, 
        there is a family~${(\mathcal{P}^\ast_i \subseteq \mathcal{P}_i \colon i \in [p])}$ of pairwise vertex-disjoint non-mixing $W'$-handlebars, each of size~${f_{\ref{lem:sidon}}(2^{m\omega + 1} \cdot f_{\ref{lem:alltransversals}}(m,\omega,\maxk,p),m)}$. 
        By applying Lemma~\ref{lem:sidon} to the restriction of~${\gamma(\mathcal{P}_i^\ast)}$ to~${\prod_{j \in Z_i} \Gamma_j}$ for each~${i \in [p]}$, 
        we deduce that there is a family of subsets~${(\mathcal{P}'_i \subseteq \mathcal{P}^{\ast}_i \colon i \in [p])}$, each of size~${2^{m\omega+1} f_{\ref{lem:alltransversals}}(m,\omega,\maxk,p)}$, 
        satisfying properties~\ref{item:manyhandles-a} and~\ref{item:manyhandles-f} with the set~${I' = [p]}$. 
        They also satisfy properties~\ref{item:manyhandles-b} and~\ref{item:manyhandles-c} with an arbitrary family~${(g_i \in \gamma(\mathcal{P}'_i) \colon i \in [p])}$, by~\ref{item:mcguffin-2} and~\ref{item:mcguffin-3}. 
        Observe that properties~\ref{item:manyhandles-a}, \ref{item:manyhandles-b}, \ref{item:manyhandles-c}, and~\ref{item:manyhandles-f} hold for any subset~$I$ of~$I'$ (and the corresponding subfamilies~${(\mathcal{P}'_i \colon i \in I)}$, ${(Z_i \colon i \in I)}$, and~${(g_i \colon i \in I)}$) because we may assume that~$f_{\ref{lem:alltransversals}}$ is increasing in its fourth argument. 
        Now property~\ref{item:manyhandles-d} holds for~$I'$ by property~\ref{item:mcguffin-ast}, so taking a minimal subset~$I$ of~$I'$ satisfying property~\ref{item:manyhandles-d}, we have that property~\ref{item:manyhandles-e} is also satisfied by~$I$. 
    \end{proofofclaim}
    
    The next claim is used to find an obstruction. 
    The first condition makes sure that for all choices of one path from each handlebar, the sum of their $\gamma$-values is allowable.
    The second condition makes sure that 
    the sum of $\gamma$-values of some paths in the handlebars is allowable only if each handlebar contributes at least one path to the sum.
    The third condition requires that the sum of $\gamma$-values of some paths in the handlebars is allowable only if each handlebar not in series 
    contributes an odd number of handles to the sum.
    \begin{claim}
        \label{clm:almostobstruction}
        Let~$W''$ be a ${(c-2)}$-column-slice of~$W'$ containing $C_2^{W'}$ and~${C_{c-1}^{W'}}$.
        Then there is 
        a family~${( \mathcal{P}''_i \colon i \in I )}$ of pairwise vertex-disjoint non-mixing $W''$-handlebars, each of size~$\maxk$, such that 
        \begin{enumerate}
            [label=(\alph*)]
            \item \label{item:almostobstruction-1} 
                for each~${j \in Z}$ and each~${(P_i \colon i \in I)}$ with ${P_i \in \mathcal{P}''_i}$ for all~${i \in I}$, we have~${\sum_{i \in I} \pi_j(\gamma(P_i)) \notin \Omega_j}$, 
            \item \label{item:almostobstruction-2} 
                for each~${i \in I}$ and each~${g \in \gen{ \gamma(P) \colon P \in \bigcup_{i' \in I \setminus \{i\}} \mathcal{P}''_{i'}}}$, there is~${j \in Z}$ such that$~{\pi_j(g) \in \Omega_j}$,
            \item \label{item:almostobstruction-3} 
                for each ${y \in I}$ such that~$\mathcal{P}''_{y}$ is not in series and every function~${f \colon \bigcup_{i \in I} \mathcal{P}''_{i} \to \mathbb{Z}}$ for which ${\sum_{P \in \mathcal{P}''_{y}} f(P)}$ is even, there is some~${j \in Z}$ such that~${\sum_{i \in I} \sum_{P \in \mathcal{P}''_{i}} f(P) \pi_j(\gamma(P)) \in \Omega_j}$.
        \end{enumerate}
    \end{claim}
    
    \begin{proofofclaim}
        If~${I = \emptyset}$, then~${0 \notin \Omega_j}$ for all~${j \in Z}$ by Claim~\ref{clm:manyhandles}\ref{item:manyhandles-d} and therefore this claim is trivially true. 
        Thus we may assume that~${I \neq \emptyset}$. 
        Let~$S$ be a maximal subset of~$I$ such that 
        \[ 
            \bigl\langle\{ 2g_i \colon i \in S\} \cup \{g_i \colon i \in I \setminus S\}\bigr\rangle 
            \cap 
            \bigcap_{j \in Z \setminus \bigcup_{i \in I} Z_i}\
            \pi_j^{-1}(\Gamma_j\setminus \Omega_j)\neq\emptyset.
        \]
        Note that such a set~$S$ exists, since  Claim~\ref{clm:manyhandles}\ref{item:manyhandles-d} implies that the empty set satisfies this condition. 
        By Lemma~\ref{lem:omega-avoiding}, there exist a family of integers~${(d_i \colon i \in I)}$ such that~${d_i \in [2^{m\omega+1}]}$ is even for each~${i \in S}$, that~${d_i \in [2^{m\omega}]}$ for each~${i \in I \setminus S}$, 
        and that~${\pi_j\left(\sum_{i \in I} d_i g_i \right) \notin \Omega_j}$ for all~${j \in Z \setminus \bigcup_{i \in I} Z_i}$. 
        Now~$d_i$ is odd for all~${i \in I \setminus S}$ by the choice of~$S$.
        By Lemma~\ref{lem:combining-handles}, there is a family ${(\mathcal{P}^\ast_i \colon i \in I)}$ of pairwise vertex-disjoint non-mixing $W''$-handlebars each of size~${f_{\ref{lem:alltransversals}}(m,\omega,\maxk,\abs{I})}$ such that for each~${i \in I}$ and~${Q \in \mathcal{P}^\ast_i}$, there is a set~${\{ P_{\ell,Q} \in \mathcal{P}_i \colon \ell \in [d_i] \}}$ of size~$d_i$ satisfying the following three properties:
        \begin{itemize}
            \item ${\bigcup_{\ell=1}^{d_i}  P_{\ell,Q} \subseteq Q \subseteq W' \cup \bigcup_{\ell=1}^{d_i} P_{\ell,Q}}$.
            \item $\mathcal{P}^\ast_i$ is in series for each~${i \in S}$.
            \item $\mathcal{P}^\ast_i$ is of the same type as~$\mathcal{P}'_i$ for each~${i \in I \setminus S}$.
        \end{itemize}
        Note that ${\pi_j(\gamma(\mathcal{P}^\ast_i)) = \{d_i \pi_j(g_i)\}}$ for all~${i \in I}$ and ${j \in Z \setminus Z_i}$
        by Claim~\ref{clm:manyhandles}\ref{item:manyhandles-c},
        as well as that ${\abs{\pi_j(\gamma(\mathcal{P}^\ast_i))} = \abs{\mathcal{P}^\ast_i}}$ for all~${i \in I}$ and~${j \in Z_i}$ by Claim~\ref{clm:manyhandles}\ref{item:manyhandles-f}. 

        Since~${\abs{\mathcal{P}^\ast_i} = f_{\ref{lem:alltransversals}}(m,\omega,\maxk,\abs{I}) > m\omega\geq \abs{\bigcup_{i\in I}Z_i}\,\omega}$ for each~${i \in I}$, 
        by Lemma~\ref{lem:vectorsum}, there is a family ${(g'_i \colon i \in I)}$ of elements of~$\Gamma$ such that 
        \begin{enumerate}
            \item ${g'_i \in \gamma(\mathcal{P}^\ast_i)}$ for each~${i \in I}$ and 
            \item ${\pi_j \left( \sum_{i \in I} g'_i \right) \notin \Omega_j}$ for all~${j \in \bigcup_{i \in I} Z_i}$. 
        \end{enumerate}
        By Lemma~\ref{lem:alltransversals}, for each~${i \in I}$ there is a subset~$\mathcal{P}''_i$ of~$\mathcal{P}^\ast_i$ of size~$\maxk$ such that~${(\mathcal{P}''_i \colon i \in I)}$ satisfies property~\ref{item:almostobstruction-1}. 
        Now ${(\mathcal{P}''_i \colon i \in I)}$ satisfies property~\ref{item:almostobstruction-2} by Claim~\ref{clm:manyhandles}\ref{item:manyhandles-e}.
        
        To prove property~\ref{item:almostobstruction-3}, suppose that~$\mathcal{P}''_y$ is not in series for some~${y \in I}$, and~${f \colon \bigcup_{i \in I} \mathcal{P}''_i \to \mathbb{Z}}$ is a function such that~${\sum_{P \in \mathcal{P}''_y} f(P)}$ is even. 
        Since~$\mathcal{P}''_y$ is not in series, we have~${y \in I \setminus S}$. 
        By Claim~\ref{clm:manyhandles}\ref{item:manyhandles-c}, 
        \[
            \pi_j(\gamma(P))=d_i\pi_j(g_i) 
            \text{ for all~${i \in I}$, all~${P \in \mathcal{P}''_i}$, and all~${j \in Z \setminus \bigcup_{i' \in I} Z_{i'}}$.}
        \]
        In particular, if~$d_i$ is even or~${\sum_{P \in \mathcal P_i''} f(P)}$ is even, then~${\sum_{P \in \mathcal{P}''_i} f(P) \pi_j(\gamma(P)) \in \pi_j(\gen{2g_i})}$. 
        
        Let~${S' = S \cup \{y\}}$. 
        Then for all~${i \in S'}$, either~$d_i$ or~${\sum_{P \in \mathcal P_i''} f(P)}$ is even. 
        Let 
        \[
            {g = \sum_{i \in I} \sum_{P \in \mathcal{P}''_i} f(P) \gamma(P)}.
        \]
        Then there exists~${g' \in \Gamma}$ such that~${\pi_j(g) = \pi_j(g')}$ for all~${j \in Z \setminus \bigcup_{i' \in I} Z_{i'}}$ and
        \[
            {g' \in \gen{\{ 2g_i \colon i \in S'\} \cup \{g_i \colon i \in I \setminus S'\}}}.
        \]
        By the maximality of~$S$, we have that~${g' \notin \bigcap_{j \in Z \setminus \bigcup_{i' \in I} Z_{i'}} \pi_j^{-1}(\Gamma_j \setminus \Omega_j)}$. 
        Therefore, there exists some ${j \in Z \setminus \bigcup_{i' \in I} Z_{i'}}$ such that~${\pi_j(g') = \pi_j(g) \in \Omega_j}$. 
        This proves property~\ref{item:almostobstruction-3}. 
    \end{proofofclaim}
    
    Let~$H$ be the union of~$W''$ and~${\bigcup \{ \bigcup \mathcal{P}''_i \colon i \in I\}}$. 
    Note that~$W''$ has at least~${c_2(\abs{Z})}$ columns and at least~${r_2(\abs{Z})}$ rows and therefore the order of~$W''$ is greater than or equal to~$\theta$. 
    
    We now find a half-integral packing in a similar manner as in the proof of~\cite[Theorem~1]{GollinHKKO2021}. 
    
    \begin{claim}
        \label{clm:halfintegral}
        $H$ contains a half-integral packing of~$\maxk$ cycles in~$\mathcal{O}$. 
        Moreover, if~${I = \emptyset}$, then~$H$ contains a packing of~$\maxk$ cycles in~$\mathcal{O}$. 
    \end{claim}
    
    \begin{proofofclaim}
        Since~${\abs{\mathcal{P}''_i} = \maxk}$ for each~${i \in I}$, 
        there exists a family~${( \mathcal{Q}_x \subseteq \bigcup_{i \in I} \mathcal{P}''_i \colon x \in [\maxk] )}$ of pairwise disjoint sets such that~${\abs{\mathcal{Q}_x \cap \mathcal{P}''_i} = 1}$ for all~${i \in I}$ and~${x \in [\maxk]}$. 
        Note that if~${I = \emptyset}$, then~${\mathcal Q_x = \emptyset}$ for all~${x \in [\maxk]}$. 
        By Claim~\ref{clm:almostobstruction}\ref{item:almostobstruction-1}, for each~${x \in [\maxk]}$ and~${j \in Z}$, we have~${\sum_{P \in \mathcal{Q}_x} \gamma_j(P) \notin \Omega_j}$. 
        We remark that if~${I = \emptyset}$, then~${0 \notin \Omega_j}$ for all~${j \in Z}$. 
        
        Since~$W''$ has at least~$c_2(\abs{Z})$ columns and~${c_2(\abs{Z}) \geq \maxk c_{\ref{lem:omega-avoiding-cycle}}(\widehat{p}, \psi(\abs{Z}+1)+2, m, \omega)}$, 
        there exists a set~${\{ W_x \colon x \in [\maxk] \}}$ of~$\maxk$ pairwise vertex-disjoint 
        ${c_{\ref{lem:omega-avoiding-cycle}}(\widehat{p}, \psi(\abs{Z}+1)+2, m, \omega)}$-column-slices of~$W''$. 
        Note that~$W''$ has at least~${r_{\ref{lem:omega-avoiding-cycle}}(\widehat{p}, \psi(\abs{Z}+1)+2, m, \omega)}$ rows.  
        For each~${x \in [\maxk]}$, let~$\mathcal{Q}^\ast_x$ be the row-extension of~$\mathcal{Q}_x$ to~$W_x$. 
        Note that if~${I = \emptyset}$, then~$\mathcal{Q}^\ast_x$ is also empty for each~${x \in [\maxk]}$.
        Since~${\abs{I} \leq \widehat{p}}$, 
        by Lemma~\ref{lem:omega-avoiding-cycle}, for each~${x \in [\maxk]}$, there is a cycle~$O_x$ in~${W_x \cup \bigcup \mathcal{Q}^\ast_x}$ such that~${\gamma_j(O_x) \notin \Omega_j}$ for all~${j \in [m]}$. 
        Observe that no vertex is in more than two of the subgraphs in $\{W_x \cup \bigcup \mathcal{Q}^\ast_x \colon x\in [\maxk]\}$ and therefore no vertex is in more than two of the cycles in $\{O_x \colon x\in [\maxk]\}$.
        Moreover, if~${I = \emptyset}$, then~$O_x$ is contained in~$W_x$ for each~${x \in [\maxk]}$, and therefore~$H$ contains a packing of~$\maxk$ cycles in~$\mathcal{O}$.
    \end{proofofclaim}
    
    By Claim~\ref{clm:halfintegral}, 
    $I$ is nonempty because we assumed that~${\nu(G) < k \leq \maxk}$. 
    Let~${J := [m] \setminus Z}$ and let~$\gamma''$ be the $\left(\Gamma / \Gamma_{J} \right)$-labelling induced by the restriction of~$\gamma$ to~$H$. 
    Since we assumed that statement~\ref{item:main-obstruction} fails
    and~$H$ has a half-integral packing of~$\kappa$ cycles in~$\mathcal{O}$ by Claim~\ref{clm:halfintegral}, 
    we have that~${(H,\gamma'') \notin \mathcal{C}(\kappa, \theta, \Gamma / \Gamma_{J}, A + \Gamma_{J} )}$. 
    We will find the desired contradiction by constructing in~$H$ a packing of~$k$ cycles in~$\mathcal{O}$. 
    
    Recall the properties~\ref{item:obstructions-union}--\ref{item:obstructions-handlebars} of~${\mathcal{C}(\kappa, \theta, \Gamma / \Gamma_{J}, A + \Gamma_{J} )}$ in Definition \ref{def:obstructions}. 
    By definition,~$(H,\gamma'')$ satisfies~\ref{item:obstructions-union}, \ref{item:obstructions-zerowall}, \ref{item:obstructions-allowabletransversals}, \ref{item:obstructions-minimality}, and~\ref{item:obstructions-even} 
    by Claim~\ref{clm:almostobstruction}.
    Thus, \ref{item:obstructions-handlebars} fails to hold\footnotemark.
    \footnotetext{Note that the number of handlebars in~$H$ is at most~${\widehat{p}}$, which depends only on~$m$ and~$\omega$. 
    We could strengthen Theorem~\ref{thm:mainobstruction} by imposing this additional restriction on the class~${\mathcal{C}(\kappa,\theta,\Gamma,A)}$, 
    which would give us that there are only finitely many types of obstructions to consider in terms of the arrangement of the handlebars around the wall.}

    First consider the case that~${( \mathcal{P}''_i \colon i \in I)}$ contains a $W''$-handlebar that is not in series.
    Then ${( \mathcal{P}''_i \colon i \in I)}$ contains an even number of crossing $W''$-handlebars and no $W''$-handlebar that is in series
    because~\ref{subitem:obstructions-oddcrossing} and~\ref{subitem:obstructions-seriesnonseries} fail to hold respectively. 
    By Lemma~\ref{lem:combining-handlebars}, there exist an $N^W$\nobreakdash-anchored subwall~$W^{\ast}$ of~$W''$ 
    with at least~${c_1(\abs{Z})}$ columns and at least~${r_1(\abs{Z})}$ rows 
    as well as a nested $W^{\ast}$\nobreakdash-handlebar~$\mathcal{Q}_1$ of size~$k$ such that~${\gamma_j(Q) \notin \Omega_j}$ for all~${Q \in \mathcal{Q}_1}$ and ${j \in Z}$. 
    Let us define~${q := 1}$. 
    
    If the first case does not hold, then all $W''$-handlebars in~${(\mathcal{P}''_i \colon i \in I)}$ are in series. 
    Let ${q := \abs{I}}$ and observe that~${q \in \{0,1,2\}}$ because~\ref{subitem:obstructions-3series} fails to hold. 
    For each~${i \in [q]}$, let us define~$\mathcal{Q}_i$ to be~$\mathcal{P}''_j$ for the $i$-th element~$j$ of~$I$. 
    Let~${W^{\ast} := W''}$. 
    Note that~${c_2(\abs{Z}) \geq c_1(\abs{Z})}$ and~${r_2(\abs{Z}) \geq r_1(\abs{Z})}$. 
    
    In either case, we can apply Lemma~\ref{lem:almost-finding-cycles} to obtain, for each~${x \in [k]}$, an $N^W$-anchored \linebreak ${(c_0(\abs{Z}),r_0(\abs{Z}))}$\nobreakdash-subwall~$W_x$ of~$W^{\ast}$ 
    and a set~${\mathcal{H}_{x} = \{H_{x,i} \colon i \in [q]\}}$ of~$q$ pairwise vertex-disjoint $W_x$-handles 
    such that 
    \begin{itemize}
        \item for distinct~${x,x' \in [k]}$, the graphs ${W_x \cup \bigcup \mathcal{H}_{x}}$ and ${W_{x'} \cup \bigcup \mathcal{H}_{x'}}$ are vertex-disjoint and 
        \item $\sum_{i \in [q]} \gamma_j(H_{x,i}) \notin \Omega_j$ for each~${x \in [k]}$ and~${j \in Z}$. 
    \end{itemize}
    Finally, we apply Lemma~\ref{lem:omega-avoiding-cycle} to obtain a packing of~$k$ cycles in~$\mathcal{O}$. 
    This contradiction completes the proof. 
\end{proof}

\section{Applications and Discussions}
\label{sec:applications}

\subsection{The obstructions have no packing of three allowable cycles}
\label{subsec:reallyobstructions}

We now demonstrate that the graphs described in Definition~\ref{def:obstructions} do not contain three vertex-disjoint cycles with values in~$A$. 
Recall that in the third statement of Theorem~\ref{thm:mainobstruction}, we find obstructions with an additional property of containing a large half-integral packing of allowable cycles, and therefore they admit no small hitting set for the allowable cycles. 
Thus, these obstructions really do form counterexamples to Erd\H{o}s-P\'{o}sa type duality claims.

\begin{proposition}
    \label{prop:obstructions}
    Let~$\kappa$ and~$\theta$ be positive integers, let~$\Gamma$ be an abelian group, and let~${A \subseteq \Gamma}$. 
    If ${(G,\gamma) \in \mathcal{C}(\kappa,\theta,\Gamma,A)}$,
    then $G$ has no three vertex-disjoint cycles whose $\gamma$-values are in~$A$, and
    if~$G$ has two vertex-disjoint cycles whose $\gamma$-values are in~$A$, then~${(G,\gamma)}$ satisfies property~\ref{subitem:obstructions-seriesnonseries} of Definition~\ref{def:obstructions}. 
\end{proposition}

\begin{proof}
    Let~$W$ be the wall and let~${\mathfrak{P} = (\mathcal{P}_i \colon i\in [t])}$ be the family of $W$-handlebars described in Definition~\ref{def:obstructions}, and let~${o := 3}$ if property~\ref{subitem:obstructions-seriesnonseries} holds and let~${o := 2}$ otherwise.
    Suppose that~$G$ has a set~${\mathcal{O} = \{O_i \colon i \in [o]\}}$ of~$o$ pairwise vertex-disjoint cycles whose $\gamma$-values are in~$A$. 
    
    Suppose first that property~\ref{subitem:obstructions-oddcrossing} or~\ref{subitem:obstructions-seriesnonseries} holds. 
    Let~$n$ be the number of nested $W$-handlebars in~$\mathfrak{P}$, 
    let~$x$ be the number of crossing $W$-handlebars in~$\mathfrak{P}$, 
    and let~$s$ be the number of $W$\nobreakdash-handlebars in~$\mathfrak{P}$ that are in series. 
    Note that~${n + x \geq 1}$ because property~\ref{subitem:obstructions-oddcrossing} or~\ref{subitem:obstructions-seriesnonseries} holds. 
    By rearranging indices, we may assume that~$\mathcal{P}_i$ is nested for all~${i \in [n]}$, crossing for all~${i \in [n+x] \setminus [n]}$, and in series for all~${i \in [n+x+s] \setminus [n+x]}$. 
    We remark that the ordering of $(\mathcal P_i:i\in [t])$ is not related to their locations relative to the wall and our choice of indices after rearrangement is purely for the convenience.

\begin{figure}%
    \centering
    \begin{tikzpicture}

  \draw[thick] (0,0) circle (2.5cm);
  \draw[thick, blue] (60:2.6cm) arc[start angle=60, end angle=120, radius=2.6cm];
\node at (0,2.8) {$\arc(\alpha, \beta)$};

  \draw[dashed] (0,0) -- (0:2.5cm);
  \draw[dashed] (0,0) -- (60:2.5cm);
  \draw[dashed] (0,0) -- (120:2.5cm);

  \filldraw[black] (0,0) circle (2pt) node[below left] {O};

  \draw[blue] (60:2.6cm) circle (2pt);
  \draw[blue] (120:2.6cm) circle (2pt);
  
  \draw (1.25,0) arc[start angle=0, end angle=60, radius=1.25cm];
  \draw (1.5,0) arc[start angle=0, end angle=120, radius=1.5cm];
  \node at (0.7,0.3) {$\frac{\alpha}{n+x}\pi$};
\node at (0,1.1) {$\frac{\beta}{n+x}\pi$};

\end{tikzpicture}
    \caption{An illustration of $\arc(\alpha, \beta)$. }
    \label{fig:arc}
\end{figure}

    Consider the complex closed unit disc~${D := \{ z \in \mathbb{C} \colon \abs{z} \leq 1\}}$ and let~$S$ be the complex unit circle ${\{ z \in \mathbb{C} \colon \abs{z} = 1 \}}$. 
    Let ${\xi := e^{i\pi/(n+x)}}$ and for~${\alpha, \beta \in \mathbb{R}}$,
    let~${\arc(\alpha,\beta)}$ be the open arc~${\{ \xi^\gamma  \colon \alpha < \gamma < \beta \}}$ in~$S$, see Figure~\ref{fig:arc}.
    We now form a surface in which~$G$ embeds by identifying a pair of closed arcs in~$S$ for each nested or crossing handlebar in $\mathfrak{P}$.
    For each~${j \in [n+x]}$, let~$P_j$ be a $W$-handle in~$\mathcal{P}_j$ and let~${\{v_{\ell} \colon \ell \in [2n+2x]\}}$ be the set of endvertices of paths in~${\{P_j \colon j \in [n+x]\}}$, where~${v_{\ell} \prec_W v_k}$ if and only if~${k < \ell}$. 
    Let~$f$ and~$g$ be injective maps from~${[n+x]}$ to~${[2n+2x]}$ such that for all~${j \in [n+x]}$, the endvertices of~$P_j$ are~$v_{f(j)}$ and~$v_{g(j)}$, and~${f(j) < g(j)}$. 
    
    Let~$\sim$ be the equivalence relation on~$S$ obtained by taking the transitive closure with respect to the following properties; 
    \begin{itemize}
        \item \label{item:eqrel1} ${\xi^{f(j)+\alpha} \sim \xi^{g(j)+1-\alpha}}$ for each ${j \in [n]}$ and each~${\alpha \in [0,1]}$, 
        \item \label{item:eqrel2} ${\xi^{f(j)+\alpha} \sim \xi^{g(j)+\alpha}}$ for each~${j \in [n+x] \setminus [n]}$ and each~${\alpha \in [0,1]}$. 
    \end{itemize}
    Finally, let~$\mathbb{S}$ be the surface~${D / \sim}$.
    If~${\abs{z} < 1}$, then in~$\mathbb{S}$, we have that~$z$ is not identified with any other point of~$D$ and therefore we write~$z$ to denote the equivalence class~${\{z\}}$ in~$\mathbb{S}$ when~${\abs{z} < 1}$ for convenience. 
    Let~${D^\ast := \{ z \in \mathbb{C} \colon  \abs{z} < 1 \}\subseteq \mathbb{S}}$, and let~$S^\ast$ be the complement of~$D^\ast$ in~$\mathbb{S}$. 
    
    There is an embedding~$\phi$ of~$G$ in~$\mathbb{S}$ such that 
    \begin{enumerate}[label=(\roman*)]
        \item \label{item:embedding1} ${W \cup \bigcup \{ \bigcup \mathcal{P}_j \colon j \in [n + x + s] \setminus [n + x] \}}$ is embedded in~$D^\ast$,
        \item \label{item:embedding2} for each~${j \in [n+x]}$ and each~${P \in \mathcal{P}_j}$, the subset of~$D$ corresponding to~${\phi(P)}$ is the union of two curves of positive length, each of which intersects~$S$ exactly once, at equivalent points in the arcs~${\arc(f(j),f(j)+1)}$ and~${\arc(g(j),g(j)+1)}$, and
        \item \label{item:embedding3} for each~${j \in [n+x+s] \setminus [n+x]}$, there is a component of~${\mathbb{S} \setminus \phi(G)}$ whose boundary in~$\mathbb{S}$ contains~$\phi(\bigcup\mathcal{P}_j)$. 
    \end{enumerate}
    
    For each~${j \in [o]}$ and~${k \in [2n+2x]}$, let~$X_{j,k}$ be the set of points in~${\arc(k,k+1)}$ corresponding to points in~${\phi(O_j) \cap S^\ast}$ and let~${X_j := \bigcup_{k \in [2n+2x]} X_{j,k}}$. 
    By~\ref{item:embedding2}, 
    $X_{j,k}$ is a finite set.
    Note that the elements of~${\{ X_i \colon i \in [o] \}}$ are pairwise disjoint since~$\mathcal{O}$ is a set of pairwise vertex-disjoint cycles.
    Also, property~\ref{item:obstructions-even} implies that~$\abs{X_{j,k}}$ is odd for each~${j \in [o]}$ and~${k \in [2n+2x]}$. 
    This implies that for all~${k \in [2n+2x]}$, we have~${\abs{X_{1,k} \cup X_{2,k}}}$ (and hence~${\abs{X_1 \cup X_2}}$) is even.
    
    Let~${\{ z_j \colon j \in [\abs{X_1 \cup X_2}]\}}$ be the enumeration of~${X_1 \cup X_2}$ such that 
    if ${z_j = \xi^{\alpha}}$ and~${z_k = \xi^{\beta}}$ for some $j,k\in [\abs{X_1 \cup X_2}]$ and ${\alpha, \beta \in \mathbb{R}}$ with~${0 < \alpha < \beta < 2n+2x}$, then ${j < k}$. 
    For~${j \in [2]}$, let~$M_j$ be the subset of~$D$ corresponding to~${\phi(O_j)}$. 
    Each component of~${M_1 \cup M_2}$ is a curve~$C$ which separates~$D$ and contains exactly two points in~${X_1 \cup X_2}$. 
    It follows that each component of~${D \setminus C}$ contains an even number of points in~${X_1 \cup X_2}$, and hence~$C$ contains exactly one point in each of~${Z_1 := \{z_{2j-1}\colon j\in [\frac{1}{2}\abs{X_1 \cup X_2}]\}}$ and~${Z_2 := \{z_{2j}\colon j\in [\frac{1}{2}\abs{X_1 \cup X_2}]\}}$. 
    Thus,~${\abs{X_1 \cap Z_1} = \abs{X_1 \cap Z_2}}$. 
    
    Let~${j \in [n+x]}$ and~${k,\ell \in [\abs{X_1 \cup X_2}]}$ be such that~$z_k$ and~$z_\ell$ are equivalent and are contained in~${\arc(f(j),f(j)+1)}$ and~${\arc(g(j),g(j)+1)}$ respectively. 
    Recall that ${\abs{X_{1,a} \cup X_{2,a}}}$ is even for every $a\in [2n+2x]$.
    For each~${a,b \in [\abs{X_1 \cup X_2}]}$ such that ${z_a \in X_{1,f(j)} \cup X_{2,f(j)} \setminus \{z_k\}}$ and~$z_b$ is the point in ${X_{1,g(j)} \cup X_{2,g(j)}}$ equivalent to~$z_a$, we have that ${\abs{\{z_a,z_b\} \cap \arc(z_k,z_{\ell})}}$ is even if and only if~${j \in [n]}$. 
    Thus, if~${j \in [n]}$, then~${\ell-k}$ is odd and ${\abs{(X_{1,f(j)} \cup X_{1,g(j)}) \cap Z_1} = \abs{(X_{1,f(j)} \cup X_{1,g(j)}) \cap Z_2}}$, 
    and if~${j \in [n+x] \setminus [n]}$, then~${\ell-k}$ is even and ${\abs{(X_{1,f(j)} \cup X_{1,g(j)}) \cap Z_1} - \abs{(X_{1,f(j)} \cup X_{1,g(j)}) \cap Z_2}}$ is congruent to~$2$ modulo~$4$. 
    Now, 
    \begin{align*}
        0 &= \abs{X_1\cap Z_1} - \abs{X_1\cap Z_2}\\
        &= \sum_{j=1}^{n+x} \Big( \abs{X_{1,f(j)}\cup X_{1,g(j)}\cap Z_1}-\abs{X_{1,f(j)}\cup X_{1,g(j)}\cap Z_2} \Big)\\
        &= \sum_{j=n+1}^{n+x} \Big( \abs{X_{1,f(j)}\cup X_{1,g(j)}\cap Z_1}-\abs{X_{1,f(j)}\cup X_{1,g(j)}\cap Z_2} \Big)
    \end{align*}
    and therefore~$x$ is even. 
    Hence, we may assume that property~\ref{subitem:obstructions-seriesnonseries} holds, and so ${o = 3}$. 
    
    For a set~$\mathcal{O}'$ of pairwise vertex-disjoint cycles in~$G$, we define an auxiliary multigraph~${H(\mathcal{O}')}$ whose vertex set is the set of all components of~${\mathbb{S} \setminus \bigcup \{ \phi(O) \colon O \in \mathcal{O}' \}}$ 
    where for each~${O \in \mathcal{O}'}$, there is an edge~$e_O$ between the components that contain~$O$ in their boundary. 
    We remark that if there is only one component whose boundary contains~$O$, then~$e_O$ is a loop. 
    
    \begin{claim*}
        \label{clm:noloops}
        If~$\mathcal{O}'$ is a subset of~$\mathcal{O}$ of size at least~$2$, 
        then the graph~${H(\mathcal{O}')}$ has no loop. 
    \end{claim*}
    
    \begin{proof}
        Without loss of generality, we may assume that~${\{O_1,O_2\} \subseteq \mathcal{O}'}$, and it is sufficient to prove the claim when~${\{O_1,O_2\} = \mathcal{O}'}$. 
        Recall that~$M_j$ is the subset of~$D$ corresponding to~$\phi(O_j)$ for~$j\in[2]$.
        
        Let us $2$-colour the components of~$D\setminus (M_1\cup M_2)$ so that every point in~$M_1\cup M_2$ is on the boundary of two components of different colours (this is equivalent to 2-colouring the bounded faces of~$S\cup M_1\cup M_2$ seen as an outerplanar graph).
        We claim that the components of~$D\setminus (M_1\cup M_2)$ corresponding to one component of $\mathbb{S}\setminus \phi(O_1\cup O_2)$ receive the same colour.
        It suffices to show that any two points on $S\setminus (M_1\cup M_2)$ that are identified in $\mathbb{S}$ receive the same colour; that is, there is an even number of points of~$X_1\cup X_2$ contained in an arc of~$S$ between the two identified points.
        
        Recall that~$\abs{X_{j,k}}$ is odd for each~${j \in [2]}$ and~${k \in [2n+2x]}$, hence ${\abs{X_{1,k} \cup X_{2,k}}}$ is even for all~${k\in[2n+2x]}$. 
        Note that two identified points on $S\setminus (M_1\cup M_2)$ can be written either as $\xi^{f(j)+\alpha}, \xi^{g(j)+1-\alpha}$ for some $j\in[n]$ and $\alpha\in[0,1]$, or as $\xi^{f(j)+\alpha},\xi^{g(j)+\alpha}$ for some $j\in[n+x]\setminus[n]$ and $\alpha\in[0,1]$.
        
        In the first case, let $A_{j,\alpha} := {\arc(f(j)+\alpha, g(j)+1-\alpha)}$ and note that~${A_{j,\alpha} \setminus \{\xi^{k} \colon k \in [2(n+x)] \}}$ is the disjoint union 
        \[
            {\arc(f(j)+\alpha,f(j)+1)} 
            \cup 
            {\arc(g(j),g(j)+1-\alpha)} 
            \cup 
            \bigcup \{ \arc(k,k+1) \colon k \in [g(j)-1] \setminus [f(j)] \}.
        \]
        Observe that~${\abs{(X_1\cup X_2) \cap A_{j,\alpha}}}$ is even, because
        \[
            {\abs{(X_1 \cup X_2) \cap \arc(f(j)+\alpha,f(j)+1)} = \abs{(X_1 \cup X_2)\cap \arc(g(j),g(j)+1-\alpha)}}. 
        \]
        Similarly, in the second case, we have that $\abs{(X_1 \cup X_2) \cap \arc(f(j)+\alpha, g(j)+\alpha)}$ is even because the assumption that ${\xi^{f(j)+\alpha} \notin X_1 \cup X_2}$ implies
        \[
            \abs{X_k \cap (\arc(f(j)+\alpha,f(j)+1) \cup \arc(g(j),g(j)+\alpha))}=\abs{X_{k,f(j)}}
        \]
        for each~${k \in [2]}$.

        \ifx
        It follows that we can $2$-colour the components of~${S^\ast \setminus \phi(O_1\cup O_2)}$ such that every point in $\phi(O_1\cup O_2)\cap S^\ast$ is on the boundary of two components of different colours. 
        Now consider a curve~$C$ in~${\mathbb{S}\setminus \phi(O_1\cup O_2)}$ between two points in~$S^\ast$ whose interior is entirely in~$D^\ast$. 
        Let~$X$ be the set of points in~$D$ corresponding to points in~${\phi(O_1 \cup O_2)}$ and let~$Y$ be the set of points in~$D$ corresponding to points in~$C$. 
        Now each component of~$X$ contains exactly two points in~${X_1 \cup X_2}$, so each component of~${S \setminus Y}$ contains an even number of points in~${X_1 \cup X_2}$. 
        It follows that the endpoints of~$C$ are in components of~${S^\ast \setminus \phi(O_1\cup O_2)}$ with the same colour. 
        Now for every component~$Z$ of~${\mathbb{S} \setminus \phi(O_1\cup O_2)}$, the components of~${S^\ast \setminus \phi(O_1\cup O_2)}$ contained in~$Z$ all have the same colour.
        \fi
        
        It follows that the components of~$D\setminus(M_1\cup M_2)$ corresponding to one component of~$\mathbb{S}\setminus\phi(O_1\cup O_2)$ have the same colour.
        Therefore, ${H(\mathcal{O}')}$ is $2$-colourable, and hence contains no loop. 
    \end{proof}
    
    By property~\ref{subitem:obstructions-seriesnonseries}, we have that~${s \geq 1}$. 
    By property~\ref{item:obstructions-minimality} and~\ref{item:embedding3}, there is a component of~${\mathbb{S} \setminus \phi(\bigcup \mathcal{O})}$ whose boundary intersects each of~$\phi(O_1)$, $\phi(O_2)$, and~$\phi(O_3)$, which means that some vertex of~${H(\mathcal{O})}$ is incident with all three edges. 
    If two edges of~${H(\mathcal{O})}$ are parallel, say~$e_{O_1}$ and~$e_{O_2}$, then~${H(\{O_1,O_3\})}$ has a loop, contradicting 
    the claim. 
    It follows that~$H(\mathcal{O})$ is isomorphic to the star~$K_{1,3}$. 
    
    Without loss of generality, there are points~${z_1 \in X_{1,1}}$ and~${z_2 \in X_{2,1}}$ such that~${X_{3,1} \subseteq \arc(z_1,z_2)}$. 
    Since~$\abs{X_{3,1}}$ is odd, there are points~${z_a, z_b \in X_{1,1} \cup X_{2,1}}$ such that~${\arc(z_a,z_b) \cap (X_1 \cup X_2)}$ is empty and~${\arc(z_a,z_b) \cap X_3}$ is odd. 
    It follows that each of the two components of~${\mathbb{S} \setminus \phi(\bigcup \mathcal{O})}$ whose boundary contains~$\phi(O_3)$ also contains either~$\phi(O_1)$ or~$\phi(O_2)$ in its boundary, contradicting that~${H(\mathcal{O})}$ is isomorphic to~$K_{1,3}$.
    
    We conclude that neither property~\ref{subitem:obstructions-oddcrossing} nor property~\ref{subitem:obstructions-seriesnonseries} 
    holds. 
    Hence, property~\ref{subitem:obstructions-3series} holds and~${n = x = 0}$. 
    Let~$G'$ be a graph obtained from~$G$ by adding for each~${i \in [3]}$ a vertex~$w_i$ with neighbourhood~${V(\bigcup\mathcal{P}_i) \cup \{ w_j \colon j \in [3] \setminus \{i\}\}}$. Note that~$G'$ is a planar graph. 
    Since~$W$ is connected, there is a ${(V(O_1),V(O_2))}$\nobreakdash-path~$P$ in~$W$ that contains some edge~$e$. 
    Note also that for each~${i \in [2]}$ and~${j \in [3]}$, the cycle~$O_i$ contains a path in~$\mathcal{P}_j$ by property~\ref{item:obstructions-minimality}. 
    Now the graph obtained from~${G'[\{w_1,w_2,w_3\} \cup V(O_1 \cup O_2 \cup P)]}$ 
    by contracting all edges in~${E(O_1 \cup O_2 \cup P)\setminus \{e\}}$
    is isomorphic to~$K_{5}$, which contradicts the fact that 
    every minor of a planar graph is planar.
\end{proof}

\subsection{Deriving Theorems~\ref{thm:maingroup2} and~\ref{thm:maingroup}}

As a consequence of Proposition~\ref{prop:obstructions}, we now straightforwardly obtain Theorem~\ref{thm:maingroup2}. 

\groupnecessary*

\begin{proof}
    First, suppose that ${\gen{2a} \cap A = \emptyset}$ for some~${a \in A}$.
    Let~$G_{\Gamma,A,t}$ be a graph consisting of a wall~$W$ of order at least~${t+1}$ and a crossing $W$-handlebar~$\mathcal{P}$ of size~${t+1}$ such that each row of~$W$ contains at most one vertex of~${\bigcup \mathcal{P}}$. 
    Defining~$\gamma$ such that~${\gamma(e) = 0}$ for all~${e \in E(W)}$ and~${\gamma(P) = a}$ for all~${P \in \mathcal{P}}$ yields that~${(G_{\Gamma,A,t}, \gamma) \in \mathcal{C}(t+1,t+1,\Gamma,A)}$. 
    It now follows from Proposition~\ref{prop:obstructions} that $G_{\Gamma,A,t}$ has no two vertex-disjoint cycles in~$\mathcal{O}$. 
    Now consider a set~${T \subseteq V(G_{\Gamma,A,t})}$ of size at most~$t$. 
    Observe that there exist some column~$C^W_k$ and some $W$-handle~${P \in \mathcal{P}}$ intersecting two rows~$R^W_i$ and~$R^W_j$ 
    such that~${P \cup R^W_i \cup R^W_j \cup C^W_k}$ is disjoint from~$T$ and contains a cycle in~$\mathcal{O}$. 
    Hence,~$T$ is not a hitting set for~$\mathcal{O}$ as desired.
    
    Now suppose that condition~\ref{item:mainintro2-2} fails and thus there are~${a_1,a_2,a_3 \in \Gamma}$ such that ${\gen{a_1,a_2,a_3} \cap A \neq \emptyset}$ and~${(\gen{a_1,a_2} \cup \gen{a_2,a_3} \cup \gen{a_1,a_3}) \cap A = \emptyset}$. 
    By possibly replacing~$a_i$ with another element of~${\gen{a_i}}$ for each~${i \in [3]}$, we may assume that~${(a_1 + a_2 + a_3) \in A}$. 
    Let~$G_{\Gamma,A,t}$ be a graph consisting of a wall~$W$ of order at least~${t+2}$ and three pairwise vertex-disjoint non-mixing $W$\nobreakdash-handlebars 
    $\mathcal P_1$, $\mathcal P_2$, and $\mathcal P_3$, 
    each of size~${t+1}$ and each in series. 
    Defining~$\gamma$ such that~${\gamma(e) = 0}$ for all~${e \in E(W)}$ and~${\gamma(P) = a_i}$ for all~${i \in [3]}$ and~${P \in \mathcal{P}_i}$ yields~${(G_{\Gamma,A,t}, \gamma) \in \mathcal{C}(t+1,t+2,\Gamma,A)}$. 
    It now follows from Proposition~\ref{prop:obstructions} that~$G_{\Gamma,A,t}$ has no two vertex-disjoint cycles in~$\mathcal{O}$. 
    Now let~$T$ be a set of at most~$t$ vertices of~$G_{\Gamma,A,t}$.
    Note that there are two columns~$C^W_{\ell_1}$ and~$C^W_{\ell_2}$ of~$W$
    and three $W$-handles~${P_1 \in \mathcal{P}_1}$, ${P_2 \in \mathcal{P}_2}$, and~${P_3 \in \mathcal{P}_3}$ 
    such that for the two rows~$R^W_{j_i}$ and~$R^W_{k_i}$ that intersect~$P_i$ for each~${i \in [3]}$,  
    we have that~${\bigcup \{ P_i \cup R^W_{j_i} \cup R^W_{k_i} \colon i \in [3] \} \cup C^W_{\ell_1} \cup C^W_{\ell_2}}$ is disjoint from~$T$ and contains a cycle in~$\mathcal{O}$. 
    Hence,~$T$ is not a hitting set for~$\mathcal{O}$ as desired.
\end{proof}

The following theorem is a strengthening of Theorem~\ref{thm:maingroup}. 
The main point is that for any $\Gamma_j$ with the property that every large subwall contains a $\gamma_j$-non-zero cycle, we do not have to check conditions \ref{item:mainintro1} and \ref{item:mainintro2} of Theorem~\ref{thm:maingroup} for this coordinate.
Note that in the following theorem, $m-m'$ is the number of such coordinates.
Theorem~\ref{thm:maingroup} is immediately obtained by taking $m'=m$ and $\theta=1$.
\begin{theorem}
    \label{thm:maingroup-involved}
    For every three positive integers~$m$,~$\omega$, and~$\theta$, there is a function~${f_{m,\omega,\theta} \colon \mathbb{N} \to \mathbb{N}}$ satisfying the following property. 
    Let~${\Gamma = \prod_{j \in [m]} \Gamma_j}$ be a product of~$m$ abelian groups and let ${m' \in \{0\} \cup [m]}$. 
    For every~${j \in [m]}$, let~$\Omega_j$ be a subset of~$\Gamma_j$ with~${\abs{\Omega_j} \leq \omega}$
    and let~$A_j$ be the set of all elements~${g \in \Gamma}$ such that~${\pi_j(g) \in \Gamma_j \setminus \Omega_j}$.
    Let~${A := \bigcap_{j \in [m]} A_j}$ and let~${A' := \bigcap_{j \in [m']} A_j}$. 
    Suppose that  
    \begin{enumerate}
        [label=(\arabic*)]
        \item \label{item:maingroup1} ${\gen{2a} \cap A' \neq \emptyset}$ for all ${a \in A'}$ and 
        \item \label{item:maingroup2} if ${a,b,c \in \Gamma}$ and ${\gen{a,b,c} \cap A' \neq \emptyset}$, then ${(\gen{a,b} \cup \gen{b,c} \cup \gen{a,c}) \cap A' \neq \emptyset}$.
    \end{enumerate}
    Let~$G$ be a graph with a $\Gamma$-labelling~$\gamma$ such that for each~${j \in [m] \setminus [m']}$, 
    every wall in~$G$ of order at least~$\theta$ 
    contains a cycle whose $\gamma_j$-value is non-zero.\footnotemark 
    \footnotetext{Instead, we could restrict to $\Gamma$-labelled graphs such that for each~${j \in [m] \setminus[m']}$, every subgraph of~$G$ of tree-width at least~$\theta$ contains a cycle whose $\gamma_j$-value is non-zero.}
    Let~${\mathcal{O}}$ be the set of all cycles of~$G$ whose $\gamma$-values are in~$A$.
    Then for all~${k \in \mathbb{N}}$, there exists 
    a set of~$k$ pairwise vertex-disjoint cycles in~$\mathcal{O}$ 
    or a hitting set for~$\mathcal{O}$ of size at most~${f_{m,\omega,\theta}(k)}$. 
\end{theorem}

\begin{proof}
    We set~${f_{m,\omega,\theta}(k) := \widehat f_{m,\omega}(k,k,\theta)}$ for the function~$\widehat f_{m,\omega}$ as in Theorem~\ref{thm:mainobstruction}. 
    
    Suppose that~$G$ has 
    neither a set of~$k$ pairwise vertex-disjoint cycles in~$\mathcal{O}$
    nor a hitting set for~$\mathcal{O}$ of size at most~${f_{m,\omega,\theta}(k)}$. 
    Then by Theorem~\ref{thm:mainobstruction}, there exist 
    a $\Gamma$-labelling~$\gamma'$ of~$G$ that is shifting equivalent to~$\gamma$ 
    and a subgraph~$H$ of~$G$ such that~${(H,\gamma'') \in \mathcal{C}(\kappa, \theta, \Gamma / \Gamma_J, A + \Gamma_J )}$ 
    for some~${J \subseteq [m]}$ and the ${\left(\Gamma / \Gamma_J \right)}$-labelling~$\gamma''$ induced by the restriction of~$\gamma'$ to~$H$. 
    Let~$W$ be a wall of order~$\theta$, let~$t$ be a positive integer, 
    and let~${( \mathcal{P}_i \colon i \in [t] )}$ be a family of pairwise vertex-disjoint non-mixing $W$-handlebars with~${H = W \cup \bigcup \{ \bigcup \mathcal{P}_i \colon i \in [t] \}}$ as in Definition~\ref{def:obstructions}. 
    
    By property~\ref{item:obstructions-zerowall}, every cycle in~$W$ is $\gamma''$-zero and therefore~${[m] \setminus [m'] \subseteq J}$. 
    By property~\ref{item:maingroup1} and property~\ref{item:obstructions-even},
    we deduce that~$\mathcal{P}_i$ is in series for each~${i \in [t]}$. 
    In particular, neither property~\ref{subitem:obstructions-oddcrossing} nor~\ref{subitem:obstructions-seriesnonseries}  
    holds. 
    By property~\ref{item:maingroup2} and properties~\ref{item:obstructions-allowabletransversals} and~\ref{item:obstructions-minimality}, 
    we conclude that property~\ref{subitem:obstructions-3series} does not hold either, contradicting the assumption that~${(H,\gamma'') \in \mathcal{C}(\kappa, \theta, \Gamma / \Gamma_J, A + \Gamma_J )}$. 
\end{proof}

\subsection{%
\texorpdfstring{$\mathcal{S}$-cycles of length~$\ell$ modulo~$z$}%
{S-cycles of length l modulo m}}
\label{subsec:lmodm}

We now prove a generalisation of Theorem~\ref{cor:mainmod}, which additionally allows us to recover many known Erd\H{o}s-P\'{o}sa type results as discussed in Section~\ref{sec:intro}. 
Recall that 
for a family~$\mathcal{S}$ of sets of vertices, an \emph{$\mathcal{S}$-cycle} is a cycle containing at least one vertex from each member of~$\mathcal{S}$. 
Theorem~\ref{cor:mainmod} is an immediate consequence of the following theorem by taking $L=1$, $t=0$, and $\mathcal S=\emptyset$.

\begin{theorem}
    \label{thm:Sml}
    Let~$\ell$,~$z$,~$t$, and~$L$ be integers with~${z \geq 1}$ and~${t \geq 0}$. Let~${p_1^{a_1} \cdots p_n^{a_n}}$ be the prime factorisation of~$z$ with~${p_{i} < p_{i+1}}$ for all~${i \in [n-1]}$.
    The following statements are equivalent. 
    \begin{itemize}
        \item There is a function~${f \colon \mathbb{N} \to \mathbb{N}}$ such that 
        for every graph~$G$ with a family~$\mathcal{S}$ of~$t$ subsets of~${V(G)}$ and every positive integer~$k$, 
        either~$G$ contains~$k$ vertex-disjoint $\mathcal{S}$-cycles of length~$\ell$ modulo~$z$ and of length at least~$L$, or 
        there is a set of at most~${f(k)}$ vertices hitting all such cycles. 
    \item All of the following conditions hold.
        \begin{enumerate}
            [label=(\arabic*)]
            \item \label{item:Sml1} ${t \leq 2}$.
            \item \label{item:Sml2} If~${p_1 = 2}$, then~${\ell \equiv 0 \pmod {p_1^{a_1}}}$.
            \item \label{item:Sml3} There do not exist~${3 - t}$ distinct~${i \in [n]}$ for which~${\ell \not\equiv 0 \pmod {p_{i}^{a_{i}}}}$.
        \end{enumerate}
    \end{itemize}
\end{theorem}

Before presenting the proof, we present a simple lemma on integers to be used in the proof.
\begin{lemma}
    \label{lem:nonzero}
    Let~$K$ be a non-zero integer.
    Let~$n$ be a positive integer and for each~${i \in [n]}$, let~$a_i$ be an integer such that~${\abs{a_i} < \abs{K}}$.
    Then~${\sum_{i=1}^n a_iK^{i-1} = 0}$ if and only if~${a_i = 0}$ for all~${i \in [n]}$. 
\end{lemma}

\begin{proof}
    It is enough to prove the forward direction. 
    We proceed by induction on~$n$. 
    We may assume that~${n > 1}$. 
    Since~${0 = \sum_{i=1}^n a_i K^{i-1} \equiv a_1 \pmod K}$, we deduce that~${a_1 = 0}$
    and thus~${\sum_{i=2}^n a_i K^{i-2} = 0}$. 
    By the induction hypothesis,~${a_i = 0}$ for all~${i \in [n]}$.
\end{proof}

\begin{proof}[Proof of Theorem~\ref{thm:Sml}]
    Let~${m := t+2}$, let~${\omega := \max\{L,z\}}$, and let~${m' := t+1}$. 
    For all~${j \in [t]}$, let~${\Gamma_j := \mathbb{Z}}$, let~${\Gamma_{t+1} := \mathbb{Z}_{z}}$, and let~${\Gamma_{t+2}:=\mathbb Z}$.
    Let~${\Gamma := \prod_{j \in [m]} \Gamma_j}$. 
    For all~${j \in [t]}$, let~${\Omega_j := \{0\}}$, let~${\Omega_{t+1} = \mathbb{Z}_z\setminus \{\ell\}}$, 
    and let~${\Omega_{t+2} := [L-1]}$. 
    For each~${j \in [m]}$, let~$A_j$ be the set of all~${g \in \Gamma}$ such that~${\pi_j(g) \in \Gamma_j \setminus \Omega_j}$. 
    Let~${A := \bigcap_{j \in [m]} A_j}$ and~${A' := \bigcap_{j \in [m']} A_j}$. 
    For any graph~$G$ together with a family~${\mathcal{S} = ( S_j \colon j \in [t] )}$ of subsets of~${V(G)}$, we define a $\Gamma$-labelling~$\gamma_{G,\mathcal{S}}$ as follows. 
    For each~${j \in [t]}$, let~${\gamma_{j}(e) = 1}$ if~${e \in E(H)}$ is incident with a vertex in~$S_{j}$ and~${\gamma_{j}(e) = 0}$ otherwise.
    For each~${j \in [m] \setminus [t]}$ and~${e \in E(H)}$, let~${\gamma_j(e) := 1}$. 
    Let~$\gamma_{G,\mathcal{S}}$ be the $\Gamma$-labelling of~$G$ with~${\pi_j \circ \gamma_{G,\mathcal{S}} = \gamma_j}$ for all~${j \in [m]}$. 
    Let~$\mathcal{O}_{G,\mathcal{S}}$ be the set of all $\mathcal{S}$-cycles in~$G$ of length~$\ell$ modulo~$z$ and of length at least~$L$. 
    Then the $\gamma_{G,\mathcal{S}}$-value of a cycle of~$G$ is in~$A$ if and only if it is in~${\mathcal{O}_{G,\mathcal{S}}}$. 
    
    Suppose that conditions~\ref{item:Sml1}, \ref{item:Sml2}, and~\ref{item:Sml3} hold and let~${f := f_{m,\omega,3}}$ of Theorem~\ref{thm:maingroup-involved}. 
    To apply Theorem~\ref{thm:maingroup-involved}, we verify that the two conditions in Theorem~\ref{thm:maingroup-involved} are satisfied for the subset~$A'$ of~$\Gamma$. 
    
    To check the first condition, let~${g \in A'}$.
    Then~${\pi_j(g) \neq 0}$ for each~${j \in [t]}$ and~${\pi_{t+1}(g) \equiv \ell \pmod z}$. 
    By condition~\ref{item:Sml2}, ${\gcd(2\ell, z) = \gcd(\ell, z)}$, which implies that~${\gen{2\ell} = \gen{\ell}}$ in~$\mathbb{Z}_z$. 
    Let~$x$ be a non-zero integer such that~${\ell \equiv x(2\ell)}\pmod{z}$. 
    Then,~${x \pi_{t+1}(2g) \equiv 2x \ell \equiv \ell \pmod z}$. 
    Now for all~${j \in [t]}$, we have that~${x \pi_j(2g) \neq 0}$ because~${2x \neq 0}$ and~${\pi_j(g) \neq 0}$.
    We conclude that~${x(2g) \in {\gen{2g} \cap A'}}$ and so~${\gen{2g} \cap A' \neq \emptyset}$. 
    
    Now let us check the second condition. 
    Let~${g_1,g_2,g_3 \in \Gamma}$ be such that~${\gen{g_1,g_2,g_3} \cap A' \neq \emptyset}$. 
    Then there exist integers~${x_1,x_2,x_3}$ such that 
    \begin{itemize}
        \item ${x_1 \pi_j(g_1) + x_2 \pi_j(g_2) + x_3 \pi_j(g_3) 
        \neq 0}$ for all~${j \in [t]}$ and
        \item ${x_1 \pi_{t+1}(g_1) + x_2 \pi_{t+1}(g_2) + x_3 \pi_{t+1}(g_3) \equiv \ell \pmod z}$.
    \end{itemize}
    Let~${I = \{i \in [n] \colon \ell \not\equiv 0 \pmod{p_i^{a_i}}\}}$. 
    For each~${i \in I}$, let~${q_i := \prod_{j \in [n] \setminus \{i\}} p_j^{a_j}}$.
    For each~${i \in I}$, let ${d_i \in \{g_1,g_2,g_3\}}$
    such that~${\gcd(\pi_{t+1}(d_i),p_i^{a_i})}$ is minimum. 
    Then,~$\ell$ is divisible by~${\gcd(\pi_{t+1}(d_i),p_i^{a_i})}$, which is equal to~${\gcd(q_i\pi_{t+1}(d_i),p_i^{a_i})}$. 
    Hence, there exists an integer~$y_i$ such that 
    \[
        {y_i q_i \pi_{t+1}(d_i) \equiv \ell \pmod{p_i^{a_i}}}.
    \]
    Observe that~${y_i q_i \pi_{t+1}(d_i) \equiv 0 \pmod{p_j^{a_j}}}$ for all~${j \in [n] \setminus \{i\}}$. 
    Let 
    \[
        \widehat{g} := \sum_{i \in I} (y_i q_i) d_i\in \Gamma.
    \] 
    Then,~${\pi_{t+1}(\widehat{g}) \equiv \ell \pmod z}$ and therefore, 
    for every~${g \in \Gamma}$, we have that~${\pi_{t+1}(\widehat{g}+zg) \equiv \ell\pmod z}$. 
    Let~$K$ be an integer such that~${K > \max\{ \abs{\pi_j(g_1)}, \abs{\pi_j(g_2)}, \abs{\pi_j(g_3)}, \abs{\pi_j(\widehat g)}\}}$ for all~${j \in [t]}$. 
    For each ${j \in [t]}$, there exists ${c_j \in \{g_1,g_2,g_3\}}$ such that~${\pi_j(c_j) \neq 0}$, 
    since~${x_1\pi_j(g_1) + x_2\pi_j(g_2) + x_3\pi_j(g_3) \neq 0}$. 
    By conditions~\ref{item:Sml1} and~\ref{item:Sml3}, $t+\abs{I}\leq 2$ and therefore ${\{ c_j \colon j \in [t] \} \cup \{ d_i \colon i \in I \}}$ is a proper subset of~${\{g_1,g_2,g_3\}}$, and by construction and Lemma~\ref{lem:nonzero}, we deduce that  
    \[
        \widehat{g} + \sum_{j=1}^t(K^j z) c_j  
        \in \gen{\{ c_j \colon j \in [t] \} \cup \{ d_i \colon i \in I \}} \cap A'.
    \]
    Therefore, both properties of Theorem~\ref{thm:maingroup-involved} are satisfied. 
    Let~$G$ be a graph and let~${\mathcal{S} = ( S_j \colon j \in [t] )}$ be a family of subsets of~${V(G)}$. 
    Let~${\gamma := \gamma_{G,\mathcal{S}}}$ and~${\mathcal{O} := \mathcal{O}_{G,\mathcal{S}}}$ as defined above. 
    Note that since the $\gamma_m$-value of every edge of~$G$ is positive, every wall in~$G$ contains a cycle whose $\gamma_{m}$-value is non-zero.
    Now applying Theorem~\ref{thm:maingroup-involved}, we conclude that~$G$ contains~$k$ vertex-disjoint cycles in~$\mathcal{O}$ or a hitting set for~$\mathcal{O}$ of size at most~${f(k)}$. 
    
    Now we prove the converse.
    Suppose that the first statement holds for some function~$f$.
    Let~${I := \{i \in [n]  \colon \ell \not\equiv 0 \pmod{p_i^{a_i}}\}}$. 
    Let~$G$ be a graph consisting of a wall~$W$ of order at \linebreak least~${f(3) + 2}$ together with a set~${\mathfrak{P} = \{ \mathcal{P}_i \colon i \in I \} \cup \{ \mathcal{Q}_j \colon j \in [t] \}}$ of size~${\abs{I} + t}$ 
    of pairwise vertex-disjoint non-mixing $W$-handlebars, each of size at least~${f(3) + 1}$ such that
    \begin{enumerate}
        \item every~$N^W$-path in~$W$ has length~${(\abs{L}+1)z}$,
        \item for all~${j \in [t]}$, every $W$-handle in~$\mathcal{Q}_j$ has length~${2z}$,
        \item for all~${i \in I}$, every $W$-handle in~$\mathcal{P}_i$ has length congruent to~$\ell$ modulo~$p_i^{a_i}$ and congruent to~$0$ modulo~$p_j^{a_j}$ for all~${j \in [n] \setminus \{i\}}$, 
        \item if~${1 \in I}$ and~${p_1 = 2}$, then~$\mathcal{P}_1$ is crossing and has no vertex of the first column of~$W$, and 
        \item every~${\mathcal{P} \in \mathfrak{P}}$ is in series and has no vertex of the last column of~$W$, unless ${1 \in I}$, ${p_1 = 2}$, and ${\mathcal{P} = \mathcal{P}_1}$.
    \end{enumerate}
    For each~${j \in [t]}$, let~${S_j := V(\bigcup \mathcal{Q}_j) \setminus V(W)}$ and  let~${\mathcal{S} := (S_j \colon j \in [t])}$. 
    Recall that~${\Gamma_{\{m\}}}$ denotes the subgroup~${\{ g \in \Gamma \colon \pi_j(g) = 0\text{ for all } j \in [m-1]\}}$. 
    Let~$\gamma$ be the ${(\Gamma/\Gamma_{\{m\}})}$\nobreakdash-labelling induced by~$\gamma_{G,S}$ and let~$\mathcal{O} := \mathcal{O}_{G,\mathcal{S}}$ as defined above. 
    Note that a cycle of~$G$ is in~$\mathcal{O}$ if and only if its $\gamma$-value is in~${A + \Gamma_{\{m\}}}$. 
    Note that for~${\mathcal{C}(f(3)+1,f(3)+2,\Gamma/\Gamma_{\{m\}},A+\Gamma_{\{m\}})}$, we have that~${(G,\gamma)}$ satisfies  properties~\ref{item:obstructions-union}--\ref{item:obstructions-even}. 
    In particular, every cycle of~$G$ that contains exactly one $W$-handle from each~${\mathcal{P} \in \mathfrak{P}}$ is in~$\mathcal{O}$.

    Now we claim that there is no set of vertices of size at most~${f(3)}$ hitting all cycles in~$\mathcal{O}$. 
    For a set~$T$ of vertices with~${\abs{T} \leq f(3)}$, 
    there are two columns~$C$ and~$C'$ of~$W$ containing no vertex of~$T$,
    because~$W$ has more than~${\abs{T} + 2}$ columns. 
    Similarly, for each~${\mathcal{P} \in \mathfrak{P}}$, there is a $W$\nobreakdash-handle~${P_{\mathcal{P}} \in \mathcal{P}}$ such that 
    no vertex of~$T$ is contained in~$P_{\mathcal{P}}$ or 
    the two rows~$R_{\mathcal{P}}$ and~$R'_{\mathcal{P}}$ that intersect~$P_{\mathcal{P}}$.
    Then,~${\bigcup \{ P_{\mathcal{P}} \cup R_{\mathcal{P}} \cup R'_{\mathcal{P}} \colon \mathcal{P} \in \mathfrak{P} \} \cup C \cup C'}$ is disjoint from~$T$ and contains a cycle in~$\mathcal{O}$. 
    Therefore,~$T$ does not hit all cycles in~$\mathcal{O}$.

    We deduce from the definition of~$f$ that~$G$ contains three vertex-disjoint cycles in~$\mathcal{O}$. 
    Hence, ${(G,\gamma) \notin \mathcal{C}(f(3)+1, f(3)+2, \Gamma/\Gamma_{\{m\}}, A+\Gamma_{\{m\}})}$ by Proposition~\ref{prop:obstructions}.
    Now if~${p_1 = 2}$, then since property~\ref{subitem:obstructions-oddcrossing} is not satisfied, we have that~${1 \notin I}$. 
    Therefore, condition~\ref{item:Sml2} holds. 
    Now every $W$-handlebar in~$\mathfrak{P}$ is in series. 
    As property~\ref{subitem:obstructions-3series} is not satisfied,  conditions~\ref{item:Sml1} and~\ref{item:Sml3} hold. 
\end{proof}

\subsection{Restriction to graphs embeddable in orientable surfaces}
\label{subsec:orientable}

In this subsection, we study the implications of Theorem~\ref{thm:mainobstruction} when restricting to the class of graphs embeddable in a fixed orientable surface. 

It is known that large Escher walls are not embeddable on any fixed compact orientable surface (see for example~\cite{ABY1963}). 
The following proposition provides a condition under which no graph in~${\mathcal{C}(\kappa,\theta,\Gamma,A)}$ is embeddable on a fixed orientable surface.
\begin{proposition}
    \label{prop:planar}
    Let~$\mathbb{S}$ be a compact orientable surface
    and let~$\kappa$ be an integer such that every wall of order at least~$\kappa$ with a crossing handlebar of size at least~$\kappa$ is not embeddable on~$\mathbb{S}$. 
    (If~$\mathbb{S}$ is the sphere, then~${\kappa \geq 3}$.)
    Let~${\theta := \kappa}$.
    Let~$A$ be a subset of an abelian group~$\Gamma$. 
    The following statements are equivalent. 
    \begin{enumerate}
        [label=(\roman*)]
        \item\label{item:planar-1} Every graph in~${\mathcal{C}(\kappa,\theta,\Gamma,A)}$ satisfies property~\ref{subitem:obstructions-oddcrossing}.
        \item\label{item:planar-1.5} No graph in~${\mathcal{C}(\kappa,\theta,\Gamma,A)}$ is embeddable on~$\mathbb{S}$.
        \item \label{item:planar-2} No graph in~${\mathcal{C}(\kappa,\theta,\Gamma,A)}$ is planar.
        \item\label{item:planar-3} Every finite subset~$X$ of~$\Gamma$ with~${\sum_{g \in X} g \in A}$ contains a subset~$Y$ of size~${y \leq 2}$ such \linebreak that~${\gen{yY} \cap A \neq \emptyset}$. 
    \end{enumerate} 
\end{proposition}

\begin{proof}
    By definition of~$\kappa$,~\ref{item:planar-1} implies~\ref{item:planar-1.5}. 
    Trivially,~\ref{item:planar-1.5} implies~\ref{item:planar-2}. 
    
    Now suppose that~\ref{item:planar-2} holds and~\ref{item:planar-3} does not hold. 
    Let~${X \subseteq \Gamma}$ be a counterexample to~\ref{item:planar-3} minimizing~$\abs{X}$.
    Then for every subset~$Y$ of~$X$ of size~${y \leq 2}$, we have~${\gen{yY} \cap A = \emptyset}$. 
    If there is a proper subset~$X'$ of~$X$ such that~${\gen{X'} \cap A \neq \emptyset}$, 
    then there is a subset~$X''$ of~$\gen{X'}$ with~${\abs{X''} < \abs{X}}$ such that~${\sum_{g \in X''} g \in A}$ and each element of~$X''$ is a multiple of some element of~$X$. 
    By the minimality of~$\abs{X}$, the set~$X''$ has a subset~$Y'$ of size~${y \leq 2}$ such that~${\gen{yY'} \cap A \neq \emptyset}$. 
    In this case, there is a subset~$Y$ of~$X$ of size~$y$ such that~${\gen{yY'} \subseteq \gen{yY}}$, contradicting the assumption that~${\gen{yY} \cap A = \emptyset}$. 
    Therefore, ${\gen{X'} \cap A = \emptyset}$ for every proper subset~$X'$ of~$X$. 
    
    If~${\abs{X} \geq 3}$, then~${\mathcal{C}(\kappa,\theta,\Gamma,A)}$ contains the obstruction~${(G,\gamma)}$ consisting of a wall~$W$ of order at least~$\theta$ and a family~${( \mathcal{P}_x \colon x \in X )}$ of $W$\nobreakdash-handlebars, each in series, with~${\gamma(P) = x}$ for all~${P \in \mathcal{P}_x}$. 
    Since~$G$ is planar, this contradicts~\ref{item:planar-2}. 
    Now suppose~${\abs{X} = 2}$ and let~${X = \{x_1,x_2\}}$, where~${\gen{2x_1,2x_2} \cap A = \emptyset}$.
    If~${\gen{2x_1,x_2} \cap A = \emptyset}$, then set~${(g_1,g_2) := (x_1,x_2)}$ and otherwise let~$a$ and~$b$ be integers such that~${ax_2 + 2b(x_1+x_2) \in A}$
    and set~${(g_1,g_2) := (ax_2,2b(x_1+x_2))}$. 
    Observe that~${\{g_1,g_2\}}$ is a counterexample to~\ref{item:planar-3} with~${\gen{2g_1,g_2} \cap A = \emptyset}$. 
    Now~${\mathcal{C}(\kappa,\theta,\Gamma,A)}$ contains the obstruction~${(G,\gamma)}$ consisting of a wall~$W$ of order at least~$\theta$ and two $W$-handlebars~${\mathcal{P}_1}$ and~${\mathcal{P}_2}$ 
    such that~$\mathcal{P}_1$ is nested and~$\mathcal{P}_2$ is in series, and such that~${\gamma(P) = g_i}$ for all~${P \in \mathcal{P}_i}$ for each~${i \in [2]}$. 
    As before,~$G$ is planar, contradicting~\ref{item:planar-2}.
    Therefore,~\ref{item:planar-2} implies~\ref{item:planar-3}. 
    
    Now suppose that~\ref{item:planar-3} holds, and let~${(G,\gamma) \in \mathcal{C}(\kappa,\theta,\Gamma,A)}$ which consists of a wall~$W$ and a family~${\mathfrak{P} = (\mathcal{P}_i \colon i \in [t])}$ of $W$-handlebars as described in Definition~\ref{def:obstructions}. 
    Let~${(P_i \colon i \in [t])}$ be a family such that~${P_i \in \mathcal{P}_i}$ for all~${i \in [t]}$, and let~${X = \{ \gamma(P_i) \colon i \in [t] \}}$. 
    By property~\ref{item:obstructions-minimality}, ${\gen{Y} \cap A = \emptyset}$ for every proper subset~${Y}$ of~${X}$ and therefore~${\abs{X} \leq 2}$ by~\ref{item:planar-3}. 
    If property~\ref{subitem:obstructions-seriesnonseries} holds, then~${\abs{X} \geq 2}$ and so by~\ref{item:planar-3}, we have that~${\gen{2X} \cap A \neq \emptyset}$, which implies that~$\mathcal{P}_{i}$ is in series for all~${i \in [t]}$ by~\ref{item:obstructions-even}, contradicting property~\ref{subitem:obstructions-seriesnonseries}. 
    Thus,~$G$ does not satisfy property~\ref{subitem:obstructions-seriesnonseries}.
    Now~$G$ does not satisfy property~\ref{subitem:obstructions-3series} since this would require~${\abs{X} \geq 3}$.
    We conclude that~$G$ satisfies property~\ref{subitem:obstructions-oddcrossing}, and therefore~\ref{item:planar-3} implies~\ref{item:planar-1}. 
\end{proof}

By applying Theorem~\ref{thm:mainobstruction} with Proposition~\ref{prop:planar}, 
we obtain the following corollary.

\begin{corollary}
    \label{cor:orientablesurface}
    For all positive integers~$m$ and~$\omega$ and every compact orientable surface~$\mathbb{S}$, there is a function~${f \colon \mathbb{N} \to \mathbb{N}}$ satisfying the following property. 
    Let~${\Gamma = \prod_{j \in [m]} \Gamma_j}$ be a product of~$m$ abelian groups, and for each~${j \in [m]}$, let~$\Omega_j$ be a subset of~$\Gamma_j$ with~$\abs{\Omega_j} \leq \omega$.
    Let~$A$ be the set of all elements~${g \in \Gamma}$ such that~${\pi_j(g) \in \Gamma_j \setminus \Omega_j}$ for all~${j \in [m]}$. 
    Suppose that 
    \begin{enumerate}
        [label=$(\ast)$]
        \item\label{item:orientablesurface} every finite subset~$X$ of~$\Gamma$ with~${\sum_{g \in X} g \in A}$  
        contains a subset~$Y$ of size~${y \leq 2}$ such \linebreak that~${\gen{yY} \cap A \neq \emptyset}$. 
    \end{enumerate}
    Let~$G$ be a $\Gamma$-labelled graph embeddable in~$\mathbb{S}$ with $\Gamma$-labelling~$\gamma$ and let~${\mathcal{O}}$ be the set of all cycles of~$G$ whose $\gamma$-values are in~$A$.
    Then for all~${k \in \mathbb{N}}$, there exists 
    a set of~$k$ pairwise vertex-disjoint cycles in~$\mathcal{O}$
    or a hitting set for~$\mathcal{O}$ of size at most~${f(k)}$. 
\end{corollary}

\begin{proof}
    We may assume that~${A \neq \emptyset}$, because otherwise the result is trivial. 
    Let~$\kappa$ be an integer such that every wall of order at least~$\kappa$ with a crossing handlebar of size at least~$\kappa$ is not embeddable on~$\mathbb{S}$. 
    Let~${\theta := \kappa}$.
    By Theorem~\ref{thm:mainobstruction}, it is enough to show that no graph in~${\mathcal{C}(\kappa,\theta,\Gamma/\Gamma_J,A+\Gamma_J)}$ is embeddable in~$\mathbb{S}$ for every~${J \subseteq [m]}$.
    By Proposition~\ref{prop:planar}, it suffices to show that condition~\ref{item:orientablesurface} holds for the group~${\Gamma/\Gamma_J}$ and its subset~${A+\Gamma_J}$ for every subset~$J$ of~${[m]}$. 
    Let~${J \subseteq [m]}$ and let~${X = \{g_i \colon i \in [s]\}}$ be a finite subset of~${\Gamma/\Gamma_J}$ such that~${\sum_{i \in [s]} g_i \in A + \Gamma_J}$. 
    Since~${A \neq \emptyset}$, we can pick a representative~${g'_i \in g_i}$ for each~${i \in [s]}$ so that~${\sum_{i \in [s]} g'_i \in A}$. 
    Now by condition~\ref{item:orientablesurface} there is a subset~${Y \subseteq [s]}$ of size~${y \leq 2}$ such that~${\gen{yg'_i \colon i \in Y} \cap A \neq \emptyset}$. 
    By the definition of~$A$, we have~${\gen{yg_i \colon i \in Y} \cap (A + \Gamma_J) \neq \emptyset}$, as required. 
\end{proof}

Now we show a converse of Corollary~\ref{cor:orientablesurface} analogous to Theorem~\ref{thm:maingroup2} for compact orientable surfaces, namely that whenever condition~\ref{item:orientablesurface} of Corollary~\ref{cor:orientablesurface} fails to hold, there is a planar obstruction.

\begin{corollary}
    \label{cor:orientablesurface2}
    Let~$A$ be a subset of an abelian group~$\Gamma$ such that~$A$ does not satisfy the following condition:
    \begin{enumerate}
        [label=$(\ast)$]
        \item \label{item:planarproperty} every finite subset~$X$ of~$\Gamma$ with~${\sum_{g \in X} g \in A}$ contains a subset~$Y$ of size~${y \leq 2}$ such \linebreak that~${\gen{yY} \cap A \neq \emptyset}$. 
    \end{enumerate}
    Then for every positive integer~$t$, there is a planar graph~$G_{\Gamma,A,t}$ with a $\Gamma$-labelling~$\gamma$ such that for the set~$\mathcal{O}$ of cycles of~$G_{\Gamma,A,t}$ with values in~$A$, 
    there are no three vertex-disjoint cycles in~$\mathcal{O}$ and there is no hitting set for~$\mathcal{O}$ of size at most~$t$. 
\end{corollary}

\begin{proof}
    Let~${X \subseteq \Gamma}$ be a counterexample to~\ref{item:planarproperty} of minimum size. 
    Then for all subsets~$Y\subseteq X$ of size~${y \leq 2}$, we have~${\gen{yY} \cap A = \emptyset}$. 
    In particular, ${0 \notin A}$ 
    and~${\gen{a} \cap A = \emptyset}$ for all~${a \in X}$.
    Furthermore, if~$a$ and~$b$ are distinct elements of~$X$, then~${\gen{2a,2b} \cap A = \emptyset}$. 
    
    Suppose that~${\abs{X} \geq 3}$. 
    If~$X$ has two elements~$x_1$ and~$x_2$ such that~${\gen{x_1,x_2} \cap A \neq \emptyset}$, then there are~${c_1, c_2 \in \mathbb{Z}}$ such that~${c_1 x_1 + c_2 x_2 \in A}$. 
    Since~${\gen{x_1} \cap A = \emptyset}$ and~${\gen{x_2} \cap A = \emptyset}$, we have both~${c_1 \neq 0}$ and~${c_2 \neq 0}$, as well as~${c_1 x_1 \neq c_2 x_2}$. 
    Let~${X' = \{c_1x_1,c_2x_2\}}$. 
    Since~${\gen{2c_1x_1,2c_2x_2} \subseteq \gen{2x_1,2x_2}}$, the set~$X'$ is also a counterexample to~\ref{item:planarproperty}, contradicting the minimality of~$\abs{X}$. 
    Therefore~$X$ has no two elements~$x_1$ and~$x_2$ such that~${\gen{x_1,x_2} \cap A \neq \emptyset}$ 
    and hence~$A$ does not satisfy the second property in Theorem~\ref{thm:maingroup2}. 
    Then the graph constructed in the proof of Theorem~\ref{thm:maingroup2} is planar. 

    Hence we may assume that~${X = \{x_1,x_2\}}$ for some pair of distinct elements of~$\Gamma$ since condition~\ref{item:planarproperty} is trivially satisfied by subsets of~$\Gamma$ of size at most~$1$. 
    As in the proof of Proposition~\ref{prop:planar}, we may assume that~${\gen{2x_1,x_2} \cap A = \emptyset}$. 
   
    Let~$G_{\Gamma,A,t}$ be a graph consisting of a wall~$W$ of order at least~${t+2}$ and a pair~${\mathfrak{P} = \{\mathcal{P}_i \colon i \in [2]\}}$ of pairwise vertex-disjoint non-mixing $W$-handlebars each of size~${t+1}$, such that~$\mathcal{P}_1$ is nested and~$\mathcal{P}_2$ is in series. 
    Defining~$\gamma$ such that~${\gamma(e) = 0}$ for all~${e \in E(W)}$ and~${\gamma(P) = x_i}$ for all~${i \in [2]}$ and~${P \in \mathcal{P}_i}$ yields that~${(G, \gamma) \in \mathcal{C}(t+1,t+2,\Gamma,A)}$. 
    It now follows from Proposition~\ref{prop:obstructions} that there are no three vertex-disjoint cycles in~$\mathcal{O}$. 
    Now consider a set~${T \subseteq V(G)}$ of size at most~$t$. 
    Since~$W$ has at least~${t+2}$ columns, there are two columns~$C^W_{\ell_1}$ and~$C^W_{\ell_2}$ of~$W$ having no vertex of~$T$.
    For each~${i \in [2]}$, since~${\abs{\mathcal P_i} = t + 1}$, 
    there is a $W$-handle~${P_i \in \mathcal{P}_i}$ such that neither~$P_i$ nor any of the two rows~$R^W_{j_i}$ and~$R^W_{k_i}$ that intersect~$P_i$ contains a vertex of~$T$.
    Now~${P_1 \cup R_{j_1}^W \cup R_{k_1}^W \cup P_2 \cup R_{j_2}^W \cup R_{k_2}^W \cup C^W_{\ell_1} \cup C^W_{\ell_2}}$ has no vertex in~$T$ and contains a cycle in~$\mathcal{O}$. 
    Hence,~$T$ is not a hitting set for~$\mathcal{O}$ as desired.
\end{proof}

For example, cycles that are either odd or of length~$16$ modulo~$30$ satisfy an Erd\H{o}s-P\'{o}sa type property when restricted to planar graphs; 
to see this, consider
\[
    \Gamma := \mathbb{Z}_{30} \text{ and }A := \{1,3,5,\ldots,29\} \cup \{16\},
\] 
and the class of $\Gamma$-labelled graphs whose edges all have value $1$.
Let~${X \subseteq \Gamma}$ with~${\sum_{g \in X} g \in A}$. 
We claim that~$X$ contains a subset~$Y$ of size~${y \leq 2}$ such  
that~${\gen{yY} \cap A \neq \emptyset}$. 
Since~${0 \notin A}$, we have that~$X$ is nonempty. 
If~$X$ contains some~${a \in A}$,  then~${\gen{a} \cap A \neq \emptyset}$. 
If~${\gcd(a,30) \mid 2}$ for some~${a \in X}$, then~${16 \in \gen{a}}$. 
So we may assume that every element of~$X$ is even and a multiple of either~$3$ or~$5$. 
Then~${\sum_{g \in X} g \equiv 16 \pmod{30}}$. 
Then there is~${g_1 \in X}$ such that~${g_1 \not\equiv 0 \pmod{3}}$ and similarly there is~${g_2 \in X}$ such that~${g_2 \not\equiv 0 \pmod{5}}$. 
Choose~${a_1 \in \{10,20\}}$ so that~${a_1 g_1 \equiv 1 \pmod 3}$ and choose~${a_2 \in \{6,12,18,24\}}$ so that~${a_2 g_2 \equiv 1 \pmod 5}$. 
Let~${m = a_1 g_1 + a_2 g_2}$. 
Then~$m$ is even, ${m \equiv 1 \pmod 3}$ 
and~${m\equiv 1\pmod 5}$, which imply that~${m \equiv 16 \pmod {30}}$. 
Thus,~${\gen{2g_1,2g_2} \cap A \neq \emptyset}$ and condition~\ref{item:orientablesurface} of Corollary~\ref{cor:orientablesurface} holds. 

However, the cycles that are either odd or of length~$106$ modulo~$210$ do not satisfy an Erd\H{o}s-P\'{o}sa type property when restricted to planar graphs.
This is because~${120 + 70 + 126 \equiv 106 \pmod{210}}$ and for every proper subset~$Y$ of~${\{120,70,126\}}$ of size~$y$, the subset~$\gen{yY}$ of~${\mathbb{Z}_{210}}$ has empty intersection with~${\{1,3,5,\ldots,209\} \cup \{106\}}$. 

We will now derive the exact characterisation of when cycles of length~$\ell$ modulo~$z$ satisfy an Erd\H{o}s-P\'{o}sa type result in planar graphs. 

\planarmod*
\begin{proof}
    Let~${\Gamma := \mathbb{Z}_z}$ and~${A := \{\ell\}\subseteq \Gamma}$.
    Given a $\Gamma$-labelled graph~$(G,\gamma)$ with no three pairwise vertex-disjoint cycles whose $\gamma$-values are in~$A$, we can construct a graph~$G$ with no three pairwise vertex-disjoint cycles of length~$\ell$ modulo~$z$ by replacing every edge~${e \in E(G)}$ by a path of length~$\gamma(e)$ modulo~$z$. 
    The size of the smallest hitting set for the cycles of length~$\ell$ modulo~$z$ in the newly constructed graph will equal the size of the smallest hitting set for the cycles of~$(G,\gamma)$ whose $\gamma$-values are in~$A$. 
    Conversely, given a graph~$G$, the $\Gamma$-labelling in which every edge has value~$1$ has the property that a cycle has $\gamma$-value in~$A$ if and only if it has length~$\ell$ modulo~$z$. 
    Thus, by Corollaries~\ref{cor:orientablesurface} and~\ref{cor:orientablesurface2}, it suffices to show that the second statement is equivalent to condition~\ref{item:planarproperty} for~$\Gamma$ and~$A$.
    
    First, suppose that the second statement holds. 
    By Theorem~\ref{cor:mainmod}, we may assume that~${p_1 = 2}$, that~${\ell \not\equiv 0 \pmod {2^{a_1}}}$, and that~${\ell \equiv 0 \pmod {z/2^{a_1}}}$. 
    Let~$t$ be the minimum positive integer such that~${2^t \nmid \ell}$. 
    Note that~${t \leq a_1}$.
    Let~$X$ be a subset of~$\Gamma$ such that~${\sum_{g \in X} g \equiv\ell\pmod z}$. 
    Since~${2^t \nmid \ell}$ and~${2^t \mid z}$, there exists~${g \in X}$ such that~${2^t \nmid g}$. 
    Then~${\gcd(g,z) \mid \ell}$ and therefore there is an integer~$a$ such that~${a g \equiv \ell \pmod z}$, so condition~\ref{item:planarproperty} holds. 
    
    Suppose instead that the second statement does not hold. 
    Let~${J \subseteq [n]}$ be the set of indices $j$ such that~${p_j^{a_j}\nmid \ell }$.
    By the Chinese remainder theorem,
    for each~${j \in J}$, there exists an integer~$g_j$ such that~${g_j \equiv \ell \pmod {p_j^{a_j}}}$
    and~${g_j \equiv 0 \pmod {p_k^{a_k}}}$ for all~${k \in [n] \setminus \{j\}}$. 
    Let~${X := \{g_j \colon j \in J\}}$, and note that~${\sum_{g \in X} g \equiv \ell \pmod z}$. 
    
    If~$Y$ is a proper subset of~$X$, then for each~${j \in J}$ with~${g_j \in X \setminus Y}$, 
    we have ${p_j^{a_j} \nmid \ell}$ 
    and ${p_j^{a_j} \mid g}$ for all~${g \in Y}$, 
    and therefore ${\ell \notin \gen{Y}}$.
    This implies that if~${\abs{J} \geq 3}$, then condition~\ref{item:planarproperty} of Corollary~\ref{cor:orientablesurface2} fails. 
    Thus we may assume that~${\abs{J} \leq 2}$. 
    Hence~\ref{enum:planarmod1} fails, and we have~${p_1 = 2}$, ${\ell \not\equiv 0 \pmod{p_1^{a_1}}}$, and ${\ell \not\equiv 0 \pmod{z/p_1^{a_1}}}$. 
    In particular,~${\abs{J} = 2}$ and we have~$J=\{1,m\}$ for some $m \in [n]\setminus\{1\}$.
    Let~$t$ be the minimum positive integer such that~${2^t \nmid \ell}$. 
    Then~${2^t \mid 2g_1}$ by definition and~${2^t \mid 2 g_m}$ since $2^{a_1}\mid g_m$. 
    So,~${\ell \notin \gen{2X}}$, which implies that condition~\ref{item:planarproperty} of Corollary~\ref{cor:orientablesurface2} fails, as required.
\end{proof}

\medskip

When considering surface embeddings of graphs, it is also natural to consider the homology classes of cycles. 
For graphs embedded in a fixed compact surface, Huynh, Joos, and Wollan obtained a half-integral Erd\H{o}s-P\'{o}sa result for the non-null-homologous cycles of the embedding~\cite[Theorem~6]{HuynhJW2017}, and an integral Erd\H{o}s-P\'{o}sa result for these cycles when the surface is orientable~\cite[Corollary~41]{HuynhJW2017}. 
They did this by considering a different type of group labelling, where the two orientations of each edge are assigned labels that are inverse to each other. 

Since in our setting we do not distinguish between the two orientations of an edge, we are unable to directly apply our results to homology classes in the first homology group with coefficients in~$\mathbb{Z}$. 
However orientations can be ignored when considering the first homology group with coefficients in~$\mathbb{Z}_2$, and so our results are applicable. 
Note that for a closed orientable surface, the set of simple closed curves homologous to zero for the $\mathbb{Z}_2$-homology is exactly the same as for the $\mathbb{Z}$-homology. 
This follows the universal coefficient theorem (see~\cite{Hatcher02}), which allows us to relate the $\mathbb{Z}$-homology with the $\mathbb{Z}_2$-homology by taking all coefficients modulo~${2}$. 
We then apply a classical result which states that no simple closed curve has $\mathbb{Z}$-homology class~${kh}$ for any integer~${k \geq 2}$ and any non-zero element~$h$ of the $\mathbb{Z}$-homology (see for example~\cite{Schafer1976}). 

The following elementary observation allows us to encode the $\mathbb{Z}_2$-homology classes of cycles in our group labelling setting (see {\cite[Proposition~3.5]{GollinHKKO2021}} for a proof). 
A subgraph~$H$ of~$G$ is called \emph{spanning} if~${V(H) = V(G)}$.
A graph~$H$ is called \emph{even} if every vertex of~$H$ has even degree. 
For a graph~$G$, let~${\mathcal{C}(G)}$ denote the \emph{cycle space} of~$G$ over~$\mathbb{Z}_2$, which is the vector space of all spanning even subgraphs~$H$ of~$G$ with the symmetric difference of edge sets as its operation. 

It is easy to observe the following by taking~${\gamma(e) = 0}$ for all edges in a fixed spanning tree of each component. 

\begin{observation}%
    \label{obs:cyclespace-hom-to-gamma}
    Let~$G$ be a graph, let~$\Gamma$ be an abelian group, and let~${\phi \colon \mathcal{C}(G) \to \Gamma}$ be a group homomorphism. 
    Then there is a $\Gamma$-labelling~$\gamma$ of~$G$ such that~${\gamma(H) = \phi(H)}$ for every even subgraph~$H$ of~$G$. 
\end{observation}

We also need the following lemma. 
Let~${\chi(\mathbb{S})}$ denote the Euler characteristic of a surface~$\mathbb{S}$. 

\begin{lemma}[{See Diestel~\cite[Lemma~B.6]{Diestel2018}}]
    \label{lem:euler}
    Let~$\mathbb{S}$ be a compact surface and let~$\mathcal{C}$ be a finite set of pairwise disjoint circles in~$\mathbb{S}$. 
    If 
    \begin{itemize}
        \item $\mathbb{S} \setminus \bigcup \mathcal{C}$ has a component~$D_0$ whose closure in~$\mathbb{S}$ meets every circle in~$\mathcal{C}$ and
        \item no circle in~$\mathcal{C}$ bounds a disk in~$\mathbb{S}$ that is disjoint from~$D_0$,
    \end{itemize}
    then~$\abs{\mathcal{C}}\leq 2-\chi(\mathbb{S})$. 
\end{lemma}

We now obtain a strengthening of the integral Erd\H{o}s-P\'{o}sa result of Huynh, Joos, and Wollan for graphs embedded in a fixed orientable surface~\cite[Corollary~41]{HuynhJW2017}.

\begin{corollary}
    Let~$\mathbb{S}$ be a compact orientable surface with $\mathbb{Z}_2$-homology group~$\Gamma$ and let~$A$ be a set of $\mathbb{Z}_2$-homology classes of~$\mathbb{S}$. 
    There exists a function~${f \colon \mathbb{N} \to \mathbb{N}}$
    such that for all~${k \in \mathbb{N}}$ and every graph~$G$ embedded in~$\mathbb{S}$, 
    there exist~$k$ vertex-disjoint cycles whose $\mathbb{Z}_2$-homology classes are in~$A$
    or a hitting set of size at most~${f(k)}$
    for the set of cycles  whose $\mathbb{Z}_2$-homology classes are in~$A$.
\end{corollary}

\begin{proof}
    We will apply Theorem~\ref{thm:mainobstruction} with~${m := 1}$, ${\omega := \abs{\Gamma \setminus A}}$, ${\Gamma_1 := \Gamma}$, and~${\Omega_1 := \Gamma \setminus A}$. 
    Let~$\kappa$ and~$\theta$ be integers such that~$\kappa > 2-\chi(\mathbb{S})$ and no graph containing a wall~$W$ of order at least~$\theta$ and a crossing $W$-handlebar of size~$\kappa$ is embeddable in~$\mathbb{S}$, and let~${f(k) := f_{1,\omega}(k,\kappa,\theta)}$. 
    Let~$G$ be a graph embedded in~$\mathbb{S}$, and let~$\gamma$ be a $\Gamma$-labelling of~$G$ such that~$\gamma(H)$ is the~$\mathbb{Z}_2$-homology class of~$H$ for every even subgraph~$H$ of~$G$ (see Observation~\ref{obs:cyclespace-hom-to-gamma}). 
    
    Suppose for a contradiction that there are neither~$k$ vertex-disjoint cycles whose $\mathbb{Z}_2$-homology classes are in~$A$ 
    nor a hitting set of size at most~${f(k)}$ for these cycles. 
    By Theorem~\ref{thm:mainobstruction}, for some~$\gamma'$ shifting-equivalent to~$\gamma$, there is a subgraph~$H$ of~$G$ and a subset~${J}$ of $[1]$ such that 
    ${(H,\gamma'') \in \mathcal{C}(\kappa, \theta, \Gamma / \Gamma_J, A + \Gamma_J )}$
    for the $\left(\Gamma / \Gamma_J \right)$-labelling~$\gamma''$ induced by the restriction of~$\gamma'$ to~$H$.
    Note that ${\Gamma/\Gamma_J}$ is not the trivial group by properties~\ref{item:obstructions-allowabletransversals} and~\ref{item:obstructions-minimality} of Definition~\ref{def:obstructions}, hence~${J = \emptyset}$. 
    Let~$W$ be the wall in~$H$ and let~$\mathfrak{P}$ be the family $W$-handlebars in~$H$ described in Definition~\ref{def:obstructions}. 
    By our choice of~$\kappa$ and~$\theta$, there is
    no crossing~$W$-handlebar of size~$\kappa$ in~$H$, so by property~\ref{item:obstructions-handlebars}, some $W$-handlebar~$\mathcal{P}$ in~$\mathfrak{P}$ is in series. 
    Consider the set~$S$ of cycles in the union of the first and last column of~$W$ together with~${\bigcup \mathcal{P}}$, and note that~${\abs{S} = \kappa}$. 
    By Lemma~\ref{lem:euler}, there is a cycle~$O$ in~$S$ whose image in~$\mathbb{S}$ bounds a disk, and hence~${\gamma(O)=\gamma'(O) = 0}$. 
    But now by property~\ref{item:obstructions-zerowall}, 
    every path~${P \in \mathcal{P}}$ contained in~$O$ satisfies ${\gamma'(P) = 0}$, contradicting properties~\ref{item:obstructions-allowabletransversals} and~\ref{item:obstructions-minimality}. 
\end{proof}

\subsection{Vertex-labellings}
\label{subsec:vertexlabellings}

Let~$\Gamma$ be an abelian group and let~$G$ be a graph. 
A \emph{$\Gamma$-vertex-labelling} of~$G$ is a function~${\nu \colon V(G) \to \Gamma}$. 
By slight abuse of notation, we use the analogues of terminology from the $\Gamma$-(edge-)labellings for $\Gamma$-vertex-labellings (for example \emph{$\nu$-value}) without further explanation. 

We discuss how to translate our results on edge-labelled graphs to vertex-labelled graphs. 
On the one hand, given an abelian group~$\Gamma$ and a $\Gamma$-(edge-)labelled graph~${(G,\gamma)}$, as discussed in~\cite{GollinHKKO2021} we can construct a $\Gamma$-vertex-labelled graph~${(G',\nu)}$ by subdividing each edge of~$G$ and setting~${\nu(v) = \gamma(e)}$ for the subdivision vertex~$v$ of the edge~$e$, and setting~${\nu(v) = 0}$ for every vertex~${v \in V(G)}$. 
Now the cycle~$O'$ of~$G'$ obtained from a cycle~$O$ of~$G$ by subdividing each edge of~$O$ has the property that~${\gamma(O) = \nu(O')}$. 
From this one can easily derive vertex-labelled analogues of Theorem~\ref{thm:maingroup2} and Corollary~\ref{cor:orientablesurface2}. 

On the other hand, the approach to translate positive results is slightly more elaborate than in~\cite{GollinHKKO2021}. 
The following lemma is a straightforward consequence of the fundamental theorem of finitely generated abelian groups, but the reader can extract a proof of it from the proof of~\cite[Lemma~3.4]{GollinHKKO2021}. 

\begin{lemma}
    \label{lem:vertexlabelhomomorphism}
    For every finitely generated abelian group~$\Gamma$,
    there exist an abelian group~$\Gamma'$ and an injective homomorphism~${\psi \colon \Gamma \to \Gamma'}$ whose image is~${2\Gamma'}$.\qed
\end{lemma}

Given an abelian group~$\Gamma$ and $\Gamma$-vertex-labelled graph~${(H,\nu)}$, we say that a wall~$W$ of~$H$ is \emph{$\nu$-homogeneous} if every cycle in~$W$ has $\nu$-value zero and for every $W$-handle~$P$ in~$H$ there is some~${g_P \in \Gamma}$ such that~${\nu(O) = g_P}$ for every cycle~$O$ with~${P \subseteq O \subseteq P \cup W}$. 
In this case we define~${\mu_{W,\nu}(P) := g_P}$. 
The following definition is the analogue of Definition~\ref{def:obstructions} for vertex-labellings.

\begin{definition}
    \label{def:VLobstructions}
    For positive integers~$\kappa$ and~$\theta$, an abelian group~$\Gamma$, and~${A \subseteq \Gamma}$, 
    let~${\mathcal{C}'(\kappa,\theta,\Gamma,A)}$ be the class of all $\Gamma$-vertex-labelled graphs~${(G,\nu)}$ 
    having 
    a wall~$W$ of order at least~$\theta$ and 
    a nonempty family~${( \mathcal{P}_i \colon i \in [t] )}$ of pairwise vertex-disjoint non-mixing $W$-handlebars each of size at least~$\kappa$
    such that 
    \begin{enumerate}
        [label=(O$'$\arabic*)]
        \item \label{item:VLobstructions-union} $G$ is the union of $W$ and  $\bigcup \{ \bigcup \mathcal{P}_i \colon i \in [t] \}$, 
        \item \label{item:VLobstructions-zerowall}
             $W$ is~$\nu$-homogeneous,
            \item \label{item:VLobstructions-allowabletransversals}
            ${\sum_{i \in [t]} \mu_{W,\nu}(P_i) \in  A}$ for any family ${(P_i \colon i\in [t])}$ such that~${P_i \in \mathcal{P}_i}$ for all~${i \in [t]}$, 
        \item \label{item:VLobstructions-minimality}
            for each~${i \in [t]}$, we have~${\gen{\mu_{W,\nu}(P) \colon P \in \bigcup_{j \in [t] \setminus \{i\}} \mathcal{P}_j} \cap A = \emptyset}$, 
        \item \label{item:VLobstructions-even}
            if ${\sum_{j \in [t]} \sum_{P \in \mathcal{P}_j} f(P) \mu_{W,\nu}(P) \in A}$ for a function~${f \colon \bigcup_{j \in [t]} \mathcal{P}_j \to \mathbb{Z}}$, then 
            for each ${i \in [t]}$, 
            $\mathcal{P}_i$ is in series
            or~${\sum_{P \in \mathcal{P}_i} f(P)}$ is odd, and 
        \item \label{item:VLobstructions-handlebars}
            at least one of the following properties holds. 
            \begin{enumerate}
                [label=(O$'$6\alph*)]
                \item \label{subitem:VLobstructions-oddcrossing} 
                The number of crossing $W$-handlebars in~${( \mathcal{P}_i \colon i \in [t] )}$ is odd. 
                \item \label{subitem:VLobstructions-seriesnonseries} At least one but not all $W$-handlebars in~${( \mathcal{P}_i \colon i \in [t] )}$ are in series. 
                \item \label{subitem:VLobstructions-3series} At least three $W$-handlebars in~${( \mathcal{P}_i \colon i \in [t] )}$ are in series. 
            \end{enumerate}
    \end{enumerate}
\end{definition}

Observe that if ${(G,\gamma) \in \mathcal{C}(\kappa,\theta,\Gamma,A)}$ and $\nu$ is a $\Gamma$-vertex-labelling of~$G$ such that~${\nu(O) = \gamma(O)}$ for every cycle~$O$ of~$G$, then~${(G,\nu) \in \mathcal{C}'(\kappa,\theta,\Gamma,A)}$. 
With this definition, we can obtain vertex-labelled analogues of all of the results in this paper which reference edge-labellings. 
To illustrate this, we now prove the following analogue of Theorem~\ref{thm:mainobstruction}. 
The proofs of the analogues of other results are almost verbatim the original proofs with Definition~\ref{def:VLobstructions} and Theorem~\ref{thm:vertexlabellings} in place of Defintion~\ref{def:obstructions} and Theorem~\ref{thm:mainobstruction}, so we omit them. 

\begin{restatable}{theorem}{vertexlabellings}
    \label{thm:vertexlabellings}
    For all positive integers~$m$ and~$\omega$, there is a function~${\widehat f_{m,\omega} \colon \mathbb{N}^3 \to \mathbb{Z}}$ satisfying the following property. 
    Let~${\Gamma = \prod_{j \in [m]} \Gamma_j}$ be a product of~$m$ abelian groups, and for every~${j \in [m]}$, let~$\Omega_j$ be a subset of~$\Gamma_j$ with~${\abs{\Omega_j} \leq \omega}$. 
    For each~${j \in [m]}$, let~${A_j := \pi_j^{-1}(\Gamma_j\setminus \Omega_j)\subseteq \Gamma}$ 
    and ${A := \bigcap_{j \in [m]} A_j}$. 
    Let~$G$ be a graph with a $\Gamma$-vertex-labelling~$\nu$ and let~$\mathcal{O}$ be the set of all cycles of~$G$ whose $\nu$-values are in~$A$. 
    Then for every three positive integers~$k$, $\kappa$, and~$\theta$, at least one of the following statements is true. 
    \begin{enumerate}
        [label=(\roman*)]
        \item \label{item:VL-packing} There are~$k$ vertex-disjoint cycles in~$\mathcal{O}$. 
        \item \label{item:VL-hittingset} There is a hitting set for~$\mathcal{O}$ of size at most~${\widehat f_{m,\omega}(k, \kappa, \theta)}$. 
        \item \label{item:VL-obstruction} There is a subgraph~$H$ of~$G$ such that for some~${J \subseteq [m]}$ and for the $\left(\Gamma / \Gamma_J \right)$-labelling~$\nu'$ induced by the restriction of~$\nu$ to~$H$, we have $(H,\nu') \in \mathcal{C}'(\kappa, \theta, \Gamma / \Gamma_J, A + \Gamma_J )$ and~$H$ contains a half-integral packing of~$\kappa$ cycles in~$\mathcal{O}$.
    \end{enumerate}
\end{restatable}

\begin{proof}
    Let $\widehat{f}_{m,\omega}$ be as in Theorem~\ref{thm:mainobstruction}.
    For each~${i \in [m]}$, let~${\Gamma''_i := \gen{ \Omega_i \cup \{\nu_i(v) \colon v \in V(G)\}}}$. 
    Since ${\Gamma''_i}$ is finitely generated,
    by Lemma~\ref{lem:vertexlabelhomomorphism}, there exist an abelian group~$\Gamma'_i$ and an injective homomorphism~${\psi_i \colon \Gamma''_i \to \Gamma'_i}$ whose image is~$2\Gamma'_i$.
    Let~$\Gamma''$ be the subgroup~${\prod_{i\in [m]} \Gamma''_i}$ of~$\Gamma$, let~${\Gamma' := \prod_{i \in [m]} \Gamma'_i}$, and let~${\psi \colon \Gamma''\to \Gamma'}$ such that~${\pi_i(\psi(g)) = \psi_i(\pi_i(g))}$ for all~${i \in [m]}$ and~${g \in \Gamma''}$. 
    For each~${i \in [m]}$, let~${\Omega'_i := \psi_i(\Omega_i)}$, let ${A'_i := \pi^{-1}_i(\Gamma'_i \setminus \Omega'_i)}$, and let~${A' := \bigcap_{i \in [m]} A'_i}$.

    Since $\phi(\Gamma'')=2\Gamma'$, there is a function~${f \colon V(G) \to \Gamma'}$ with~${2f(v) = \psi(\nu(v))}$ for every~${v \in V(G)}$.
    For every edge~${vw \in E(G)}$, let~${\gamma(vw) := f(v)+f(w)}$.
    Then $\gamma$ is a $\Gamma'$-(edge-)labelling of~$G$ and 
    for every cycle~$O$ of~$G$, we have
    \[ 
            \gamma(O) 
            = \sum_{vw \in E(O)} \gamma(vw) 
            = \sum_{vw \in E(O)} (f(v) + f(w))
            =  \sum_{v \in V(O)} 2f(v)
            = \psi ( \nu(O) ). 
    \]
    In particular, $\mathcal{O}$ is exactly the set of cycles whose $\gamma$-values are in~$A'$. 
    Thus, if neither condition~\ref{item:VL-packing} nor condition~\ref{item:VL-hittingset} is satisfied, then by Theorem~\ref{thm:mainobstruction} we have that for some $\Gamma'$-labelling~$\gamma'$ which is shifting equivalent to~$\gamma$, there is a subgraph~$H$ of~$G$ such that for some~${J \subseteq [m]}$ and for the ${\left(\Gamma' / \Gamma'_J \right)}$-labelling~$\gamma''$ induced by the restriction of~$\gamma'$ to~$H$, we have ${(H,\gamma'') \in \mathcal{C}(\kappa, \theta, \Gamma' / \Gamma'_J, A' + \Gamma'_J )}$ and~$H$ contains a half-integral packing of~$\kappa$ cycles in~$\mathcal{O}$. 
    Let~$\nu'$ be the ${\left(\Gamma/\Gamma_J\right)}$-vertex-labelling induced by the restriction of~$\nu$ to~$H$ and let~${\psi' \colon \Gamma/\Gamma_J \to \Gamma'/\Gamma'_J}$ be the injective homomorphism induced by~$\psi$. 
    Since~$\gamma$ and~$\gamma'$ are shifting equivalent, we have~${\psi'(\nu'(O)) = \gamma''(O)}$ for every cycle~$O$ of~$H$. 
    
    Let~$W$ be the wall of order at least~$\theta$ in~$H$ as in Definition~\ref{def:obstructions}. Given a $W$-handle~$P$ of~$W$ in~$H$ and a cycle~$O$ in~${W \cup P}$, property~\ref{item:obstructions-zerowall} implies that~${\gamma''(O) = \gamma''(P)}$ if~${P \subseteq O}$ and~${\gamma''(O) = 0}$ if~${O \subseteq W}$. 
    It follows that~$W$ is $\nu$-homogeneous, and that~${\psi'(\mu_{W,\nu}(P)) = \gamma''(P)}$ for every $W$-handle~$P$ in~$H$. 
    From this, it immediately follows that~${(H,\nu')\in \mathcal{C}'(\kappa,\theta,\Gamma/\Gamma_J,A+\Gamma_J)}$, as required.
\end{proof}

\subsection{A negative result for finite allowable subsets of infinite groups}
\label{subsec:finiteA}

The following theorem shows that if the set of allowable values of cycles is a nonempty finite subset of an infinite abelian group, then a ${(1/s)}$-integral analogue of the Erd\H{o}s-P\'{o}sa theorem fails for every positive integer~$s$.

\begin{theorem}
   \label{thm:infiniteEPctex}
    Let~${A}$ be a nonempty finite subset of an infinite abelian group~$\Gamma$.
    For integers~${s \geq 2}$ and~${t \geq 1}$, 
    there is a graph~$G$ with a $\Gamma$-labelling~$\gamma$ 
    such that 
    \begin{itemize}
        \item for every set of~$s$ cycles of~$G$ whose $\gamma$-values are in~$A$, there is a vertex that belongs to all of the~$s$ cycles,
        \item there is no hitting set of size at most~$t$ 
        for the set of all cycles of~$G$ whose $\gamma$-values are in~$A$, and 
        \item no vertex belongs to $s+1$ distinct cycles whose $\gamma$-values are in~$A$.
    \end{itemize}
\end{theorem}

\begin{proof}
    Let~${\alpha \in A}$. 
    We claim that there is an infinite set~${\{g_i \colon i \in \mathbb{N}\}}$ of elements of~$\Gamma$ such that for all integers~$k'$ with~${0 \leq k' \leq s(t+1)}$ and for all distinct finite subsets~${S_1, S_2 \subseteq \mathbb{N}}$, we have 
    \begin{align}
        k'\alpha + \sum_{i \in S_1} g_i - \sum_{j \in S_2} g_j \not\in A. 
        \label{eqn:infiniteEPctex}
    \end{align}
    Indeed, if~$\Gamma$ has an element~$g'$ of infinite order, then we may choose a sufficiently large multiple~$g$ of~$g'$ so that no non-zero element of $\gen{g}$ is in the finite set ${\{ \alpha' - k'\alpha \colon \alpha' \in A,\ 0 \leq k' \leq s(t+1)\}}$. 
    Then~${\{ 2^{i}g \colon i \in \mathbb{N}\}}$ satisfies (\ref{eqn:infiniteEPctex}). 
    If every element of~$\Gamma$ has finite order, then we may sequentially choose an arbitrary element~${g_i \not\in \gen{A \cup \{ g_j \colon 1 \leq j \leq i-1 \}}}$ for all~${i \in \mathbb{N}}$. 
    This proves the claim.
    
    We will construct a graph by constructing~${s(t+1)}$ edge-disjoint cycles with the property that any set of~$s$ of them share a common vertex but no vertex is contained in more than~$s$ of them. 
    
    Let~$V$ be the set of subsets of~${[s(t+1)]}$ of size~$s$, let~${W := [s(t+1)] \times [\binom{s(t+1)-1}{s-1}]}$,
    and let~$G$ be the complete bipartite graph with bipartition~${(V,W)}$. 
    For each~${i \in [s(t+1)]}$, let~$O_i$ be a cycle of~$G$ whose vertex set is exactly 
    the union of~${\{i\} \times [\binom{s(t+1)-1}{s-1}]}$
    and the set of vertices in~$V$ containing~$i$. 
    Let~$e_i$ be an arbitrary edge of~$O_i$. 
    Observe that $E(O_i)\cap E(O_j)=\emptyset$ for distinct $i$,~$j$.
    Let~$\gamma$ be a $\Gamma$-labelling of~$G$ 
    assigning each edge in ${E(G) \setminus \{e_i \colon i \in [s(t+1)]\}}$ a distinct value in~${\{g_i \colon i \in \mathbb{N}\}}$, and for each~${i \in [s(t+1)]}$ assigning~$e_i$ the value~${\alpha - \gamma(O_i-e_i)}$.
    Each vertex of~$G$ is contained in at most~$s$ cycles in $\{O_i \colon i \in [s(t+1)]\}$, so every hitting set for~$\{O_i \colon i \in [s(t+1)]\}$ has size at least~${t+1}$, and by construction, for every set of $s$~cycles in~$\{O_i \colon i \in [s(t+1)]\}$, there is a vertex in~$V$ that belongs to all of the~$s$ cycles
    and no vertex is in $s+1$ distinct cycles in $\{O_i:i\in [s(t+1)]\}$.

    We will finish the proof by showing that
    the set $\mathcal O$ of all cycles of $G$ whose $\gamma$-values are in $A$
    is equal to~${\{O_i \colon i \in [s(t+1)]\}}$. 
    By definition, ${\gamma(O_i) = \alpha \in A}$ for each~${i \in [s(t+1)]}$.
    Now suppose that~${O}$ is a cycle in~${\mathcal{O}}$. 
    Let~${I := \{i \in [s(t+1)] \colon e_i \in E(O)\}}$, let~${F := E(O) \setminus \{e_i \colon i \in I\}}$, and let ${F' := \bigcup_{i \in I} ( E(O_i) \setminus \{e_i\} )}$. 
    Then
    \[
        \gamma(O) 
        = \sum_{i \in I} (\alpha - \gamma(O_i - e_i)) + \sum_{e \in F} \gamma(e)
        = \abs{I}\alpha+\sum_{e\in F} \gamma(e) - \sum_{e \in F'} \gamma(e),
    \]
    so by~(\ref{eqn:infiniteEPctex}), we have~${F = F'}$.
    Since~$O$ is a cycle and cycles in~${\{O_i \colon i \in [s(t+1)]\}}$ are edge-disjoint, we deduce that~${\abs{I} = 1}$ and~${O = O_i}$ for some~${i \in [s(t+1)]}$. 
    Hence,~${\mathcal{O} = \{O_i \colon i \in [s(t+1)]\}}$. 
\end{proof}

\subsection{Open problems}\label{subsec:openproblems}

We now discuss some interesting directions for future research in this area.

\begin{problem}
    \label{prob:1}
    Characterise subsets~$A$ of an abelian group~$\Gamma$ admitting a function~${f \colon \mathbb{N} \to \mathbb{N}}$ such that for every positive integer~$k$, every $\Gamma$-labelled graph~${(G,\gamma)}$ contains~$k$ vertex-disjoint cycles whose $\gamma$-values are in~$A$ or a hitting set  of size at most~$f(k)$
    for the set of cycles of~$G$ whose $\gamma$-values are in~$A$. 
\end{problem}

As seen in Theorem~\ref{thm:infiniteEPctex}, for an infinite group, such a set~$A$ needs to be infinite as well. 
This problem is already interesting for the group~$\mathbb{Z}$. 

As a surprising negative result, for every positive integer~$t$, there is a $\mathbb{Z}$-labelled graph~${(G,\gamma)}$ with no two vertex-disjoint cycles with $\gamma$-value at least~$0$ and no hitting set of size at most~$t$ for these cycles.
Let~$G$ be the graph with vertex set~${\{ v_i \colon i \in [4t+4] \}}$, where each vertex with an even index~$2i$ is adjacent to all vertices with odd indices~$j$ for which ${j \leq 2i+1}$. 
Let~$\gamma$ be the $\mathbb{Z}$-labelling of~$G$ which assigns value~${t+3}$ to the edge~$v_{2i}v_{2i+1}$ for all~${i \in [2t+1]}$ and value~${-1-t}$ to all other edges. 
In every cycle of this graph, both edges incident with the vertex of highest index in the cycle and both edges incident with the vertex of lowest index in the cycle have value~${-1-t}$, 
so there are at least two more edges of value~${-1-t}$ than of value~${t+3}$. 
From this, it is easy to verify that 
any cycle of $\gamma$-value at least~$0$ has length at least~${2t+4}$
and the construction satisfies the desired properties. 
This construction can easily be adapted to apply to cycles of $\gamma$-value at least~$L$ for any integer~$L$. 
This is in contrast to the case of cycles of length at least~$L$, where Thomassen~\cite{Thomassen1988} showed that an Erd\H{o}s-Posa result holds. 
Thus we also present the following variant of Problem~\ref{prob:1}.

\begin{problem}
    Characterise sets~$A$ of positive integers admitting a function~${f \colon \mathbb{N} \to \mathbb{N}}$ such that for every positive integer~$k$, every graph~$G$ contains~$k$ vertex-disjoint cycles whose lengths are  in~$A$ 
    or a hitting set of size at most~$f(k)$ for 
    the set of cycles of~$G$ whose lengths are in~$A$.
\end{problem}

The construction presented above can be also adapted to show that a ${(1/s)}$-integral analogue of the Erd\H{o}s-P\'{o}sa theorem fails for cycles of non-negative values in $\mathbb{Z}$-labelled graphs for every positive integer~$s$. 
Interestingly, we know of no natural example where a half-integral analogue of the Erd\H{o}s-P\'{o}sa theorem fails but some fractional analogue of the Erd\H{o}s-P\'{o}sa theorem holds. 
In fact, we conjecture the following.
For a subset $A$ of an abelian group $\Gamma$
and a $\Gamma$-labelled graph~$(G,\gamma)$, 
let $\mathcal O^A_{G,\gamma}$ be the set of all cycles whose $\gamma$-values are in~$A$.

\begin{conjecture}
    Let~$\Gamma$ be an abelian group, let~${A \subseteq \Gamma}$, and let~${s \geq 4}$ be an integer.

    If there is a function~${f \colon \mathbb{N} \to \mathbb{N}}$ such that 
    for every $\Gamma$-labelled graph $(G,\gamma)$ 
    and every positive integer~$k$, 
    there exist either $k$ cycles in $\mathcal O^A_{G,\gamma}$ such that 
    no vertex is in~$s$ of them
    or a hitting set for  $\mathcal O^A_{G,\gamma}$ of size at most~${f(k)}$, 
    then there is a function~${f' \colon \mathbb{N} \to \mathbb{N}}$ such that for
    every $\Gamma$-labelled graph~${(G',\gamma')}$ 
    and every positive integer $k$, 
    there exist either~$k$ cycles in $\mathcal O^A_{G',\gamma'}$ such that no vertex is in three of them 
    or a hitting set for  $\mathcal O^A_{G',\gamma'}$ of size at most~${f'(k)}$. 
\end{conjecture}

If we strengthen this conjecture to allow restricting the class of $\Gamma$-labelled graphs considered, then there are examples for which this strengthening fails.
In other words, there exist
a subset $A$ of an abelian group $\Gamma$
and 
a class $\mathcal C$ of $\Gamma$-labelled graphs 
for which the following statement is false.
\begin{quote}
    \itshape 
    If there is a function $f:\mathbb N\to \mathbb N$ 
    such that for every $(G,\gamma)\in \mathcal C$
    and every positive integer $k$, 
    there exist either $k$ cycles in $\mathcal O^A_{G,\gamma}$
    such that no vertex is in $s$ of them or 
    a hitting set for $O^A_{G,\gamma}$ of size at most $f(k)$,
    then 
    there is a function~${f' \colon \mathbb{N} \to \mathbb{N}}$ such that for every ${(G',\gamma')}\in \mathcal C$ and every positive integer $k$, 
    there exist either~$k$ cycles in~$\mathcal O^A_{G',\gamma'}$ such that no vertex is in three of them 
    or a hitting set for  $\mathcal O^A_{G',\gamma'}$ of size at most~${f'(k)}$.
\end{quote}
Moreover, whenever $A$ is a nonempty finite subset of an infinite abelian group~$\Gamma$,  
the above statement is false by Theorem~\ref{thm:infiniteEPctex}.
However, we do not know of any counterexample to the above statement 
for which there is a class~${\mathcal G}$ of graphs 
such that 
${\mathcal C}$ is the class of all ${\Gamma}$-labelled graphs ${(G,\gamma)}$
with~${G\in \mathcal G}$.

\subsection*{Acknowledgements.}
The authors would like to thank the anonymous reviewers for their careful reading of the manuscript and their helpful comments.
\subsection*{Declarations}\ \\
\noindent
\textbf{Data availability.}
Data sharing not applicable to this article as no data sets were generated or analyzed during the current study.\\
\noindent
\textbf{Conflict of interest.}
On behalf of all authors, the corresponding author states that there is no conflict of interest.

\end{document}